\documentclass[11pt,letterpaper,reqno]{amsart}
\usepackage[
top=2cm, bottom=1.6cm, inner=2cm, outer=2cm,
marginparwidth=2cm 
]{geometry}

\usepackage{amssymb,latexsym, amsmath, amsxtra, mathrsfs, bm}

\usepackage[dvips]{graphics}
\usepackage[all]{xy}
\usepackage{xcolor}
\usepackage{subcaption}

\usepackage{verbatim}
\usepackage[abs]{overpic}
\usepackage{hyperref}
\usepackage{mathtools}
\usepackage{tikz}  
\usepackage{yhmath}

\usepackage[T1]{fontenc}

\usepackage{enumitem}

\usepackage{hyperref}

\setenumerate[1]{label=\upshape(\roman*)}

\DeclareMathAlphabet{\mathbbold}{U}{bbold}{m}{n}

\allowdisplaybreaks[3]

\theoremstyle{plain}
\newtheorem{theorem}{Theorem}[section]
\newtheorem*{theorem*}{Theorem}
\newtheorem*{conj*}{Conjecture}
\newtheorem{lemma}[theorem]{Lemma}
\newtheorem{prop}[theorem]{Proposition}

\newtheorem{thmx}{Theorem}
 
\newtheorem{thml}{Theorem}

\theoremstyle{definition}
\newtheorem{definition}[theorem]{Definition}
\newtheorem{rem}[theorem]{Remark}

\theoremstyle{remark}

\numberwithin{equation}{section}
\numberwithin{theorem}{section}
\numberwithin{table}{section}
\numberwithin{figure}{section}

\makeatletter
\def\th@plain{%
	\thm@notefont{}% same as heading font
	\itshape % body font
}
\def\th@definition{%
	\thm@notefont{}% same as heading font
	\normalfont % body font
}
\makeatother
\renewcommand{\le}{\leqslant}
\renewcommand{\leq}{\leqslant}
\renewcommand{\ge}{\geqslant}
\renewcommand{\geq}{\geqslant}

\newcommand{\var}  {\operatorname{var}}
\newcommand{\diam}  {\operatorname{diam}}

\newcommand{\card} {\operatorname{card}}
\newcommand{\supp}{\operatorname{supp}}

\newcommand{\R}{\mathbb{R}}

\newcommand{\N}{\mathbb{N}}      
\newcommand{\Z}{\mathbb{Z}}

\providecommand{\abs}[1]{\lvert#1\rvert}
\providecommand{\Absbig}[1]{\bigl\lvert#1\bigr\rvert}
\providecommand{\Absbigg}[1]{\biggl\lvert#1\biggr\rvert}
\providecommand{\AbsBig}[1]{\Bigl\lvert#1\Bigr\rvert}

\providecommand{\norm}[1]{\|#1\|}
\providecommand{\Normbig}[1]{\bigl\|#1\bigr\|}

\renewcommand{\:}{\colon}

\newcommand{\Lip}{\operatorname{Lip}}

\renewcommand{\=}{\coloneqq}

\newcommand{\MME}{\operatorname{MME}}

\newcommand{\cA}{\mathcal{A}}

\newcommand{\cC}{\mathcal{C}}
\newcommand{\cD}{\mathcal{D}}
\newcommand{\cE}{\mathcal{E}}

\newcommand{\cG}{\mathcal{G}}

\newcommand{\cK}{\mathcal{K}}
\newcommand{\cL}{\mathcal{L}}
\newcommand{\cM}{\mathcal{M}}

\newcommand{\cO}{\mathcal{O}}
\newcommand{\cP}{\mathcal{P}}

\newcommand{\cR}{\mathcal{R}}

\newcommand{\fE}{\mathfrak{E}}

\newcommand{\fL}{\mathfrak{L}}

\newcommand{\sD}{\mathscr{D}}

\newcommand{\sR}{\mathscr{R}}

\newcommand{\sZ}{\mathscr{Z}}

\newcommand{\oG}{\overline{G}}

\newcommand{\oc}{\overline{c}}

\newcommand{\ox}{\overline{x}}
\newcommand{\oy}{\overline{y}}

\newcommand{\tz}{\widetilde{z}}

\newcommand{\energy}{Q(T,\phi)}

\newcommand{\topo}{\operatorname{top}}
\newcommand{\im}{\operatorname{im}}

\newcommand{\cstar}[3][black]
{
    \tikz[baseline=(X.base)]{\node[circle, draw, inner sep=0pt, text=#1, minimum height=#2, font=#3] (X) {$*$};}
}
\newcommand{\ostar}{\!\!\scalebox{0.8}{\cstar{1pt}{\tiny}}\!\!}

\newcommand{\cdiv}{  
	\begin{tikzpicture}[baseline=(center.center)]  
		\node[circle, draw, minimum size=7.5pt,inner sep=0pt] (center){}; 
		\node at (center) {\scriptsize $\div$}; 
	\end{tikzpicture}  
}
\newcommand{\odiv}{{\hspace{-4pt}\raisebox{0.6ex}{\cdiv}\hspace{-4pt}}}

\newcommand{\oplusu}[1] {\underset{#1}{\oplus}}

\newcommand{\Rm}{\overline{\underline{\mkern-1mu\R\mkern-1mu}}}
\newcommand{\Rminus}{\underline{\mkern-1mu\R\mkern-1mu}}

\newcommand{\fless}{\leq}

\newcommand{\density}{\cD}

\newcommand{\es}[3] {\underline{\mkern-1mu\cE\mkern-1mu}_{#3}(#1,#2)}
\newcommand{\esbig}[3] {\underline{\mkern-1mu\cE\mkern-1mu}_{#3}\bigl(#1,#2\bigr)}

\newcommand{\tdim}{\underline{\mkern-1mu\dim\mkern-1mu} \,}

\newcommand{\twidehat}[1]{\widehat{\mkern+1mu\underline{\mkern-1mu #1 \mkern-1mu}\mkern+1mu}}

\newcommand{\toverline}[1]{\overline{\mkern+1mu\underline{\mkern-1mu #1 \mkern-1mu}\mkern+1mu}}

\newcommand{\assocmodu}{\widetilde{\omega}_{\lambda}}

\newcommand{\assocsum}[2]{\widetilde{\omega}_{#1,#2}}

\newcommand{\tai}{\underline{\mkern-1mu 0 \mkern-1mu} }

\newcommand{\tmi}{\underline{\mkern-1mu 1 \mkern-1mu} }

\newcommand{\taif}{\underline{\mkern-1mu \mathbbold{0} \mkern-1mu} }

\newcommand{\tmif}{\underline{\mkern-1mu \mathbbold{1} \mkern-1mu} }

\newcommand{\tinftyf}{\underline{\mkern-1mu \bm{\infty} \mkern-1mu} }

\begin{document}
\title[Typical periodic optimization and freezing phase transition: Dini potentials and beyond]{Typical periodic optimization and freezing phase transition:\\ Dini potentials and beyond}

\makeatletter
% Remove \uppercase and \author@andify so \author prints its argument as-is
\renewcommand{\@setauthors}{%
  \begingroup
  \trivlist
  \centering\normalsize \@topsep30\p@\relax
  \advance\@topsep by -\baselineskip
  \item\relax
  \authors
  \vspace{14pt}
  \endtrivlist
  \endgroup
}
\makeatother

\author[Yunping~Jiang, Zhiqiang~Li, and Yiqing~Sun]{%
  \begin{tabular}{c}Yunping~Jiang\\[-3pt]\Small City University of New York
  \end{tabular}\;
  \begin{tabular}{c}Zhiqiang~Li\\[-3pt]\Small Peking University  \end{tabular}\;
  \begin{tabular}{c}Yiqing~Sun\\[-3pt]\Small Peking University  \end{tabular}
}


    \begin{abstract}
        We investigate the Typical Periodic Optimization (TPO) Conjecture and freezing phase transitions for Dini potentials and beyond. We show that the TPO conjecture fails already at Dini regularity. We identify a sharp regularity threshold for freezing phase transitions, which we call weak log-log Dini. Both results reveal the stability of aperiodic ground states. Moreover, we identify Dini regularity as the sharp threshold for a quantitatitve analog of the freezing phase transition. 
    \end{abstract}

	\maketitle

    \setcounter{tocdepth}{1} 
	\tableofcontents

\section{Introduction}

In this article, we investigate the role of the regularity of potential functions for ergodic optimization properties. More specifically, we investigate the Typical Periodic Optimization (TPO) Conjecture and freezing phase transitions for Dini potentials and beyond. 

Dini potentials sit naturally between Hölder and merely continuous potentials. Thermodynamic formalism, including the Ruelle--Perron--Frobenius theorem and convergence results for transfer operators, extends to this intermediate regularity class beyond H\"{o}lder regime (cf.~\cite{FJ01a,FJ01b}). The Dini regularity has a rich history in analysis, originating in the classical Dini test for Fourier series (cf.~\cite{Zy59}) and also showing up in elliptic regularity and free-boundary problems (cf.~\cite{ALS13,MP11}). We explore some basic properties in ergodic optimization for Dini potentials.

\begin{theorem*}[Ergodic optimization properties for Dini potentials]
    Let $(\Sigma,\sigma)$ be the full shift on $\Sigma\=\{0,\,1\}^{\N}$ equipped with the metric $d_2$ and $\omega$ be a modulus of continuity satisfying the Dini condition and $\omega(\diam_{d_2}\Sigma)>0$. Then the following hold:
    \begin{enumerate}
        \smallskip
        \item For each $\omega$-continuous function $\phi\: \Sigma\to\R$, $\phi$ does not admit super-exponential suppression, i.e., for each nonempty open subset $U$ of $\Sigma$, the mass of the equilibrium state $\mu_{\beta\phi}$ on $U$ does not decay super-exponentially along any subsequence of $\beta$ converging to $+\infty$.  Consequently, every $\omega$-continuous function $\phi$ does not admit a freezing phase transition. 

        \smallskip
        \item For a generic $\omega$-continuous function $\phi\: \Sigma\to\R$, the equilibrium state $\mu_{\beta\phi}$ as $\beta\to+\infty$ weak$^*$ converges to the unique maximizing measure $\mu$ for $\phi$, obeying a large deviation principle. 
        \smallskip
        \item TPO may not hold for the space of $\omega$-continuous functions $H^{\omega}(\Sigma,\R)$. More explicitly, there may exist a nonempty open subset of $H^{\omega}(\Sigma,\R)$ in which every potential does not have a periodic maximizing measure. 
    \end{enumerate}
\end{theorem*}
Here statement~(i) follows from Theorem~\ref{t: Dini quanta no freezing}~(i), statement~(ii) follows from Theorem~\ref{t: Dini quanta no freezing}~(ii) and Remark~\ref{r: generic for symbolic}, and statement~(iii) follows from Theorem~\ref{t: non TPO}. 

Formulated as a generalization of the Yuan--Hunt Conjecture~\cite{YH99} in the spirit of Hunt and
Ott~\cite{HO96a,HO96b}, the TPO Conjecture asserts that for a suitably hyperbolic dynamical system, a Banach space consisting of suitably regular continuous potentials contains an open and dense subset in which every potential has a periodic maximizing measure, that is, the support of the maximizing measure is a periodic orbit \cite[Conjecture~7.3]{Je19}. In (iii), we show that there exists a modulus of continuity $\omega$ satisfying the Dini condition such that the TPO Conjecture fails for $H^\omega(\Sigma,\R)$ (see Theorem~\ref{t: non TPO} for more details).

Stability of aperiodic structures, as indicated by (iii), is a natural phenomenon in the physical world. Since the discovery of quasicrystals~\cite{SBGC84}, which are spatially ordered but aperiodic, a fundamental problem in statistical mechanics has been whether quasicrystalline ground states can arise from microscopic interactions and remain stable under both changes in temperature and the underlying interaction (see \cite{Mi24} for a concise overview). In 1990, van Enter and Mi\k{e}kisz \cite{EM90} provided a rigorous example of quasicrystalline ground states stable under small temperature variations. This low-temperature stability phenomenon also appears in dynamical systems in the particularly pronounced form of \emph{freezing phase transitions}, first investigated by Bruin and Leplaideur \cite{BL13}: below a finite temperature, the equilibrium state ceases to vary and coincides with the ground state. 

Here we say that a continuous potential $\phi\: X\to\R$ admits \emph{super-exponential suppression} if there exists some equilibrium state $\mu_{\beta\phi}$ for $\beta\phi$ for each $\beta\in(0,+\infty)$ and a nonempty open subset $U$ of $X$ such that $\liminf_{\beta\to+\infty}\frac{1}{\beta}\log\mu_{\beta\phi}(U)=-\infty$. The family $\{\mu_{\beta\phi}\}_{\beta\in (1,+\infty)}$ is said to satisfy the \emph{large deviation principle}, which provides a \emph{rate function} $I\: X\to [0,+\infty]$ such that the exponential decay rate of $\mu_{\beta\phi}$ is given by $-\inf_{x\in U}I(x)$ (see Definition~\ref{d: ldp} for a precise definition). To prove (i) and~(ii), we need the full strength of the theory of tropical thermodynamic formalism developed in Section~\ref{s: TTF for Dini}. A key ingredient in our proof of the failure of TPO in (iii) is the notion of accident introduced by Bruin and Leplaideur in their work on freezing phase transitions \cite{BL13,BL15}. 

We further identify sharp regularity thresholds for both super-exponential suppression and freezing phase transitions, with the former as an asymptotic quantitative relaxation of the latter: while the
Dini condition is the threshold for super-exponential suppression, the threshold for freezing phase transitions is finer, given by a
condition we call \emph{weak log-log Dini} (see (\ref{eq: key threshold})).

\begin{theorem*}[Freezing thresholds]
     Let $\omega\:[0,+\infty)\to[0,+\infty)$ be a modulus of continuity.
    \begin{enumerate}
        \smallskip
        \item The modulus of continuity $\omega$ satisfies the Dini condition if and only if for every mixing expanding covering map $T\: X\to X$ on a compact metric space $(X,d)$ and every $\omega$-continuous function $\phi\:X\to\R$, $\phi$ does not admit super-exponential suppression. 

        \smallskip
        \item The modulus of continuity $\omega$ satisfies the weak log-log Dini condition if and only if for every mixing expanding covering map $T\: X\to X$ on a compact metric space $(X,d)$ and every $\omega$-continuous function $\phi\: X\to\R$, $\phi$ does not admit a freezing phase transition.
    \end{enumerate}
\end{theorem*}

Here statement~(i) follows from Theorems~\ref{t: Dini quanta no freezing} and~\ref{t: super-exponential suppression for non-Dini}, and statement~(ii) follows from Theorems~\ref{t: non-freezing for log Dini} and~\ref{t: freezing for not log Dini}.

\subsection{Backgrounds}

A remarkable phenomenon, suggested by numerical experiments of Hunt and Ott in the 1990s \cite{HO96a,HO96b}, is that for a chaotic map $T$, a typical potential $\phi$ within a sufficiently regular class, such as the class of smooth or H\"{o}lder potentials, has a unique maximizing measure supported on some $T$-periodic orbit $\cO$, i.e., the $\phi$-average $\frac{1}{\card \cO}\sum_{x\in\cO}\phi(x)$ is strictly larger than the $\phi$-average $\int\!\phi\,\mathrm{d}\mu$ with respect to any $T$-invariant Borel probability measure $\mu$ not supported on $\cO$. Hunt and Ott conjectured in \cite{HO96a, HO96b} that this phenomenon is typical in a probabilistic sense, while Yuan and Hunt \cite{YH99} subsequently formulated a topological analog, that for a uniformly expanding map or an Axiom A diffeomorphism $T$, the maximizing measure is unique and periodic for an open and dense subset of Lipschitz potentials. These two conjectures are among the central problems in \emph{ergodic optimization}, the study of maximizing measures; for comprehensive surveys, see \cite{Boc18,Je06,Je19}. Progress towards the Hunt--Ott Conjecture has been made for H\"{o}lder potentials over full shifts by \cite{BZ16,DLZ24}, and for smooth potentials over circle expanding maps by \cite{GSZ25}. The Yuan--Hunt conjecture has stimulated a sustained period of work (see \cite{Bou01,Bou08,BQ07,CLT01,Mo08,QS12}) and was eventually proved by Contreras \cite{Co16} for expanding maps and by Huang, Lian, Ma, Xu, and Zhang \cite{HLMXZ25} for the remaining uniformly hyperbolic cases. Going beyond the setting of uniform hyperbolicity, Li and Zhang \cite{LZ25} proved an analogous result for expanding Thurston maps from complex dynamics; Hao, Huang, Jenkinson, and Li \cite{HHJL25} showed that beta-transformations $T_{\beta}$ has TPO in all H\"{o}lder spaces for Lebesgue almost every $\beta>1$ and for a residual set of $\beta$; Huang, Jenkinson, and Li \cite{HJL25} investigated the Gauss map and established TPO for essentially compact potentials; Huang, Jenkinson, Xu, and Zhang \cite{HJXZ26a} established TPO for a wide class of shift spaces.

These developments led naturally to the broader TPO Conjecture \cite[Conjecture~7.3]{Je19} stated earlier, which generalizes the Yuan–Hunt setting to suitably hyperbolic systems and Banach spaces of suitably regular continuous potentials. While the positive results mentioned above were obtained for H\"{o}lder or Lipschitz potentials, the TPO Conjecture fails in the $C^0$ category: for systems with Bowen's specification property, a generic continuous potential has a unique fully supported maximizing measure, which is therefore not periodic (see \cite{BJ02, Bre08, JM08,Mo10}). The above literature has demonstrated the following natural theme:
\begin{quote}
Changing the regularity of the potentials can fundamentally alter the behavior of maximizing measures.
\end{quote}
This theme also appears in the study of \emph{exposure problem} (cf.~\cite[Problem~3.9]{Je06}) in ergodic optimization. In \cite{HJLZ26}, it is shown that no potential that has Haar measure as the unique maximizing measure belongs to the Wiener algebra. 

The intuition that periodic structures are typical has long guided condensed matter physics: crystals were generally expected to be periodic \cite{Cah01,Da14}. However, the discovery of quasicrystals \cite{SBGC84} showed that crystals may be not periodic, prompting the study of the microscopic mechanism of occurrence of quasicrystalline ground states. The stability of such aperiodic ground states, both under temperature variations \cite{EM90} and under perturbations of the underlying interaction \cite{Mi97,EM20}, was subsequently investigated by van Enter and Mi\k{e}kisz. 

Phase transitions constitute a central theme in statistical mechanics, with deep connections to probability theory and dynamical systems through phenomena such as universality and critical behavior (see e.g.~\cite{Gol92,Ru78,PY98,Sin72}), referring to abrupt changes of the physical systems with respect to the parameters (e.g., the temperature). These changes correspond to nonanalyticities of the energy and qualitative shifts in the structure of the Gibbs states. In physics, the energy (Hamiltonian) is usually expressed in terms of an external field and an interaction, and the Gibbs state is a probability measure over the configuration space that minimizes the free energy. In dynamical systems, the energy is the \emph{pressure function} and the Gibbs state (known as the \emph{equilibrium state}) is an invariant probability which maximizes the \emph{measure-theoretic pressure}.

Around the early 1970s, inspired by statistical mechanics, Sinai, Bowen, and Ruelle introduced thermodynamic formalism into ergodic theory (see~\cite{Sin72, Bow75, Ru78}). To be more precise, let $T\: X\to X$ be a continuous map on a compact metric space $(X,d)$, and $\phi\: X\to\R$ be a continuous function. The measure-theoretic pressure is defined by
\begin{equation*}
	P_\mu(T,\phi)\=h_{\mu}(T)+\int\!\phi\,\mathrm{d}\mu,
\end{equation*}
where $h_\mu(T)$ is the \emph{measure-theoretic entropy} and $\mu$ is a $T$-invariant Borel probability measure. The measure $\mu$ is called an equilibrium state if $\mu$ maximizes $P_\mu(T,\phi)$ and the maximum is the \emph{topological pressure}, which we denote by $P(T,\phi)$. The function $\beta\to P(T,\beta\phi)$ is called the pressure function, where $\beta$ is a real parameter and is known as the \emph{inverse temperature}.

Inspired by the phase transitions associated with the Hofbauer potential \cite{Hofb77} and Pomeau--Manneville-type systems \cite{BLL12}, Bruin and Leplaideur \cite{BL13} initiated the study of the freezing phase transition, a particularly pronounced form of the low-temperature stability. More precisely, a continuous potential $\phi\:X\to\R$ over a topological dynamical system $T\: X\to X$ is said to \emph{admit a freezing phase
transition} if there exists a temperature $t_c\in (0,+\infty)$ such that the sets $\fE(T,\phi/t_1)$ and $\fE(T,\phi/t_2)$ of equilibrium states coincide for all $t_1,\,t_2\in (0,t_c)$, while $\fE(T,\phi/t)\neq\fE(T,\phi/t')$ for all $0<t<t_c<t'$. We say that a $T$-invariant Borel probability measure $\mu$ is a \emph{frozen state} for $\phi$ if there exists a temperature $t_c\in (0,+\infty)$ such that $\mu\in\fE(T,\phi/t)$ for all $t\in(0,t_c)$. While the earlier works \cite{BLL12,Hofb77} can be viewed retrospectively as prototypes of freezing at fixed points, a central motivation behind the work of Bruin and Leplaideur \cite{BL13,BL15} was to investigate whether freezing can occur at aperiodic invariant measures that model quasicrystals. They proposed classes of potentials freezing onto the Thue--Morse subshift and the Fibonacci subshift, respectively \cite{BL13,BL15}, with a recent treatment of the Thue--Morse subshift case given by B\'edaride, Cassaigne, Huber, and Leplaideur in \cite{BCHL25}. More recently, Buzzi, Kloeckner, and Leplaideur \cite{BKL23}, Chazottes, Kucherenko, and Quas \cite{CKQ25}, and Hedges \cite{He24} have extended the scope considerably, realizing freezing at a large variety of prescribed invariant measures.

\subsection{Stability of aperiodic ground states}

We begin with the following result, which is the main consequence of tropical thermodynamic formalism for Dini potentials. It serves as the common starting point from which we develop deeper results of this article: the failure of TPO and the freezing thresholds.

\begin{thmx}[Logarithmic-type zero-temperature limits for Dini potentials]\label{t: Dini quanta no freezing}
    Let $T\: X\to X$ be a mixing expanding covering map on a compact metric space $(X,d)$, $\omega$ be a modulus of continuity satisfying the Dini condition, and $\phi\: X\to\R$ be an $\omega$-continuous function.
    Then the following hold:
    \begin{enumerate}
        \smallskip
        \item  For each nonempty open subset $U$ of $X$, there exists an invariant tropical Gibbs density $b\:X\to [-\infty,0]$ for $\phi$ such that 
     \begin{equation}\label{ineq: quantative non-freezing for dini}
        \liminf_{\beta\to+\infty}\beta^{-1}\log(\mu_{\beta\phi}(U))\geq \sup_{x\in U}b(x)>-\infty,
     \end{equation}
     where $\mu_{\beta\phi}$ is the unique equilibrium state for the potential $\beta\phi$. In particular, $\phi$ does not admit super-exponential suppression. Consequently, $\phi$ does not admit a freezing phase transition. 

        \smallskip
        \item If $\phi$ has a unique maximizing measure, then the family of equilibrium states $\{\mu_{\beta\phi}\}_{\beta\in (1,+\infty)}$ satisfies the large deviation principle as $\beta\to+\infty$ with rate function $I\: X\to[0,+\infty]$, where $-I$ is an invariant tropical Gibbs density. Consequently, this property holds for a generic subset of $H^{\omega}(X,\R)$ if $d_{\omega}(\,\cdot\,,\,\cdot\,)\=\omega(d(\,\cdot\,,\,\cdot\,))$ is a metric on $X$. 
    \end{enumerate}
\end{thmx}

Here the notion of \emph{tropical Gibbs density} is the tropical counterpart of the notion of Gibbs states, see Definition~\ref{d: tropical gibbs density} for a precise definition. The notion of \emph{large deviation principle} is recalled in Definition~\ref{d: ldp}. A \emph{generic} subset means a countable intersection of open and dense subsets. 

\begin{rem}
    Note that for a modulus of continuity $\omega$ that vanishes only at $0$, it is straightforward to check that $d_{\omega}$ is a metric on $X$ if either $d$ is an ultrametric (see e.g.~\cite[Theorem~9]{PT14}) or $\omega$ is subadditive in the sense that $\omega(s+t)\leq \omega(s)+\omega(t)$ for all $s,\,t\in [0,\diam_dX]$. 
\end{rem}

While the (qualitative) non-freezing conclusion in Theorem~\ref{t: Dini quanta no freezing} forms a counterpart to \cite[Theorem~B]{CKQ25} in the context of the expanding dynamical systems investigated in \cite{J99a,FJ01a}, the estimate (\ref{ineq: quantative non-freezing for dini}), which quantifies the non‑freezing phenomenon, is obtained by developing the theory of \emph{tropical thermodynamic formalism} (see \cite{LS26} for H\"{o}lder potential, cf.~\cite{BLL13,BLT06,MO24,MO25}). The conclusion of (ii) is finer than that of (i) in that the density $-I$ is independent of the choice of $U$. For (ii), a notable difference from the H\"{o}lder setting is that the TPO Conjecture holds for H\"{o}lder potentials so that for an open and dense subset of H\"{o}lder potentials $\phi$, $\phi$ has a unique maximizing measure, while in general, uniquely maximizing potentials are only known to be generic (cf.~\cite[Theorem~3.2]{Je06}). 

Now that both thermodynamic formalism and tropical thermodynamic formalism for Dini potentials are developed, with the latter serving as a bridge between thermodynamic formalism and ergodic optimization, it is natural to investigate ergodic optimization and in particular, the TPO Conjecture for Dini potentials. As a contrast, we show in Theorem~\ref{t: non TPO} below that it already fails in this regime. Thus, it remains open whether there exists an open and dense subset of $H^\omega(X,\R)$ consisting of uniquely maximizing potentials, where $\omega$ is a modulus of continuity satisfying the Dini condition. Unlike the recently constructed counterexamples to the TPO Conjecture arising from lack of openness of the system \cite{HJXZ26a,HJXZ26b} and from lack of weak$^*$ compactness of the space of invariant measures \cite{HJL25} (cf.~\cite{BoLe26} for the failure of an analog of the TPO conjecture), the obstruction in our example comes entirely from the regularity of the potentials. To the best of our knowledge, this is the first failure of the TPO Conjecture for a uniformly expanding system under a quantitative regularity assumption strictly stronger than mere continuity. 

Let $T\:X\to X$ be a continuous map on the compact metric space $(X,d)$ and $\cM(X,T)$ be the space of $T$-invariant Borel probability measures. A measure $\mu\in\cM(X,T)$ is called \emph{periodic} if it is supported on some $T$-periodic orbit $\cO$. For a continuous potential $\phi\: X\to\R$, let
\begin{equation}\label{eq: defi of ergodic supremum}
\energy\=\sup\biggl\{\int\!\phi\,\mathrm{d}\mu:\mu\in\cM(X,T)\biggr\}
\end{equation}
be the \emph{ergodic supremum} of $\phi$. A measure $\mu\in\cM(X,T)$ is called a $\phi$-\emph{maximizing measure} if $\int\!\phi\,\mathrm{d}\mu=\energy$. 

A nondecreasing function $\omega\:[0,+\infty)\to [0,+\infty)$ with $\lim_{t\to 0^+}\omega(t)=\omega(0)=0$ is called a \emph{modulus of continuity}.  A modulus of continuity $\omega$ is said to satisfy the \emph{Dini condition} if \begin{equation}\label{ineq: defi of Dini}
\int_{0}^1\!\frac{\omega(t)}{t}\,\mathrm{d}t<+\infty.
\end{equation}
An $\omega$-continuous function is a continuous function whose modulus of continuity is controlled by $\omega$, the space consisting of $\omega$-continuous functions is denoted by $H^\omega(X,\R)$, and the related norm is denoted by $\norm{\,\cdot\,}_{\omega}$ (see Subsection~\ref{ss: expanding maps and dini potentials}). The Fibonacci subshift $K$ and its related properties are reviewed in Subsection~\ref{ss: Fibonacci subshift and Rauzy graph}. The definition of the metric $d_2$ is given in Subsection~\ref{ss: notation}. We prove:

\begin{thmx}[Failure of TPO: robust maximizing on the Fibonacci subshift]\label{t: non TPO}
    Let $(\Sigma,\sigma)$ be the full shift on $\Sigma\=\{0,\,1\}^\N$ equipped with the metric $d_2$. There exists a modulus of continuity $\omega$ satisfying the Dini condition and an nonempty open subset $U$ of the space $H^\omega(\Sigma,\R)$ of $\omega$-continuous functions such that if $\phi\in U$, then the unique $\phi$-maximizing measure is equal to the unique $\sigma$-invariant Borel probability measure $\mu_K$ supported on the Fibonacci subshift $K$, which is, in particular, not a periodic measure.
\end{thmx}

Here we show that the TPO Conjecture fails by establishing robust maximizing on the Fibonacci subshift (resp.~at the invariant measure $\mu_K$), which is a particular instance of Sturmian subshifts (resp.~Sturmian measures). Sturmian measures arise naturally as maximizing measures in a variety of settings (cf.~\cite{Bou00, Je06, Je07, AJ12}). The phenomenon of unique Sturmian maximization on a set of parameters of positive Lebesgue measure is recently observed by Bochi and Lessa \cite{BoLe26} in the context of skew-product systems, where the hyperbolicity of the system is weakened, instead of the weakened regularity of potentials in the present work.  

While the preceding discussion focuses on ergodic optimization properties for Dini potentials, we now further investigate the phenomena of \emph{super-exponential suppression} and \emph{freezing phase transitions}, which are not admitted by Dini potentials as indicated by Theorem~\ref{t: Dini quanta no freezing}. Indeed, we determine \emph{sharp regularity thresholds} for both phenomena in Theorems~\ref{t: non-freezing for log Dini} to ~\ref{t: freezing for not log Dini} below. More precisely, we say that a modulus of continuity $\omega$ satisfies the \emph{weak log-log Dini condition} if
    \begin{equation}\label{eq: key threshold}\liminf_{r\to0^+}\frac{\int_{r}^1\!\omega(t)/t\,\mathrm{d}t}{\log (\log ( r^{-1}))}=0.
    \end{equation}
    We prove:

\begin{thmx}[Non-freezing for weak log-log Dini potentials]\label{t: non-freezing for log Dini}
    Let $\omega$ be a modulus of continuity satisfying the weak log-log Dini condition and $T\: X\to X$ be a mixing expanding covering map on some compact metric space $(X,d)$. Then every $\omega$-continuous potential $\phi\: X\to\R$ does not admit a freezing phase transition.
\end{thmx}

To construct a potential admitting super-exponential suppression (resp.~a freezing phase transition) for a modulus of continuity $\omega$ that is not Dini (resp.~weak log-log Dini), we turn to a most widely studied and natural class of potentials, potentials of the form $-\omega(d(\,\cdot\,,K))$ (cf.~\cite{BL13,BL15,CKQ25}), where $K$ is a closed invariant set and $\omega$ is a modulus of continuity. We refer to these as $(\omega,K)$-\emph{plateau potentials}, or simply, \emph{plateau potentials}. While Theorem~\ref{t: super-exponential suppression for non-Dini} below shows that $(\omega,\cO)$-plateau potentials admit super-exponential suppression for moduli of continuity $\omega$ not satisfying the Dini condition and periodic orbits $\cO$, Theorem~\ref{t: freezing for not log Dini} below provides an aperiodic subshift $K$ independent of the modulus of continuity $\omega$ such that the plateau potential admits a freezing phase transition whenever $\omega$ is not weak log-log Dini. We prove:

\begin{thmx}[Super-exponential suppression for the potential that is not Dini]\label{t: super-exponential suppression for non-Dini}
     Let $(X,\sigma)$ be a mixing one-sided subshift of finite type of positive entropy on $(X,d_2)$, $\omega$ be a modulus of continuity not satisfying the Dini condition, and $\cO$ be a periodic orbit. Then the $(\omega,\cO)$-plateau potential $-\omega(d_2(\,\cdot\,,\cO))$ is $\omega$-continuous and admits super-exponential suppression. 
\end{thmx}

\begin{thmx}[Freezing for the potential that is not weak log-log Dini]\label{t: freezing for not log Dini}
Let $\bigl(\{0,\,1,\,2\}^{\N},\sigma\bigr)$ be the full shift on $\bigl(\{0,\,1,\,2\}^{\N},d_2\bigr)$. Then there exists a $\sigma$-invariant closed subset $K$ of $\{0,\,1,\,2\}^{\N}$ containing no periodic points such that for every modulus of continuity $\omega$ not satisfying the weak log-log Dini condition, the $(\omega,K)$-plateau potential $-\omega(d_2(\,\cdot\,,K))$ is $\omega$-continuous and admits a freezing phase transition with a unique frozen state. 
\end{thmx}

In the context of topological dynamics, Buzzi, Kloeckner, and Leplaideur \cite{BKL23} obtained a realization theorem showing that every ergodic measure of zero entropy is the unique frozen state for some continuous potential. A broader abstract characterization including the case where the frozen state is not unique was obtained by Hedges \cite{He24}. These results are highly general but are essentially not constructive, following from functional analysis results for the space of invariant measures. By contrast, Chazottes, Kucherenko, and Quas \cite{CKQ25} consider plateau potentials $-\omega(d(\cdot,X))$ for a proper subshift $X$ of $A^{\Z^d}$ and provide a sufficient condition (different from our weak log-log Dini condition) on $\omega$ under which $-\omega(d(\cdot,X))$ admits a freezing phase transition. Our construction in Theorem~\ref{t: freezing for not log Dini} identifies the weak log-log Dini condition as the necessary and sufficient condition on $\omega$ (cf.~Theorem~\ref{t: non-freezing for log Dini}). 

It is possible to establish some of the results of this article for more general dynamical systems, 
but such investigations are beyond the scope of the present work and are presented elsewhere.

\subsection{Overview of proofs}

\begin{enumerate}
    \smallskip
    \item [\textbf{Part a:}]\emph{Theorem~\ref{t: Dini quanta no freezing} and tropical thermodynamic formalism}. 
    
    \smallskip
    To prove Theorem~\ref{t: Dini quanta no freezing}, we develop the theory of tropical thermodynamic formalism for Dini potentials. While tropical eigenfunctions of the Bousch operator for a H\"older potential remain H\"{o}lder continuous, tropical eigenfunctions of the Bousch opreator for an $\omega$-continuous potential are only $\assocmodu$-continuous, and thus generally less regular than the potential. 
    \begin{enumerate}
    \smallskip
    \item[\textbf{Step a.1:}]\emph{(Invariant tropical Gibbs density)}. As a main component of the theory, we prove:
    \end{enumerate}
\end{enumerate}
\smallskip

    \begin{thmx}[Existence of tropical eigenfunctions and eigendensities]\label{t: existence}
    Let $T\:X\rightarrow X$ be a transitive expanding map on a compact metric space $(X,d)$ with expanding constants $(\lambda,\gamma,\xi)$, $\omega$ be a modulus of continuity satisfying the Dini condition, and $\phi\in H^\omega(X,\R)$. Then the following hold:
    \begin{enumerate}[label=\rm{(\roman*)}]	
    \smallskip
    \item {\rm(Uniqueness of tropical eigenvalue)}. The ergodic supremum $\energy$ is the unique tropical eigenvalue of $\cL_\phi$ (resp.~$\cL_\phi^{\ostar}$) on the space of continuous functions $C(X,\R)$ (resp.~the density space $\density(X)$).
    \smallskip
    \item {\rm(Existence of tropical eigenfunction)}. For each $u\in C(X,\R)$, define
    \begin{equation*}
    v_u(x)\coloneqq\limsup_{n\to+\infty}\cL_{\toverline{\phi}}^n(u)(x)
    \end{equation*}
    for each $x$ in $X$, where $\toverline{\phi}\=\phi-\energy$. Then $v_u$ is in $H^{\assocmodu}(X,\R)$ and is a tropical eigenfunction of $\cL_\phi$ associated with eigenvalue $\energy$.
    \smallskip
    \item {\rm(Existence of tropical eigendensity)}. There exists a tropical eigendensity of $\cL_{\phi}^{\ostar}$ associated with eigenvalue $\energy$.
    \smallskip
    \item {\rm(Existence of invariant tropical Gibbs density)}. For each tropical eigenfunction $v$ of $\cL_\phi$ in $C(X,\R)$ and each tropical eigendensity $b$ of $\cL_\phi^{\ostar}$, $v\otimes b$ is an invariant tropical Gibbs density for $\phi$.
    \end{enumerate}
    \end{thmx}

\begin{enumerate}
 \smallskip\item[]
    Here the number $\energy$, the tropical operations $(\oplus,\otimes)$, the Bousch operator $\cL_\phi$, the adjoint Bousch operator $\cL_\phi^{\ostar}$, the space $\density(X)$,  and the notions of tropical eigenfunctions and tropical eigendensities are defined in Subsection~\ref{ss: backgrounds for TTF}. Expanding constants $(\lambda,\gamma,\xi)$ and the associated modulus $\assocmodu$ are defined in Subsection~\ref{ss: expanding maps and dini potentials}, respectively. Theorem~\ref{t: existence}~(iv) supplies the invariant tropical Gibbs density that appears in the quantitative non-freezing estimate (\ref{ineq: quantative non-freezing for dini}) in Theorem~\ref{t: Dini quanta no freezing}.
    
    \begin{enumerate}
    \smallskip
    \item[\textbf{Step a.2:}]\emph{(Logarithmic-type zero-temperature limits)}. The logarithmic-type zero-temperature limits serve as a bridge connecting thermodynamic formalism objects to their tropical counterparts. Recall that a family of real-valued continuous functions on $X$ (resp.\ functionals on $C(X,\R)$) is a \emph{normal family} if, for every sequence of functions (resp.\ functionals) of this family, there exists a subsequence that is uniformly converging on every compact subset of $X$ (resp.\ $C(X,\R)$). We prove:
    \end{enumerate}
\end{enumerate}

    \smallskip
    \begin{thmx}[Logarithmic-type zero-temperature limits]\label{t: zero-temperature limit}
	Let $T\:X\rightarrow X$ be a mixing expanding covering map on a compact metric space $(X,d)$, $\omega$ be a modulus of continuity satisfying the Dini condition, and $\phi\in H^\omega(X,\R)$. Then the following hold:
	\begin{enumerate}[label=\rm{(\roman*)}]	
        \smallskip
        \item The family $\bigl\{\frac{1}{\beta}\log u_{\beta\phi}\bigr\}_{\beta\in (1,+\infty)}$ is normal, and the (pointwise) limit of every convergent subsequence $\bigl\{\frac1{\beta_n}\log u_{\beta_n\phi}\bigr\}_{n\in\N}$ with $\beta_n\to+\infty$ as $n\to+\infty$ is some function in $\es{\cL_\phi}{C(X,\R)}{\energy}$.
		\smallskip
		\item The family $\bigl\{l_{\beta}^m\big|_{C(X,\R)}\bigr\}_{\beta\in(1,+\infty)}$ is normal, and the (pointwise) limit of every convergent subsequence $\bigl\{l_{\beta_n}^m\bigr\}_{n\in\N}$ with $\beta_n\to+\infty$ as $n\to+\infty$ is a tropically continuous linear functional whose density in $\density(X)$ is in $\esbig{\cL_\phi^{\ostar}}{\density(X)}{\energy}\smallsetminus\{\tinftyf,\,\taif\}$.
  
		\smallskip
        \item The family $\bigl\{l_{\beta}^{\mu}\big|_{C(X,\R)}\bigr\}_{\beta\in(1,+\infty)}$ is normal, and the (pointwise) limit of every convergent subsequence $\bigl\{l_{\beta_n}^{\mu}\bigr\}_{n\in\N}$ with $\beta_n\to +\infty$ as $n\to+\infty$ is a tropically continuous linear functional whose density in $\density(X)$ is the tropical product $v \otimes b$ of some $v \in \es{\cL_\phi}{C(X,\R)}{\energy}$ and some $b \in \esbig{\cL_\phi^{\ostar}}{\density(X)}{\energy}\smallsetminus\{\tinftyf,\,\taif\}$.
	   \end{enumerate}
    \end{thmx}

\begin{enumerate}
    \smallskip\item[] 
    Here the eigenfunctions $\{u_{\beta\phi}\}_{\beta\in(0,+\infty)}$, the families $\bigl\{l_{\beta}^m\bigr\}_{\beta\in(0,+\infty)}$ and $\bigl\{l_{\beta}^\mu\bigr\}_{\beta\in(0,+\infty)}$, and the two eigenspaces $\es{\cL_\phi}{C(X,\R)}{\energy}$ and $\esbig{\cL_\phi^{\ostar}}{\density(X)}{\energy}$ are defined in Subsection~\ref{ss: backgrounds for TTF}. The functions $\tinftyf$ and $\taif$ are the constant $+\infty$ and $-\infty$ functions on $X$, respectively (cf.~Subsection~\ref{ss: backgrounds for TTF}). Regarding the notion of the density in $\density(X)$ of a tropically continuous linear functional, see Proposition~\ref{p: dual space equals completion} and Remark~\ref{r: equivalence of densities}. 

    \begin{enumerate}
    \smallskip
    \item[\textbf{Step a.3:}]\emph{(No super-exponential suppression)}. Let $U$ be a nonempty open subset of $X$. We can take a sequence $\{\beta_n\}_{n\in\N}$ with $\beta_n\to+\infty$ as $n\to+\infty$ such that $\liminf_{\beta\to+\infty}\frac{1}{\beta}\log (\mu_{\beta\phi}(U))$ is equal to $\lim_{n\to+\infty}\beta_n^{-1}\log(\mu_{\beta_n\phi}(U))$. Passing to a subsequence of $\{\beta_n\}_{n\in\N}$ via Theorem~\ref{t: zero-temperature limit}~(iii) and applying Theorem~\ref{t: existence}~(iv) yields an invariant tropical Gibbs density $b_0$ that fulfills the first inequality in (\ref{ineq: quantative non-freezing for dini}). The invariance of $b_0$, together with the transitivity of $T\:X\to X$, yields the second inequality in (\ref{ineq: quantative non-freezing for dini}). Since the maximizing measure cannot be fully supported when the potential $\phi$ is not cohomologous to constants in $C(X,\R)$ (cf.~\cite[Proposition~2 and Theorem~1]{Bou01}), (\ref{ineq: quantative non-freezing for dini}) excludes a freezing phase transition.

    \smallskip
    \item[\textbf{Step a.4:}]\emph{(Large deviation principle)}. With the concepts of the Aubry set and the Ma\'{n}\~{e} potential originating from weak KAM theory (cf.~Definitions~\ref{d: Aubry set} and~\ref{d: Mane potential}), we establish representations of tropical eigenfunctions and eigendensities in Proposition~\ref{p: representation}. When the potential $\phi$ has a unique maximizing measure, the Ma\'{n}\~{e} potential $S_{\phi}$ satisfies $S_{\phi}(x,y)\otimes S_{\phi}(y,x)=0$ for all $x,\,y$ in the Aubry set $\Omega_\phi$, leading to the uniqueness of both eigenobjects.
    
    \end{enumerate}
    \smallskip
    \item [\textbf{Part b:}]\emph{Failure of TPO (Theorem~\ref{t: non TPO}) and the energy transport estimate}.

    \smallskip
    To prove Theorem~\ref{t: non TPO}, we want to find a small neighborhood of $-\omega(d_2(\,\cdot\,,K))$ in $H^\omega(\Sigma,\R)$ where every potential has a unique maximizing measure $\mu_K$. In particular, we consider the family of moduli of continuity $\{\omega_\alpha\}_{\alpha\in(0,+\infty)}$ given by 
    \begin{equation}\label{eq: defi of the omegaalpha family}
        \omega_\alpha(t)\=\begin{cases}
        (1-\log_2t)^{-\alpha}&\text{if }t\in[0,1],\\
        1&\text{if }t\in(1,+\infty).
        \end{cases}
    \end{equation}
    for all $\alpha\in(0,+\infty)$. Note that $\omega_\alpha$ satisfies the Dini condition when $\alpha\in (1,2)$. It suffices to prove:

\end{enumerate}
    \smallskip
    \begin{thmx}[Energy Transport estimate]\label{t: energy transport estimate}
    Let $(\Sigma,\sigma)$ be the full shift on $\Sigma=\{0,\,1\}^{\N}$ equipped with the metric $d_2$, and $\mu_K$ be the unique $\sigma$-invariant Borel probability measure supported on the Fibonacci subshift $K$. For each $\alpha\in(0,2)$, there exists a constant $C>0$ such that for each $\sigma$-invariant Borel probability measure $m$ and each $\omega_\alpha$-continuous function $g\: \Sigma\to\R$ with $\norm{g}_{\omega_\alpha}\leq 1$,
    \begin{equation}\label{ineq: transport estimate in theorem}
        \int_{\Sigma}\! g\,\mathrm{d}\mu_K-\int_{\Sigma}\! g\,\mathrm{d}m\leq C\int_{\Sigma}\!\omega_\alpha(d_2(x,K))\,\mathrm{d}m(x).
    \end{equation} 
\end{thmx}

\begin{enumerate}
    \item[]
    
    As a byproduct, Theorem~\ref{t: energy transport estimate} also yields a
result of independent interest concerning \emph{robust freezing}, which provides a first instance of \emph{joint} stability of quasicrystals in dynamical systems---the aperiodic ground state persists not only under temperature variations, but also under small perturbations of the potential. Note that if $\mu$ is a frozen state for $f$ and is a $g$-maximizing measure, then $\mu$ is also a frozen state for $f+g$. Therefore, combined with the fact that for each $\alpha\in (0,1)$, $\mu_K$ is the unique frozen state for $-\omega_\alpha(d_2(\,\cdot\,,K))$ (cf.~\cite[Theorem~5.3]{CKQ25}), Theorem~\ref{t: energy transport estimate} immediately implies Theorem~\ref{t: robust freezing for omega_alpha} below.

Let $V$ be a Banach space consisting of certain real-valued continuous functions on $X$. For a topological dynamical system $T\colon X\to X$, we say that freezing at some $T$-invariant Borel probability measure $\mu$ is \emph{robust} in $V$ if there exist a nonempty open subset $U$ of $V$ and a critical temperature $t_c>0$ such that for each $\phi\in U$ and each $t\in(0,t_c)$, $\mu$ is the unique equilibrium state for $\phi/t$. 
\end{enumerate}

\begin{thmx}[Robust freezing for $\omega_\alpha$]\label{t: robust freezing for omega_alpha}
    Let $(\Sigma,\sigma)$ be the full shift on $\Sigma=\{0,\,1\}^\N$ equipped with the metric $d_2$ and $\mu_K$ be the unique $\sigma$-invariant Borel probability measure supported on the Fibonacci subshift $K$. For each $\alpha\in (0,1)$, freezing at $\mu_K$ is robust in $H^{\omega_\alpha}(\Sigma,\R)$.
\end{thmx}

\begin{enumerate}
\smallskip\item[] 
    \begin{enumerate}
        \item[] We now turn to the proof of Theorem~\ref{t: energy transport estimate}. 
        \smallskip
        \item[\textbf{Step b.1:}]\emph{(Wasserstein distance)}. We use tools from \emph{optimal transport} to transform the problem. Applying the Kantorovich duality formula (cf.~\cite[Remark~6.5]{Vi09}), we see that to estimate the supremum $\sup\bigl\{\int_\Sigma\! g\,\mathrm{d}\mu_K-\int_{\Sigma}\!g\,\mathrm{d}m:g\in H^{\omega_\alpha}(\Sigma,\R),\,\norm{g}_{\omega_\alpha}\leq 1\bigr\}$, it suffices to estimate the $\omega_\alpha$-Wasserstein distance $W_{\omega_\alpha}(m,\mu_K)$ between $m$ and $\mu_K$ (cf.~Definition~\ref{d: Wasserstein}). With the multiscale formula (cf.~\cite[Lemma~3.1]{Kl15}), the Wasserstein distance $W_{\omega_\alpha}(m,\mu_K)$ can be estimated through the cylinder discrepancies $D_n(m,\mu_K),\,n\in\N$ (cf.~Definition~\ref{d: D_n}).
        
        \smallskip
        \item[\textbf{Step b.2:}]\emph{(Periodic shadowing and accidents)}. Since periodic measures are weak$^*$ dense in $\cM(\Sigma,\sigma)$ (cf.~\cite{Sig74}), it suffices to prove (\ref{ineq: transport estimate in theorem}) for periodic measures. To compare a periodic measure $m_{\cO}$ with $\mu_K$, we first construct a piecewise shadowing orbit of the periodic orbit $\cO$ on the Fibonacci subshift $K$ (cf.~Definition~\ref{d: defi of eta_z}). This construction relies on the useful concept of \emph{accidents} (cf.~Definition~\ref{d: shadow length and accident}) introduced by Bruin and Leplaideur in \cite{BL13,BL15}, where they investigated freezing phase transitions. Let $\eta_z$ denote the empirical distribution of the shadowing. As $\int_\Sigma\! g\,\mathrm{d}m_{\cO}-\int\! g\,\mathrm{d}\eta_z$ is naturally controlled by the right-hand side of (\ref{ineq: transport estimate in theorem}), we need to estimate $W_{\omega_\alpha}(\eta_z,\mu_K)$. Since accidents only take place at \emph{bi-special} scales $q_m$ (cf.~Proposition~\ref{p: bi-special scales}), it is convenient to estimate the discrepancies $D_{q_m+1}(\eta_z,\mu_K)$, $m\geq 4$ (cf.~Proposition~\ref{p: reduction to bi-special scales}).

        \smallskip
        \item[\textbf{Step b.3:}]\emph{(Flows on Rauzy graphs)}. The use of flows on \emph{Rauzy graphs} (cf.~\cite{Ra83,Fe99}; see Definition~\ref{d: Rauzy graph} for a precise definition) to estimate cylinder discrepancies is the key new ingredient. At each bi-special scale $q_m$, the cylinder discrepancy $D_{q_m+1}(\eta_z,\mu_K)$ is related to a flow $F_m$ on the Rauzy graph $\Gamma_m$ (cf.~(\ref{eq: defi of good edge chain})). Each junction between successive shadowing pieces (cf.~Proposition~\ref{p: geometry of bad window}) corresponds to a pair of source and sink for the flow. We add suitable connecting flows between these pairs to get a circulation $H_m$ (cf.~Proposition~\ref{p: boundary repair}). After this repair, we estimate $D_{q_m+1}(\eta_z,\mu_K)$ with the stable coordinate $s_m(H_m)$ (cf.~(\ref{eq: defi of two coordinates})) in Proposition~\ref{p: cylinder discrepancy estimated through H_m}. The inductive structure of the Rauzy graphs for the Fibonacci subshift (cf.~Propositions~\ref{p: structure of Gamma_m} to \ref{p: cycle properties}) contracts the stable coordinate, allowing the discrepancies accumulated over different scales to be controlled.
    \end{enumerate}
    
    \smallskip
    \item [\textbf{Part c:}]\emph{Non-freezing for weak log-log Dini potentials (Theorem~\ref{t: non-freezing for log Dini})}.

    \smallskip
    The proof builds upon a general non-freezing criterion (cf.~Lemma~\ref{l: a general non-freezing criterion}): a maximizing measure cannot be a frozen state for the potential if we can find invariant measures whose entropy gain over the maximizing measure dominates the energy they lose.  Let $\phi\in H^\omega(X,\R)$ be an $\omega$-continuous potential. It suffices to consider the case where $\int\!\phi\,\mathrm{d}\mu$ is not constant over all invariant measures $\mu$. Let $\mu$ be a $\phi$-maximizing measure. To construct measures for which the non-freezing criterion applies, we code the system with a mixing subshift of finite type $(X',\sigma)$ through a Markov partition. We start with a non-maximizing periodic point $y$, yielding a bad word $J_n=y_1y_2\ldots y_{l_n}$ of proper length $l_n$ for each $n\in\N$ so that the lift $\mu'$ of $\mu$ on $X'$ has limited weight on $[J_n]$ (cf.~Lemma~\ref{l: maximizing measure not much weight on nonmaximizing block}). 
    
    Let $n\in\N$. We take the $2n$-block distribution of $\mu$ and sample a word $u$ of length $2n$ with respect to this distribution. We choose uniformly a position $j$ among roughly $2n/(l_n+2s+1)$ places, cut $u$ at $j$-th position, insert the bad word $J_n$ into $u$, connect with two words of length $s\in\N_0$ since the system is mixing, and obtain a distribution on new words $w(u,j)$. 
    
    Sample a new word $V$ with respect to this distribution. Gluing independent copies of $V$ with fixed bridges of length $s$ and taking a $2n+l_n+3s$-translation average produces a $\sigma$-invariant measure $\nu_n$ on the $X'$ which projects to a $T$-invariant measure $\mu_n$ on $X$. Now we need to estimate $h_{\mu_n}(T)-h_{\mu}(T)$ and $\int_{X}\!\phi\,\mathrm{d}\mu-\int_X\!\phi\,\mathrm{d}\mu_n$. While the energy cost is controlled by $\frac{1}{n}\sum_{i=1}^n\omega(\lambda^{-i})+l_n/n$ for each $n\in\N$, where $\lambda$ is the expanding constant (cf.~Remark~\ref{r: expanding parameter and long pull backs}), the entropy gain is about $\log n_k/n_k$ along a subsequence $\{n_k\}_{k\in\N}$ in $\N$, where we also have $\sum_{i=1}^{n_k}\omega(\lambda^{-i})=o(\log n_k)$ (cf.~Lemma~\ref{l: associated modulus rate}) and $l_{n_k}=o(\log n_k)$ so that the non-freezing criterion applies.

    \smallskip
    \item [\textbf{Part d:}]\emph{Freezing construction (Theorem~\ref{t: freezing for not log Dini})}.

    \smallskip
    To prove Theorem~\ref{t: freezing for not log Dini}, we construct a subshift $K$ on the alphabet $\{0,\,1,\,2\}^{\N}$ to which a general freezing criterion applies. 
    \begin{enumerate}
    \smallskip
    \item[\textbf{Step d.1:}]\emph{(Freezing criterion for subshifts with positive entropy)}. We first establish the freezing criterion based on the two complexity assumptions on the subshift (cf.~Lemma~\ref{l: freezing criterion for positive entropy subshift}), namely, the growth rate of its language and the growth rate of its minimal forbidden words. More precisely, if the subshift has positive entropy $h$, its language complexity has at most exponential growth rate $h$ up to a polynomial factor, and the number of minimal forbidden words grows at most exponentially at rate $\gamma$ for some $\gamma<h$, then for every modulus of continuity $\omega$ not satisfying the weak log-log Dini condition, the plateau potential over the subshift will admit a freezing phase transition. To prove this result, we decompose an arbitrary word from left to right into bad phrases, each consisting of a legal prefix followed by a minimal forbidden word. For large inverse temperature $\beta$, the contribution of these bad phrases in the $n$-th partition function is then dominated by $e^{nh}$ due to the complexity assumptions, and every equilibrium state is thus forced to be supported on the subshift. 

    \smallskip
    \item[\textbf{Step d.2:}]\emph{(Construction of $K$)}. We construct a decreasing sequence of mixing subshifts of finite type $\{Y_n\}_{n\in\N_0}$ on $\{0,\,1,\,2\}^{\N}$, together with words $f_s\in\{0,\,1,\,2\}^{m_s}$ such that the language of $Y_s$ $\fL(Y_s)\=\bigl\{u\in\bigcup_{n=1}^{+\infty}\{0,1,2\}^{n}: f_1,\,f_2,\,\ldots,\,f_s \text{ are not factors of } u\bigr\}$, the length scales $m_s$ grow geometrically, and the entropy gaps $h_{s-1}-h_s$ should decay geometrically. These features will later translate into the complexity assumptions required by Lemma~\ref{l: freezing criterion for positive entropy subshift}.

    \hspace*{1.5em} When necessary, we insert stationary stages for some $s\in\N$ in which $Y_{s+1}=Y_{s}$ while $m_{s+1}=e^{\theta m_s}$. In a genuine stage, let $\cO$ be the first periodic orbit (in some fixed enumeration of periodic orbits) surviving in $Y_s$, and apply the periodic block deletion (cf.~Lemma~\ref{l: deletion of periodic orbit}) to a sufficiently long block $w_m(\cO)$ in $\cO$. Setting $f_{s+1}\=w_m(\cO)$, we delete $\cO$ in $Y_{s+1}$ and decrease the entropy by a small amount. As a result, we obtain $K\=\bigcap_{s=0}^{+\infty}Y_s$. 

    \smallskip
    \item [\textbf{Step d.3:}]\emph{(Properties of $K$)}. Every periodic orbit is deleted at $Y_{s+1}$ for some $s\in\N$, so $K$ contains no periodic points. The geometric decay of the entropy gaps guarantees that $h_{\mathrm{top}}(\sigma|_K)>0$. Uniqueness of the measure of maximal entropy follows from the upper semi-continuity of entropy (cf.~Lemma~\ref{l: entropy-isolation}). The complexity estimate is more delicate. Representing $K$ by a deterministic labelled graph and applying Lemma~\ref{l: finite-presentation} controls the language growth. The minimal forbidden words of $K$ are precisely the words $f_{s+1}$ from the genuine stages where $Y_{s+1}\subsetneq Y_s$ so that for each $m\in\N$, the number of minimal forbidden words of length $m$ is at most $1$. Thus, Lemma~\ref{l: freezing criterion for positive entropy subshift} applies and the proof is complete.
    \end{enumerate}

    \smallskip
    \item [\textbf{Part e:}]\emph{Super-exponential suppression (Theorem~\ref{t: super-exponential suppression for non-Dini})}.

    \smallskip
    Let $\omega$ be a modulus of continuity not satisfying the Dini condition and $\cO$ be a periodic orbit. We first find an open set $U$ which captures every orbit outside $\cO$ so that $\cO$ is the only orbit never entering $U$. It suffices to show that $\lim_{\beta\to+\infty}\frac{1}{\beta}\log\sup_{\mu\in\fE(\sigma,-\beta\omega(d_2(\,\cdot\,,\cO)))}\mu(U)=-\infty$.
    
    Take $\mu\in\fE(\sigma,-\beta\omega(d_2(\,\cdot\,,\cO)))$. To estimate $\mu(U)$, we induce on $U$ and calculate both the energy $\int\!\omega(d_2(\,\cdot\,,\cO))\,\mathrm{d}\mu$ and the entropy $h_{\mu}(\sigma)$ in terms of the return time distribution $p_n\=\frac{1}{\mu(U)}\mu(\{x\in U:\tau(x)=n\})$, where the energy on $\{x\in U:\tau(x)=n\}$ is about $\assocsum{2}{n}(\delta)$ for some $\delta\in(0,+\infty)$. As $\mu$ is an equilibrium state for $-\beta\omega(d_2(\,\cdot\,,\cO))$, we get an inequality in terms of $p_n$ and $\assocsum{2}{n}(\delta)$ (cf.~(\ref{ineq: entropy-energy written out})). While Kac's formula relates $\mu(U)$ to $\sum_{n=1}^{+\infty}np_n$, we combine $\sum_{n=1}^{+\infty}np_n$ with the terms in (\ref{ineq: entropy-energy written out}) and apply Jensen's inequality to get a partition function $Z_{\beta'}$. As $\omega$ does not satisfy the Dini condition, we see that $\lim_{\beta'\to+\infty}Z_{\beta'}=0$, which forces $\sum_{n=1}^{+\infty}np_n$ to be large, leading to $\lim_{\beta\to+\infty}\frac{1}{\beta}\log\bigl(\sup_{\mu\in\fE(\sigma,-\beta\omega(d_2(\,\cdot\,,\cO)))}\mu(U)\bigr)=-\infty$.
\end{enumerate}

\subsection{Organization of the article}

We collect some key notation used throughout the article
in Subsection~\ref{ss: notation}. In Section~\ref{s: TTF for Dini}, we develop the tropical
thermodynamic formalism for Dini potentials and prove
Theorems~\ref{t: Dini quanta no freezing}, \ref{t: existence}, and~\ref{t: zero-temperature limit}. Some of the technical proofs that are similar to the ones for H\"{o}lder potentials are deferred to Appendix~\ref{ss: Appendix A}. 

In Section~\ref{s: failure of TPO}, we review properties of the Fibonacci subshift and prove robust maximizing on it, establishing Theorems~\ref{t: energy transport estimate} and~\ref{t: non TPO}. In Section~\ref{s: non-freezing for log Dini}, we discuss the non-freezing criterion, establishing Theorem~\ref{t: non-freezing for log Dini}. In Section~\ref{s: freezing for non log Dini}, we prove a freezing criterion for subshifts with positive entropy (cf.~Lemma~\ref{l: freezing criterion for positive entropy subshift}) and establish Theorem~\ref{t: freezing for not log Dini}. Theorem~\ref{t: super-exponential suppression for non-Dini} is established in Section~\ref{s: super-exponential suppression beyond Dini}.

\subsection{Notation}\label{ss: notation}
We follow the convention that $\N=\{1,\,2,\,3,\,\ldots\}$ and $\N_0\=\N\cup\{0\}$. The cardinality of a set $A$ is denoted by $\card A$.

Let $(X,d)$ be a metric space. For subsets $A,\,B\subseteq X$, we set $d(A,B) \= \inf \{d(x,y) : x\in A,\,y\in B\}$, and $d(A,x)=d(x,A) \= d(A,\{x\})$ for each $x\in X$. For each subset $Y\subseteq X$, we denote the diameter of $Y$ by $\diam Y \= \sup\{d(x,y) : x, \, y\in Y\}$, for each $\epsilon>0$, denote the $\epsilon$-neighbourhood of $Y$ (in $X$) by $B(Y,\epsilon )\= \{x\in X:d (x,Y )<\epsilon \}$, and the closure of $B(Y,\epsilon )$ by $\overline{B}(Y,\epsilon)$. For each $y\in X$ and each $\epsilon>0$, write $B(y,\epsilon)\=B(\{y\},\epsilon)$.

Let $C(X,\R)$ denote the space of continuous functions $\phi\: X\to\R$ with the supremum norm $\norm{\phi}_{C^0}\=\max_{x\in X}\abs{\phi(x)}$. Let $\cP(X)$ denote the set of Borel probability measures on $X$.

For a finite nonempty set $A$, referred to as an \emph{alphabet}, the \emph{(one-sided) full shift} on $A$ is the set 
$A^{\N}\=\bigl\{\{x_n\}_{n=1}^{+\infty}:x_n\in A\text{ for all }n\in\N\bigr\}$, which is compact equipped with the product topology. Elements of $A^{\N}$ will be written as $x=x_1x_2\ldots=\{x_n\}_{n=1}^{+\infty}$. A \emph{word} on $A$ is a finite sequence $y=y_1y_2\ldots y_k\in A^k$, where $k\in\N$ is the \emph{length} of $y$, denoted by $\abs{y}$. We use $\epsilon$ to denote the empty word and $\abs{\epsilon}\=0$. A word $y'$ is a \emph{factor} of $y=y_1y_2\ldots y_k$ if there exist $1\leq i\leq j\leq k$, $y'=y_iy_{i+1}\ldots y_j$. For each $w=\{w_n\}_{n=1}^{+\infty}\in A^{\N}$ and each $k\in\N$, denote $w|_k\=w_1w_2\ldots w_k$. A word $y=y_1y_2\ldots y_k$ on $A$ is said to \emph{occur} in some $w\in A^{\N}$ if there exists some $j\in\N$ such that $w_{j+i-1}=y_i$ for all integers $1\leq i\leq k$. 

The shift map $\sigma\:A^{\N}\to A^{\N}$ is given by $(\sigma x)_i\=x_{i+1}$ for all $x\in A^{\N}$. For each $\lambda\in (1,+\infty)$, $A^{\N}$ is metrized by the distance function $d_\lambda$ given by $d_{\lambda}\bigl(\{x_n\}_{n=1}^{+\infty},\{y_n\}_{n=1}^{+\infty}\bigr)\= \lambda^{1-p}$, where $p$ is the smallest possible integer with $x_p\neq y_p$, and $d_{\lambda}\bigl(\{x_n\}_{n=1}^{+\infty},\{y_n\}_{n=1}^{+\infty}\bigr)\=0$ if $x_n=y_n$ for all $n\in\N$. A \emph{subshift} of $A^{\N}$ is a closed $\sigma$-invariant subset $X\subseteq A^{\N}$. Let $X$ be a subshift of $A^\N$. A word $y$ on $A$ is \emph{allowed} in $X$ if there exists $x\in X$ and $k\in\N$ such that $y=x|_k$. The \emph{language} of $X$, denoted by $\fL(X)$, is the set of all its allowed words. For each $n\in\N$, let $\fL_n(X)$ be the subset of $\fL(X)$ consisting of length $n$ words. A subshift $X$ of $A^\N$ is said to be a \emph{subshift of finite type} if there exists a finite set $F$ of words on $A$ so that $X=\bigl\{x\in A^\N:\text{ every word in }F\text{ does not occur in }x\big\}$. If the finite set $F$ can be chosen to be a subset of $A^{r+1}$ for some $r\in\N$, then $X$ is called an \emph{$r$-step} subshift of finite type. A word $y$ on $A$ that is not allowed in $X$ is called \emph{minimal forbidden} if every proper factor of $y$ is allowed in $X$. For some subshift $Y$ of $A^\N$, we say that $Y$ is a \emph{subshift} of $X$ if $Y\subseteq X$. For each $m\in\N$, let $b_X(m)$ denote the number of minimal forbidden words of length $m$. We write $h_{\mu}$ and $h_{\topo}$ for the measure-theoretic entropy and the topological entropy, respectively. Denote
\begin{equation*}
    \MME(X)\=\bigl\{\mu\in\cM\bigl(A^{\N},\sigma\bigr):\mu(X)=1,\,h_{\mu}(\sigma|_X)=h_{\topo}(\sigma|_X)\bigr\}.
\end{equation*}

Let $T\: X\to X$ be a continuous map on a compact metric space $X$ and $f\in C(X,\R)$. Denote the set of equilibrium states for $f$ by
\begin{equation*}
    \fE(T,f)\=\biggl\{\mu\in\cM(X,T):h_{\mu}(T)+\int\!f\,\mathrm{d}\mu=P(T,f)\biggr\}.
\end{equation*}

\medskip
\noindent{\bf Acknowledgments.} 
Y.~H.\ and Z.~L.\ were partially supported by Beijing Natural Science Foundation (JQ25001) and National Natural Science Foundation of China (12471083). Y.~J.\ wants to thank for the hospitality of Peking University during his visit in January 2025 to attend the Conference on Holomorphic Dynamics and Fractal Geometry, during which he learnt the arXiv version of Z.~L.\ and Y.~S.'s work \cite{LS26} and initiated the collaboration. We also want to express our gratitude to Shibing Chen for discussions and references on optimal transport, and to Junwen Huang for helpful comments.

The authors confirm that no artificial intelligence (AI) or machine-learning tools were used in the research and writing of this manuscript.

\section{Tropical thermodynamic formalism for Dini potentials}\label{s: TTF for Dini}

In this section, we develop the theory of tropical thermodynamic formalism for Dini potentials and establish Theorem~\ref{t: Dini quanta no freezing}. We review expanding maps and Dini potentials in Subsection~\ref{ss: expanding maps and dini potentials}, fix some key notation and terminology in Subsection~\ref{ss: backgrounds for TTF}, and provide a proof of Theorem~\ref{t: Dini quanta no freezing} assuming Theorems~\ref{t: existence} and~\ref{t: zero-temperature limit} in Subsection~\ref{ss: Dini quanta no freezing}. The existence of tropical eigenfunctions and tropical eigendensities is investigated in Subsection~\ref{ss: existence}, which leads to the proof of Theorem~\ref{t: existence}. We investigate the logarithmic-type zero-temperature limits and establish Theorems~\ref{t: zero-temperature limit} in Subsection~\ref{ss: zero temperature limits}. The representations and uniqueness of these tropical eigenobjects, as well as the proof of Theorem~\ref{t: Dini quanta no freezing}~(ii), are presented in Subsection~\ref{ss: representation and uniqueness}. 

Since many of the arguments closely parallel those for Hölder potentials, we keep in the main text only the proofs that are concise and bring out the features specific to the Dini setting, and defer the remaining proofs to Appendix~\ref{ss: Appendix A}.

\subsection{Expanding maps and Dini potentials}\label{ss: expanding maps and dini potentials}
In this subsection, we review expanding maps and Dini potentials and fix related notations, which will be used throughout the article.

A map $T$ on a topological space $X$ is called a \emph{covering map} if $T$ is open, continuous, surjective, and is a local homeomorphism. Let $(X,d)$ be a metric space. A map $T\:X\to X$ is said to be  \emph{distance-expanding} if there exist constants $\lambda>1$ and $\gamma>0$ such that for all $x,\, y\in X$, $d(x,y)\in(0,\gamma)$ implies $d(T(x),T(y))>\lambda d(x,y)$. We say that a covering map $T\:X\to X$ is an \emph{expanding covering map} if $T$ is distance-expanding. Note that the map $T\: X\to X$ constitutes an expanding covering map if and only if it is open, continuous, surjective, and distance-expanding (cf.~\cite[Section~4.1]{PU10}).

The following lemma is well known (see e.g.~\cite[Lemmas~4.1.2 and 4.1.4]{PU10}) and we omit its proof.
\begin{lemma}\label{l:expanding systems}
	Let $T\:X \rightarrow  X$ be an open continuous distance-expanding map. Let $\lambda>1$ and $\gamma>0$ be the constants in the distance-expanding property of $T$. Then there exists a constant $\xi>0$ such that for each $x$ in $X$, $B(T(x),\xi)\subseteq T(B(x,\gamma/2))$. Moreover, the restriction $T|_{B(x, \gamma/2)}$ is injective and its inverse map $T_x^{-1} \: B(T(x),\xi) \rightarrow  B(x, \gamma/2)$ has the property that $d \bigl( T_x^{-1}(y),T_x^{-1}(z) \bigr)\leq\lambda^{-1}d(y,z)$. Furthermore, 
    \begin{equation*}
 \sup\bigl\{\card \bigl(T^{-1}(x)\bigr):x\in X\bigr\}\eqqcolon N<+\infty.
	\end{equation*} 
\end{lemma}
\begin{rem}\label{r: expanding parameter and long pull backs}
    We call the triple $(\lambda,\gamma,\xi)$ expanding constants of $T\:X\to X$. Since in our article, the metric $d$ is fixed for every compact space $X$, we will say that a quantity depends on $T$ if it depends on $T$ and $d$.
    
     For each $n\in\N$, we denote $T_x^{-n}\: B(T^n(x),\xi) \rightarrow  B(x,\lambda^{-n}\xi)$ as the composition of inverse maps $T_{T^i(x)}^{-1}$, $i = 0,\, 1,\,\ldots, \, n-1$. 
\end{rem}

Recall that a nondecreasing function $\omega\:[0,+\infty)\to [0,+\infty)$ with $\lim_{t\to0^+}\omega(t)=\omega(0)=0$ is called a \emph{modulus of continuity}. Let $(X,d)$ be a compact metric space and $\omega$ be a modulus of continuity. Denote, for each function $\phi\:X\to\R$ and each $a\in (0,+\infty]$,
\begin{equation*}
    \abs{\phi}_{\omega,a}\=\sup\biggl\{\frac{\abs{\phi(x)-\phi(y)}}{\omega(d(x,y))}:x,\,y\in X,\,0<d(x,y)\leq a\biggr\}\quad\text{ and }\quad\abs{\phi}_{\omega}\=\abs{\phi}_{\omega,+\infty}.
\end{equation*}
A function $\phi\: X\to \R$ is said to be \emph{$\omega$-continuous} if $\abs{\phi}_\omega<+\infty$. We define $H^{\omega}(X,\R)$ to be the Banach space of $\omega$-continuous functions $\phi\:X\to\R$ equipped with norm $\norm{\phi}_\omega\=\norm{\phi}_{C^0}+\abs{\phi}_\omega$.

Recall that a modulus of continuity $\omega$ is said to satisfy the \emph{Dini condition} if it satisfies (\ref{ineq: defi of Dini}). The set of such $\omega$ is denoted by $\sD$. We call $\phi\: X\to\R$ a \emph{Dini potential} if $\phi\in H^\omega(X,\R)$ for some $\omega\in\sD$.

For all modulus of continuity $\omega$ and constants $\lambda\in (1,+\infty)$, $k\in\N_0\cup\{+\infty\}$, and $t\in[0,+\infty)$, define
\begin{equation}\label{eq: defi of associated modulus}
  \assocsum{\lambda}{k}(t)\=\begin{cases}
  \sum_{n=1}^k\omega(\lambda^{-n}t) & \text{if } k\in\N\cup\{+\infty\},\\
  0 &\text{if }k=0,
  \end{cases}\quad\text{ and }\quad\assocmodu(t)\=\assocsum{\lambda}{+\infty}(t).
\end{equation}

\begin{lemma}\label{l: associated modulus of continuity}
   If $\omega\in\sD$ and $\lambda>1$, then $\assocmodu\:[0,+\infty)\to[0,+\infty)$ is a modulus of continuity.
\end{lemma}
%\begin{remark}
    %In general, the two spaces $H^\omega(X,\R)$ is a proper subset of $H^{\assocmodu}(X,\R)$.
%\end{remark}
\begin{proof}
    Since $\omega$ is nondecreasing and vanishes only at $0$, we see that $\assocmodu$ is nondecreasing, $\assocmodu(t)>0$ for $t\in (0,+\infty)$, and $\assocmodu(0)=0$. Since $\omega\in\sD$, it follows from \cite[Lemma~1]{FJ01a} that $\assocmodu$ is real-valued and $\lim_{t\to0^+}\assocmodu(t)=0$.
\end{proof}

Further assume that $T\: X\to X$ is an open continuous distance-expanding map. The Ruelle operator $\cR_\phi\: C(X,\R)\to C(X,\R)$ for $\phi\in C(X,\R)$ is defined by 
\begin{equation}\label{eq: defi of the Ruelle operator}
    \cR_\phi(u)(x)=\sum_{y\in T^{-1}(x)}e^{\phi(y)}u(y)
\end{equation}
for each $u\in C(X,\R)$ and $x\in X$. We recall from \cite{FJ01a} that if $T$ is \emph{mixing} and $\omega\in\sD$, then for all $\phi\in H^\omega(X,\R)$, $\cR_{\phi}$ (resp.\ its adjoint operator $\cR_{\phi}^*$) has a unique eigenfunction in $C(X,\R)$ (resp.\ eigenmeasure), up to a constant, associated with the eigenvalue of maximal modulus, i.e., $e^{P(T, \phi)}$, where $P(T, \phi)$ is the topological pressure of $T$ with respect to the potential $\phi$. Moreover, let $m_\phi$ be the unique eigenprobability of $\cR_\phi^*$ associated with eigenvalue $e^{P(T,\phi)}$ and $u_\phi$ be the unique eigenfunction of $\cR_\phi$ in $C(X,\R)$ associated with eigenvalue $e^{P(T,\phi)}$ satisfying $\int \! u_\phi \, \mathrm{d} m_\phi=1$, then $u_\phi$ is strictly positive and
\begin{equation*}
	\mu_\phi\=u_\phi\cdot m_\phi
	\end{equation*} 
    is the unique equilibrium state.

\subsection{Basic notions for tropical thermodynamic formalism}\label{ss: backgrounds for TTF}
In this subsection, we review the tropical algebra and tropical functional analysis notions, which are mainly used in this section. 

For a semiring $(R,+,\times,0_R,1_R)$, an \emph{$R$-semimodule} is a set $S$ equipped with a binary operation $+\: S\times S \rightarrow  S$ and a map $\times\: R\times S \rightarrow  S$ provided that the following axioms are satisfied for some $0_S \in S$ and for all $a,\,b,\,c\in S$ and $\lambda,\,\lambda'\in R$:
\begin{enumerate}[label=\rm{(\roman*)}]	
	\smallskip
	\item $0_S+ a=a$, $0_R\times a=0_S$, and $1_R\times a=a$.
	\smallskip
    \item $(a+ b)+ c=a+(b+ c)$ and $a+ b=b+ a$.
	\smallskip
	\item $\lambda\times(a+ b)=(\lambda\times a)+(\lambda\times b)$ and $(\lambda+\lambda')\times a=(\lambda\times a)+(\lambda'\times a)$.
	\smallskip
	\item $(\lambda \times\lambda')\times a=\lambda \times(\lambda'\times a)$.
\end{enumerate}
A map $\cL\:S_1\to S_2$ between two $R$-semimodules $S_1$ and $S_2$ is called a \emph{semimodule homomorphism} if $\cL(a+b)=\cL(a)+\cL(b)$ and $\cL(\lambda\times a)=\lambda\times\cL(a)$ for all $a,\,b\in S_1$ and $\lambda\in R$. A subset $M$ of an $R$-semimodule $S$ is called a \emph{subsemimodule} if it is closed under the addition and scalar multiplication of $S$.

We consider $\Rm \= \R\cup\{+\infty,\,-\infty\}$ equipped with the \emph{tropical} (or \emph{max-plus}) \emph{algebra}:
\begin{equation*}
	x\oplus y \= \max\{x,\,y\},\quad x\otimes y \=  x+y,\quad\text{ for } x,\, y\in\Rm.
\end{equation*}
Here $-\infty$ is seen as the tropical additive identity, and we adopt the convention that $(+\infty)\otimes(-\infty)=(-\infty)\otimes(+\infty)=-\infty$. For each $\cA\subseteq\Rm$, we use $\oplusu{x\in\cA}x$ to denote the supremum in $\Rm$, and $\oplusu{x\in\emptyset}x$ is defined to be $-\infty$. It is straightforward to check that $\bigl(\Rm,\oplus,\otimes\bigr)$ is a semiring and 
\begin{equation}\label{eq: exchangability of numbers}
    c\otimes \bigl(\oplusu{x\in\cA}x\bigr)=\oplusu{x\in\cA}(c\otimes x)
\end{equation}
for all $\cA\subseteq\Rm$ and $c\in\Rm$. The basis $\bigcup_{a,\,b\in\R}\{(a,b),[-\infty,a),(b,+\infty]\}$ generates the standard topology on $\Rm$. Denote $\Rminus\=\R\cup\{-\infty\}$.

Let $\Rm^{Y}$ denote the space of $\Rm$-valued functions on a set $Y$. The space $\Rm^Y$ with $(u\oplus v)(x) \=  u(x)\oplus v(x)$ for all $u,\,v\in\Rm^Y$ and $(\lambda\otimes u)(x) \= \lambda\otimes u(x)$ for all $u\in\Rm^Y$ and $\lambda\in\Rm$ is an $\Rm$-semimodule. The space $\Rm^Y$ has the natural order: $u\fless v$ if and only if $u(x)\leq v(x)$ for all $x$ in $Y$, i.e., $u\fless v$ if and only if $u\oplus v=v$. For each subset $U\subseteq\Rm^Y$, we use $\oplusu{u\in U} u$ to denote the pointwise supremum in $\Rm^Y$. 

In the sequel, $\oplus$ always means $\sup$. Denote by $\tai\=-\infty$ (resp.\ $\tmi\=0$) the tropical additive (resp.\ multiplicative) identity of $\Rm$. We use $\mathbbold{1}_S$, $\tmif_S$, $\taif_S$, and $\tinftyf_S$ to represent the constant $1$, $\tmi$, $\tai$, and $+\infty$ functions on a set $S$, respectively, and define $\mathbbold{0}_S \= \tmif_S$. We often omit the subscript $S$ when it is clear from the context.

Now we review the following tropical functional analysis notions.

\begin{definition}[Tropical completion]
	Let $S\subseteq\Rm^Y$ be a set of $\Rm$-valued functions on a set $Y$. The set $S$ is said to be \emph{tropically complete} if for each subset $U\subseteq S$, $\oplusu{u\in U}u$ is in $S$. The \emph{tropical completion} of $S$, denoted by $\twidehat{S}$, is defined to be the intersection of all tropically complete subsets of $\Rm^Y$ containing $S$.
\end{definition}

\begin{definition}[Tropical continuity]
	Let $Y$ be a set and $V, \, W$ be two $\Rm$-semimodules.
\begin{enumerate}[label=\rm{(\roman*)}]	
\smallskip
\item A map $\cL\: V\to W$ is \emph{tropical linear} if it is a semimodule homomorphism.

\smallskip
\item A map $\cL\: V\to\Rm$ is called a \emph{tropical functional}.

\smallskip
\item Assume that $V,\,W$ are two subsemimodules of $\Rm^Y$. We call a tropical linear map $\cL\: V\to W$ \emph{tropically continuous} if $\cL$ can be extended uniquely to a tropical linear map $\cL \:\twidehat{V} \rightarrow \twidehat{W}$ satisfying
\begin{equation*}
	\cL\bigl(\oplusu{u\in U}u\bigr)=\oplusu{u\in U}\cL(u)
\end{equation*} 
for every subset $U\subseteq\twidehat{V}$.
\end{enumerate}
When more than one property or term above is mentioned, we often adopt the convention to use only one ``tropical(ly)'', e.g.\ a tropically continuous linear functional and a tropical linear functional.
\end{definition}

\begin{definition}\label{d: defi of lu}
    Let $Y$ be a set. For $u,v\in\Rm^{Y}$, we define  
    \begin{equation}\label{eq: defi of lu}
    u \odiv v\=\oplusu{\substack{\lambda\otimes v\leq u\\\lambda\in\Rm}}\lambda\in\Rm\quad\text{ and }\quad
    l_u(v) \= -(u\odiv v)=\oplusu{x\in Y}(v(x)\otimes (-u(x))).
    \end{equation}
\end{definition}

\begin{definition}[Tropical dual space]
	Let $Y$ be a set and $V$ be a subsemimodule of $\Rm^Y$. The \emph{tropical dual space} $V^{\ostar}$ of $V$ is the space consisting of all tropically continuous linear functionals from $\twidehat{V}$ to $\Rm$.
\end{definition}

\begin{definition}[Density space]\label{d: density space}
    Let $X$ be a compact metric space. The space $\density(X)$ of \emph{densities} is given by
    \begin{equation*}
\density(X)\=\{b \: X\rightarrow\Rminus : b \text{ is upper semi-continuous}\}\cup\{\tinftyf\}.
\end{equation*}
\end{definition}

Let $X$ be a compact metric space. We denote the tropical completion of $C(X,\R)$ by $\twidehat{C}(X,\R)$. Note that for $S\=C(X,\R)\cup\{\tinftyf,\,\taif\}$, $\twidehat{S}=\twidehat{C}(X,\R)$. We thus denote $C(X,\R)^{\ostar}\= S^{\ostar}$ despite the fact that $C(X,\R)$ is not a subsemimodule of $\Rm^X$.

\begin{prop}[{\cite[Proposition~2.13]{LS26}}]\label{p: dual space equals completion}
     Let $X$ be a compact metric space. Then the restriction $l_{-b} \: \twidehat{C}(X,\R)\rightarrow \Rm$ is a tropically continuous linear functional in $C(X,\R)^{\ostar}$ for each $b\in\density(X)$, and the map $l_{-\bullet}\:\density(X) \rightarrow  C(X,\R)^{\ostar}$ given by $b\mapsto l_{-b}$ is a bijection.
\end{prop}

\begin{rem}\label{r: equivalence of densities}
    We call $b\in\density(X)$ the density in $\density(X)$ of the tropically continuous linear functional $l_{-b}$ (cf.~(\ref{eq: defi of lu})). We say that $b_1\: X\to\Rm$ and $b_2\: X\to\Rm$ are \emph{equivalent} if $\oplusu{x\in X}(f(x)\otimes b_1(x))=\oplusu{x\in X}(f(x)\otimes b_2(x))$ for all $f\in C(X,\R)$. For every $b_1\:X\to\Rm$, there exists a unique $b_2\in\density(X)$ equivalent to $b_1$ since every tropical linear functional on $C(X,\R)\cup\{\tinftyf,\,\taif\}$ is tropical continuous (cf.~\cite[Proposition~2.7]{LS26}). 
\end{rem}

Assume that $T\:X\to X$ is a continuous map on the compact metric space $(X,d)$. Let $\cM(X,T)$ denote the set of $T$-invariant Borel probability measures on $X$. For each $\phi\in C(X,\R)$, denote $\toverline{\phi}\=\phi-\energy$, where $\energy$ is the ergodic supremum of $\phi$ (cf.~(\ref{eq: defi of ergodic supremum})), and the set of $\phi$-maximizing measures by $\cM_{\max}(T,\phi)\=\bigl\{\mu\in\cM(X,T):\int\!\phi\,\mathrm{d}\mu=\energy\bigr\}$.

For a potential $\phi$ in $\Rm
^X$, denote 
\begin{equation}\label{eq: defi for S_n} 
	S_n^T\phi(x)\= \phi(x)\otimes \phi(T(x))\otimes\cdots\otimes \phi\bigl(T^{n-1}(x)\bigr)
\end{equation}
for all $n\in\N$ and $x\in X$; we often omit the superscript $T$ when it is understood from the context. For each $u\in C(X,\R)$ and each $\epsilon \geq 0$, denote
\begin{equation*}
    \var_u(\epsilon)\=\sup\{\abs{u(x)-u(y)}:x,\,y\in X,\,d(x,y)\leq\epsilon\}.
\end{equation*}

Recall that (cf.~\cite[Definition~1.1]{LS26}) the \emph{Bousch operator} $\cL_\phi\:\Rm^X\to\Rm^X$ is defined for $\phi\in \Rm^X$ by 
\begin{equation}\label{eq: defi of the Bousch operator}
    \cL_\phi(u)(x)\=\oplusu{y\in T^{-1}(x)}(u(y)\otimes \phi(y))
\end{equation}
and the \emph{tropical adjoint Bousch operator} $\cL_{\phi}^{\ostar}\:\density(X)\to \density(X)$ is defined for $\phi\in C(X,\R)$ by
\begin{equation}\label{eq: defi of the adjoint Bousch operator}
    \cL_\phi^{\ostar}(b)(x)\=b(T(x))\otimes \phi(x),
\end{equation}
where $\density(X)\=\{b \: X\rightarrow\Rminus : b \text{ is upper semi-continuous}\}\cup\{\tinftyf\}$ (cf.~Proposition~\ref{p: dual space equals completion} and~\cite[Remark~3.13]{LS26} for the reason of these definitions). Clearly, for all $n\in\N$, $x\in X$, $u\in\Rm^X$, and $b\in\density(X)$, 
\begin{align}
    \cL_\phi^n(u)(x) & =\oplusu{y\in T^{-n}(x)}(u(y)\otimes S_n\phi(y)) , \label{eq: iteration of Bousch}\\
    \bigl(\cL_\phi^{\ostar}\bigr)^n(b)(x)&=b(T^n(x))\otimes S_n\phi(x). \label{eq: iteration of adjoint Bousch}
\end{align}

Let $V$ be a subset of $\Rm^X$.
We recall the following related notions:
\begin{enumerate}[label=\rm{(\roman*)}]	
	\smallskip
	\item $Q\in\Rm$ is a \emph{tropical eigenvalue} of $\cL_\phi$ on $V$ if there exists $u\in V$ such that $\cL_\phi(u)=Q\otimes u$ and $u(x)\in\R$ for some $x\in X$. Such a function $u\in V$ is called a \emph{tropical eigenfunction} of $\cL_\phi$ in $V$ (associated with eigenvalue $Q$). The \emph{tropical eigenspace} $\es{\cL_\phi}{V}{Q}$ of $\cL_\phi$ for the (tropical) eigenvalue $Q$ consists of all $u\in V$ such that $\cL_\phi(u)=Q\otimes u$.

    \smallskip
	\item $Q\in\Rm$ is a \emph{tropical eigenvalue} of $\cL_\phi^{\ostar}$ (on $\density(X)$) if there exists $b\in \density(X)\smallsetminus\{\tinftyf,\,\taif\}$ of $\cL_\phi^{\ostar}$ satisfying $\cL_\phi^{\ostar}(b)=Q\otimes b$. Such a function $b$ is called a \emph{tropical eigendensity} of $\cL_\phi^{\ostar}$ (associated with eigenvalue $Q$). The \emph{tropical eigenspace} $\esbig{\cL_\phi^{\ostar}}{\density(X)}{Q}$ of $\cL_\phi^{\ostar}$ for the (tropical) eigenvalue $Q$ consists of all $b\in \density(X)$ such that $\cL_\phi^{\ostar}(b)=Q\otimes b$.

    \smallskip
    \item A subset $E$ of $\Rm^X$ is \emph{of tropical dimension one}, written as $\tdim E=1$, if there exists $u\in E$ such that $E\smallsetminus\{\taif,\,\tinftyf,\,(+\infty)\otimes u\}=\{a\otimes u:a\in\R\}$ and $u(x)\in\R$ for some $x\in X$.

    \smallskip
    \item $u\in C(X,\R)$ is a \emph{sub-action} for the potential $\phi$ if $\cL_{\toverline{\phi}}(u)\fless u$.
\end{enumerate}

Here we introduce the notion of \emph{invariant tropical Gibbs density}.

\begin{definition}[Invariant density]\label{d: T-invariance density}
    Let $T\: X\to X$ be a continuous map on the compact metric space $X$. We say a density $b\in\density(X)$ is \emph{$T$-invariant}, or simply \emph{invariant}, if $b(x)=\oplusu{y\in T^{-1}(x)}b(y)$ for all $x\in X$.
\end{definition}

\begin{definition}[Tropical Gibbs density]\label{d: tropical gibbs density}
    Let $T\: X\to X$ be a transitive expanding covering map with expanding constants $(\lambda,\gamma,\xi)$. We say a density $b\in\density(X)$ is a \emph{tropical Gibbs density} for some potential $\phi\in C(X,\R)$ if there exists a constant $M>0$ such that for all $n\in\N$ and $x\in X$,
    \begin{equation}\label{ineq: defining inequalities for tropical Gibbs}
       S_n\toverline{\phi}(x)-M\leq  \oplusu{y\in T_{x}^{-n}(B(T^n(x),\xi))}b(y)\leq S_n\toverline{\phi}(x)+M.
    \end{equation}
\end{definition}

The logarithmic-type zero-temperature limits connect thermodynamic objects to their tropical counterparts. Suppose $T\: X\to X$ is a mixing expanding covering map, $\omega\in\sD$, and $\phi\in H^\omega(X,\R)$. Denote
\begin{align} 
    \widetilde{\phi} \= \phi +\log u_{\phi}&-\log u_{\phi}\circ T-P(T, \phi),\label{eq: defi of normalized potential for thermo}\\
	l_{\beta}^{\mu}(f) \= \frac{1}{\beta}\log\int \! e^{\beta f} \, \mathrm{d} \mu_{\beta \phi},\quad&\text{ and }\quad
	l_{\beta}^{m}(f) \= \frac{1}{\beta}\log\int \! e^{\beta f} \, \mathrm{d} m_{\beta \phi},\label{eq: defi of l_beta}
\end{align}
for all $f\in C(X,\R)\cup\{\tinftyf,\,\taif\}$ and $\beta>0$.\footnote{For the constant infinity functions, the values of the functionals are defined by the standard convention for integrals of Borel measurable functions valued in $[0,+\infty]$ and the natural extension that $\log (+\infty)=+\infty$ and $\log 0=-\infty$. Hence, $l_{\beta}^\mu(\tinftyf)=l_{\beta}^m(\tinftyf)=+\infty$ and $l_{\beta}^{\mu}(\taif)=l_{\beta}^m(\taif)=-\infty$ for all $\beta>0$.}
We call pointwise limits of subsequences of $\bigl\{\frac{1}{\beta}\log u_{\beta \phi}\bigr\}_{\beta\in(1,+\infty)}$, $\bigl\{l_{\beta}^\mu\bigr\}_{\beta\in(1,+\infty)}$, and $\bigl\{l_{\beta}^m\bigr\}_{\beta\in(1,+\infty)}$, as $\beta\to+\infty$,  ``logarithmic-type zero-temperature limits''. These limits naturally connect to the large deviation principles.

\begin{definition}[Large deviation principle]\label{d: ldp}
    For all $r>0$, a family of Borel probability measures $\{\nu_{\beta}\}_{\beta\in(r,+\infty)}$ on a topological space $Y$ satisfies the \emph{large deviation principle as $\beta\to+\infty$} if there exists a lower semi-continuous function $I \: Y \rightarrow [0,+\infty]$ (called the \emph{rate function}) for which the following two inequalities hold:
\begin{align*}
		\liminf_{\beta\to+\infty}\frac{1}{\beta}\log(\nu_{\beta}(\cG))&\geq-\inf_{x\in \cG}I(x),\quad\text{ for every open set } \cG\subseteq Y,  \\
		\limsup_{\beta\to+\infty}\frac{1}{\beta}\log(\nu_{\beta}(\cK))&\leq-\inf_{x\in \cK}I(x),\quad\text{ for every closed set } \cK\subseteq Y. 
\end{align*}
Here $\beta$ is called the \emph{inverse temperature}. Note that $I(x)\in [0,+\infty]$ for all $x\in Y$, which follows from the fact that $\nu_{\beta}$ is a probability measure.	It immediately follows from Remark~(a) after \cite[Lemma~4.1.4]{DZ09} that if $Y$ is a compact metric space, then the rate function is unique as long as a large deviation principle holds. 
\end{definition}

\begin{rem}\label{r: two LDP}
By \cite[Theorems~4.3.1 and~4.4.2]{DZ09}, when $Y$ is a compact metric space, a family of probability measures $\{\nu_\beta\}_{\beta\in(r,+\infty)}$ satisfies the large deviation principle as $\beta\to+\infty$ with rate function $I$ if and only if for each $f\in C(Y,\R)$,
	\begin{equation}\label{eq: LDP in functionals}
		\lim_{\beta\to+\infty}\frac{1}{\beta}\log\int\!e^{\beta f}\,\mathrm{d}\nu_{\beta}=\sup_{x\in Y}(f(x)-I(x)).
	\end{equation}
\end{rem}

\subsection{No super-exponential suppresion}\label{ss: Dini quanta no freezing}
Here we prove Theorem~\ref{t: Dini quanta no freezing}~(i) assuming Theorems~\ref{t: existence} and ~\ref{t: zero-temperature limit}.

\begin{proof}[\bf Proof of Theorem~\ref{t: Dini quanta no freezing}~(i)]

    Fix $\phi\in H^\omega(X,\R)$ and an nonempty open subset $U$ of $X$. Note that there exists a sequence $\{\beta_n\}_{n\in\N}$ with $\beta_n\to+\infty$ as $n\to+\infty$ such that 
    \begin{equation*}
        \liminf_{\beta\to+\infty}\beta^{-1}\log(\mu_{\beta\phi}(U))=\lim_{n\to+\infty}\beta_n^{-1}\log(\mu_{\beta_n\phi}(U)).
    \end{equation*}

    By Theorems~\ref{t: zero-temperature limit}~(iii) and~\ref{t: existence}~(iv), there exists a subsequence $\{\beta'_n\}_{n\in\N}$ of $\{\beta_n\}_{n\in\N}$ such that $\bigl\{l_{\beta'_n}^\mu\bigr\}_{n\in\N}$  pointwise converges to a tropically continuous linear functional whose density $b'$ in $\density(X)$ is an invariant tropical Gibbs density for $\phi$. Thus, we see that $-b'\: X\to[0,+\infty]$ is lower semi-continuous. It then follows from Remark~\ref{r: two LDP}, (\ref{eq: defi of l_beta}), and Definition~\ref{d: ldp} that \begin{equation*}
        \lim_{n\to+\infty}\beta_n^{-1}\log (\mu_{\beta_n\phi}(U))=\lim_{n\to+\infty}(\beta'_{n})^{-1}\log(\mu_{\beta'_n\phi}(U))\geq \oplusu{x\in U}b'(x).
    \end{equation*}

    Let $M'>0$ be the constant in (\ref{ineq: defining inequalities for tropical Gibbs}) for the tropical Gibbs density $b'$ and $(\lambda,\gamma,\xi)$ be expanding constants of $T$. Note that by Lemma~\ref{l:expanding systems}, $T_{x}^{-n}(B(T^n(x),\xi))\subseteq B(x,\lambda^{-n}\xi)$ for all $x\in X$ and $n\in\N$. Since $U$ is nonempty and open, we see that for all $x\in U$, there exists $n_x\in\N$ such that $T_x^{-n_x}(B(T^{n_x}(x),\xi))\subseteq U$. Take $y\in U$. Since $b'$ is a tropical Gibbs density, we see that
    \begin{equation*}
        \oplusu{x\in U}b'(x)\geq \oplusu{x\in T_{y}^{-n_{y}}(B(T^{n_{y}}(y),\xi))}b'(x)\geq S_{n_{y}}\toverline{\phi}(y)-M'>-\infty.
    \end{equation*}
    We conclude that $\lim_{n\to+\infty}\beta_n^{-1}\log (\mu_{\beta_n\phi}(U))\geq \oplusu{x\in U}b'(x)>-\infty$, and thus $\phi$ does not admit super-exponential suppression. 

    Finally, we show that $\phi$ does not admit a freezing phase transition. If $\phi$ is cohomologous to a constant in $C(X,\R)$, then for each $t>0$, the unique equilibrium state for $\phi/t$ is the measure of maximal entropy, and thus it is independent of the temperature $t$. Now assume that $\phi$ is not comohologous to a constant in $C(X,\R)$, and suppose $\mu\in\cM(X,T)$ is a frozen state for $\phi$. By \cite[Theorem~4.1]{Je19}, $\mu$ is a $\phi$-maximizing measure. As $\phi$ is a Dini potential, by \cite[Proposition~2 and Theorem~1]{Bou01}, we see that $\phi$ satisfies the Walters condition and there exists $v\in C(X,\R)$ such that $\phi+v-v\circ T\leq\energy$. Since $\phi$ is not cohomologous to a constant in $C(X,\R)$, we see that $S\=(\phi+v-v\circ T)^{-1}(\energy)$ is a closed proper subset of $X$. Since $\energy$ is the ergodic supremum of $\phi$ (cf.~(\ref{eq: defi of ergodic supremum})), the maximizing measure $\mu$ is supported on $S$. Consequently, for the open set $X\smallsetminus S$, $\lim_{\beta\to+\infty}\frac1\beta\log(\mu_{\beta\phi}(X\smallsetminus S))=\lim_{\beta\to+\infty}\frac1\beta\log(\mu(X\smallsetminus S))=-\infty$, which contradicts (\ref{ineq: quantative non-freezing for dini}).
\end{proof}

\subsection{Existence of tropical eigenfunctions and eigendensities}\label{ss: existence}

In this subsection, we investigate the existence of tropical eigenfunctions and tropical eigendensities. To prove Theorem~\ref{t: existence}, we state the main results of this subsection in Proposition~\ref{p: main construction} and Proposition~\ref{p: uniqueness of tropical eigenvalue}, whose proofs are deferred
to Appendix~\ref{ss: Appendix A}.

\begin{prop}\label{p: main construction}
    Let $T\: X\to X$ be an open continuous distance-expanding map with expanding constants $(\lambda,\gamma,\xi)$, $\omega\in\sD$, $\phi\in H^{\omega}(X,\R)$, and $u\in C(X,\R)$. Define
    \begin{equation}\label{eq: defi of main construction}
        v_u(x)\=\limsup_{n\to+\infty}\cL_{\toverline{\phi}}^n(u)(x)
    \end{equation}
    for all $x\in X$. Then the following hold:
    \begin{enumerate}
        \smallskip
        \item There exists $L_2>0$ depending only on $T$ and $\omega$ such that $v_{u}\leq L_2\abs{\phi}_{\omega}+\norm{u}_{C^0}$ and for all $x,\, y\in X$ with $d(x,y)<\xi$,
        we have $v_{u}(x)\leq v_{u}(y)+\abs{\phi}_{\omega}\assocmodu(d(x,y))$.
        \smallskip
        \item $v_{u}\in C(X,\Rminus)$ is a tropical eigenfunction of $\cL_\phi$ associated with eigenvalue $\energy$.
        \smallskip
        \item Assume that $T$ is transitive. Then there exists a constant $C_3>0$ depending only on $T$ and $\omega$ such that $v_{u}\in H^{\assocmodu}(X,\R)$ with 
        \begin{equation*}
        \norm{v_{u}}_{C^0}\leq C_3\abs{\phi}_{\omega}\assocmodu(\diam X)+\norm{u}_{C^0}\quad \text{ and }\quad 
        \abs{v_{\tmif}}_{\assocmodu}\leq C_3\abs{\phi}_{\omega}.
        \end{equation*}
    \end{enumerate}
\end{prop}

\begin{prop}[Uniqueness of tropical eigenvalues]
	\label{p: uniqueness of tropical eigenvalue}
	Let $T\: X\to X$ be an open continuous distance-expanding map with expanding constants $(\lambda,\gamma,\xi)$, $\omega\in\sD$, and $\phi\in H^{\omega}(X,\R)$. Then the following hold:
	\begin{enumerate}[label=\rm{(\roman*)}]	
		\smallskip
        \item If there exists $u\in C(X,\R)$ and $\lambda\in\Rm$ such that $\cL_\phi(u)=\lambda\otimes u$, then $\lambda=\energy$.
        \smallskip
		\item If there exists $u\in C(X,\Rminus)\smallsetminus\{\taif\}$ and $\lambda\in\Rm$ such that $\cL_\phi(u)=\lambda\otimes u$, then $\lambda\leq\energy$, and if in addition $T$ is transitive, then $u\in C(X,\R)$ and $\lambda=\energy$.
		\smallskip
		\item If there exists $b\in \density(X)\smallsetminus\{\tinftyf,\,\taif\}$ and $\lambda\in\Rm$ such that $\cL_\phi^{\ostar}(b)=\lambda\otimes b$, then $\lambda\leq\energy$, and if in addition $T$ is transitive, then $\lambda=\energy$.
	\end{enumerate}
\end{prop}

\begin{proof}[\bf Proof of Theorem~\ref{t: existence}]
    (i) It directly follows from Proposition~\ref{p: uniqueness of tropical eigenvalue}~(ii)(iii).

    \smallskip
    (ii) It directly follows from Proposition~\ref{p: main construction}~(ii)(iii).

    \smallskip
    (iii) We consider $v\coloneqq-b$ and it suffices to find a solution of $v\circ T=v+\toverline{\phi}$ in $\twidehat{C}(X,\R)\smallsetminus\{\tinftyf,\,\taif\}$.
	
	Let $v'\=v_{\tmif}\in H^{\assocmodu}(X,\R)$ be the tropical eigenfunction of $\cL_{\phi}$ from Proposition~\ref{p: main construction}. By (\ref{eq: defi of the Bousch operator}), we see that 
	\begin{equation}\label{ineq: v_0 is a sub-action}
	v'+\toverline{\phi}\fless v'\circ T.
	\end{equation}
	Consider $w\coloneqq v-v'$ and $\varphi\coloneqq\toverline{\phi}+v'-v'\circ T\in C(X,\R)$. It suffices to find $w\in\twidehat{C}(X,\R)\smallsetminus\{\tinftyf,\,\taif\}$ such that \begin{equation}\label{eq: target equation}
    w\circ T=w+\varphi.
    \end{equation}
	
	Note that $\varphi\fless\tmif$ by (\ref{ineq: v_0 is a sub-action}) and
	\begin{equation*}
	Q(T,\varphi)
    =\max_{\mu\in\cM(X,T)}\int \! \varphi \, \mathrm{d} \mu
    =\max_{\mu\in\cM(X,T)}\int \! \toverline{\phi} \, \mathrm{d} \mu
    =Q\bigl(T,\toverline{\phi}\bigr)
    =0.
	\end{equation*}
    We fix a maximizing meaure $\mu\in \cM(X,T)$ for $\varphi$. 
	
	It follows from $C(X,\R)\ni\varphi\fless\tmif$ and $\int\!\varphi\,\mathrm{d}\mu=0$ that $\supp\mu\subseteq\varphi^{-1}(0)$. Moreover, since $\mu$ is $T$-invariant, $T(\supp\mu)\subseteq\supp\mu$.  We use these two properties to construct a solution $w=w_0$ of (\ref{eq: target equation}). 
	
	Consider $S\coloneqq\bigl\{w\in\twidehat{C}(X,\R):w+\varphi\fless w\circ T,\, w|_{\supp\mu}=0,\, \tmif\fless w\bigr\}$ and $w_0\coloneqq\oplusu{w\in S}w\in\twidehat{C}(X,\R)$ (since $\twidehat{C}(X,\R)$ is tropically complete). Since $\varphi\fless\tmif$, we see that $\tmif\in S$ and consequently $S$ is not empty and $w_0\neq\taif$. It is straightforward to check that $w_0\in S$ and it follows from $w_0|_{\supp\mu}=0$ that $w_0\neq\tinftyf$.
    
    To show that $w_0\circ T=w_0+\varphi$. Consider $w_1\coloneqq w_0\circ T-\varphi$. 
	
	 Since $w_0\in S$, we have $\tmif\fless w_0\fless w_1$. So $w_0\circ T\fless w_1\circ T$, i.e.,  $w_1+\varphi\fless w_1\circ T$. Recall that $\supp\mu\subseteq\varphi^{-1}(0)$ and $T(\supp\mu)\subseteq \supp\mu$. It follows that  $w_1|_{\supp\mu}=(w_0\circ T-\varphi)|_{\supp\mu}=0$. We conclude that $w_1\in S$. 
	 
	Since $w_0$ is the maximal element of $S$, we see that $w_1 \fless w_0$ and consequently $w_1=w_0$. We conclude that $w_0\circ T=w_0+\varphi$, and $w_0\in\twidehat{C}(X,\R)\smallsetminus\{\tinftyf,\,\taif\}$.

    \smallskip
    (iv) Assume that $v$ is a tropical eigenfunction of $\cL_\phi$ in $C(X,\R)$ and $b$ is a tropical eigendensity of $\cL_\phi^{\ostar}$, both associated with eigenvalue $\energy$ (cf.~statement (i)). It follows from (\ref{eq: defi of the Bousch operator}) and~(\ref{eq: defi of the adjoint Bousch operator}) that for all $x\in X$,
    \begin{equation*}
        \begin{aligned}
            v(x)\otimes b(x)&=\cL_{\toverline{\phi}}(v)(x)\otimes b(x)=\oplusu{y\in T^{-1}(x)}\bigl(v(y)\otimes \toverline{\phi}(y)\otimes b(T(y))\bigr)\\
            &=\oplusu{y\in T^{-1}(x)}\bigl(v(y)\otimes \cL_{\toverline{\phi}}^{\ostar}(b)(y)\bigr)=\oplusu{y\in T^{-1}(x)}(v(y)\otimes b(y)),
        \end{aligned}
    \end{equation*}
    i.e., $v\otimes b$ is $T$-invariant (cf.~Definition~\ref{d: T-invariance density}). Now it suffices to verify that $v\otimes b$ is a tropical Gibbs density (cf.~Definition~\ref{d: tropical gibbs density}).

    Fix $x\in X$ and $n\in\N$. Write $W \= T_x^{-n}(B(T^n(x),\xi))$. Since $b$ is a tropical eigendensity of $\cL_\phi^{\ostar}$ associated with eigenvalue $\energy$, it follows from (\ref{eq: iteration of adjoint Bousch}) that
    \begin{equation*}
            \oplusu{y\in W}(v(y)\otimes b(y))
            =\oplusu{y\in  W }\bigl(v(y)\otimes \bigl(\cL_{\toverline{\phi}}^{\ostar}\bigr)^n(b)(y)\bigr)
            =\oplusu{y\in  W }\bigl(v(y)\otimes S_n\toverline{\phi}(y)\otimes b(T^n(y)\bigr).
    \end{equation*}
    Note that by (\ref{ineq: difference of the nearby trajectory sum}), $\Absbig{S_n\toverline{\phi}(y)-S_n\toverline{\phi}(x)}\leq\abs{\phi}_{\omega}\assocmodu(\xi) $ for all $y\in T_x^{-n}(B(T^n(x),\xi)) =W$, and by Lemma~\ref{l:expanding systems}, $\oplusu{y\in  W }b(T^n(y))=\oplusu{z\in T^n(W)}b(z)$. We conclude that 
    \begin{equation}\label{ineq: bound for gibbs density}   
    -\abs{\phi}_\omega\assocmodu(\xi)-\norm{v}_{C^0}+\oplusu{z\in  T^n(W) }b(z)\leq\oplusu{y\in  W }(v(y)\otimes b(y))-S_n\toverline{\phi}(x)\leq \abs{\phi}_\omega\assocmodu(\xi)+\norm{v}_{C^0}+\oplusu{z\in  T^n(W) }b(z).
    \end{equation}
    Recall that $b\in\density(X)\smallsetminus\{\tinftyf,\,\taif\}$ since it is a tropical eigendensity. Thus, $\oplusu{z\in  T^n(W) }b(z)\leq\oplusu{z\in X}b(z)<+\infty$.

    \smallskip

    \emph{Claim.} There exists $m_\xi\in\R$ such that for all $y\in X$, $\oplusu{z\in B(y,\xi)}b(z)\geq m_\xi$.

    \smallskip

    If the claim holds, then it follows from (\ref{ineq: bound for gibbs density}) and $T^n(W)=B(T^n(x),\xi)$ that (\ref{ineq: defining inequalities for tropical Gibbs}) holds for $v\otimes b$ with $M\=\abs{\phi}_\omega\assocmodu(\xi)+\norm{v}_{C^0}+\abs{m_\xi}+\Absbig{\oplusu{z\in X}b(z)}$, and $v\otimes b$ is a tropical Gibbs density.

    \smallskip

    Now we prove the claim. Since $X$ is compact, there exists a finite set $\{x_1,\, x_2,\,\ldots,\, x_k\}$ such that for all $y\in X$, there exists $1\leq i\leq k$ so that $d(y,x_i)<\xi/2$, and consequently $B(x_i,\xi/2)\subseteq B(y,\xi)$. Now take $m_\xi\=\min\bigl\{\oplusu{z\in B(x_i,\xi/2)}b(z):1 \leq i\leq k\bigr\}\leq\oplusu{z\in X}b(z)<+\infty$, and we conclude that for all $y\in X$,
    $\oplusu{z\in B(y,\xi)}b(z)\geq m_\xi$.
    To see that $m_\xi\in \R$, it suffices to show that for all nonempty open subsets $U$ of $X$, $\oplusu{z\in U}b(z)>-\infty$.

    Finally we argue by contradiction. Suppose that $\oplusu{z\in U}b(z)=-\infty$ for some nonempty open subset $U$ of $X$. Recall that $\cL_\phi(b)=\energy\otimes b$, $\phi\in C(X,\R)$, and $\energy\in\R$. It follows that $\oplusu{z\in T^n(U)}b(z)=-\infty$ for all $n\in\N$. By the transitivity of $T$, $X=\bigcup_{n=1}^{+\infty}T^n(U)$. We conclude that $\oplusu{z\in X}b(z)=-\infty$, which contradicts with $b\in\density(X)\smallsetminus\{\tinftyf,\,\taif\}$. The proof is now complete.
\end{proof}

\subsection{Zero temperature limits}\label{ss: zero temperature limits}
Here we study logarithmic-type zero-temperature limits and prove Theorem~\ref{t: zero-temperature limit}. Indeed, the results in this subsection are all based on the normality of the family $\bigl\{\beta^{-1}\log u_{\beta\phi}\bigr\}_{\beta\in (1,+\infty)}$ (Lemma~\ref{l: normality and zero limit of eigenfunction}), where $u_{\beta\phi}$ is the unique eigenfunction of the Ruelle operator $\cR_{\beta\phi}$. We first prove the other results assuming Lemma~\ref{l: normality and zero limit of eigenfunction} and conclude this subsection with its proof.

\begin{lemma}\label{l: normality and zero limit of eigenfunction}
    Let $T\:X\rightarrow X$ be a mixing expanding covering map, $\omega\in\sD$, and $\phi\in H^\omega(X,\R)$. Then the family $\bigl\{\beta^{-1}\log u_{\beta\phi}\bigr\}_{\beta\in (1,+\infty)}$ is normal.
\end{lemma}

Recall that for all $f\in C(X,\R)\cup\{\tinftyf,\,\taif\}$ and $\beta>0$,
\begin{equation*} 
	l_{\beta}^{\mu}(f)=\frac{1}{\beta}\log\int \! e^{\beta f} \, \mathrm{d} \mu_{\beta \phi}\quad\text{ and }\quad
 	l_{\beta}^{m}(f)=\frac{1}{\beta}\log\int \! e^{\beta f} \, \mathrm{d} m_{\beta \phi}.
 \end{equation*}

\begin{rem}\label{r: l_beta at infinity}
Note that for all $\beta>0$, letting $l_{\beta}$ be either $l_{\beta}^m$ or $l_{\beta}^{\mu}$, we have
\begin{equation}\label{eq: l_beta at infinity}
l_{\beta}(\tinftyf)=+\infty\quad\text{ and }\quad l_{\beta}(\taif)=\tai=-\infty,
\end{equation}
 and thus, if $\{l_{\beta_k}|_{C(X,\R)}\}_{k\in\N}$ is pointwise convergent as $k\to+\infty$, then $\{l_{\beta_k}\}_{k\in\N}$ is also pointwise convergent on $C(X,\R)\cup\{\tinftyf,\,\taif\}$ as $k\to+\infty$.
\end{rem}
It is straightforward to check that for all $f,\,g\in C(X,\R)$ and $\beta>0$, letting $l_\beta$ be either $l_{\beta}^m$ or $l_{\beta}^{\mu}$, we have
\begin{equation}\label{ineq: regularity of functionals at positive temperature}
	\abs{l_{\beta}(f)} \leq \norm{f}_{C^0}\quad\text{ and }\quad
	\abs{l_{\beta}(f)-l_{\beta}(g)}\leq\norm{f-g}_{C^0}. 
\end{equation}

\begin{lemma}\label{l: zero tem limit of pressure}
    Let $T\:X\rightarrow X$ be a mixing expanding covering map, $\omega\in\sD$, and $\phi\in H^\omega(X,\R)$. Then
    \begin{equation*}
        \lim_{\beta\to+\infty}\beta^{-1}P(T,\beta\phi)=\energy.
    \end{equation*}
\end{lemma}
\begin{proof}
    We argue by contradiction. If $\frac{1}{\beta}P(T,\beta\phi)$ does not converge to $\energy$ as $\beta\to+\infty$, then there exists a sequence $\{\beta_n\}_{n\in\N}$ converging to $+\infty$ such that either $\liminf_{n\to+\infty}\beta_n^{-1}P(T,\beta_n\phi)>\energy$ or $\limsup_{n\to+\infty}\beta_n^{-1}P(T,\beta_n\phi)<\energy$.

     By Lemma~\ref{l: normality and zero limit of eigenfunction}, without loss of generality, we may assume that $\beta_n^{-1}\log u_{\beta_n\phi}$ uniformly converges to $v\in C(X,\R)$ and 
    $\lim_{n\to+\infty}\beta_n^{-1}P(T,\beta_n\phi)=c\in\Rm$.
    Note that by direct computation, we have $\oplusu{1\leq i\leq n}a_i\leq \frac1\beta\log \bigl(\sum_{i=1}^ne^{\beta a_i}\bigr)\leq \frac{\log n}{\beta}\otimes \bigl(\oplusu{1\leq i\leq n}a_i\bigr)$ for all $\beta>0$, $n\in\N$, and $a_i\in\R$, where $1\leq i\leq n$. It then follows from the definitions of the two operators ((\ref{eq: defi of the Ruelle operator}) and~(\ref{eq: defi of the Bousch operator})) that for all $u\in C(X,\R)$, 
    \begin{equation}\label{ineq: difference between the two operators}
        \Normbig{\cL_{\phi}(u)-\beta^{-1}\log (\cR_{\beta\phi}(\exp(\beta u)))}_{C^0}\leq \beta^{-1}\log N,
    \end{equation}
    where $N=\sup\bigl\{\card \bigl(T^{-1}(x)\bigr): x\in X\bigr\}<+\infty$ (cf.~Lemma~\ref{l:expanding systems}). Thus, by Lemma~\ref{l: projective estimate}, 
    \begin{equation*}
        \begin{aligned}
            \cL_\phi(v)
            &=\lim_{n\to+\infty}\cL_\phi\bigl(\beta_n^{-1}\log u_{\beta_n\phi}\bigr)
            =\lim_{n\to+\infty}\beta_n^{-1}\log(\cR_{\beta_n\phi}(u_{\beta_n\phi}))\\
            &=\lim_{n\to+\infty}\bigl(\beta_n^{-1}P(T,\beta_n\phi)+\beta_n^{-1}\log u_{\beta_n\phi}\bigr)
            =c\otimes v.
        \end{aligned}
    \end{equation*}
    Hence, by Proposition~\ref{p: uniqueness of tropical eigenvalue}~(i), $c=\energy$ and $v\in\es{\cL_\phi}{C(X,\R)}{\energy}$. This is a contradiction.     
    We conclude that $\lim_{\beta\to+\infty}\frac1\beta P(T,\beta\phi)=\energy$.
\end{proof}

Since $X$ is a compact metric space, we see that both $X$ and $C(X,\R)$ are separable. Thus, by the Arzel\`a--Ascoli theorem (cf.~\cite[Theorem~X.2.2]{Bour07}), an equicontinuous family of real-valued continuous functionals on $C(X,\R)$ (resp.~functions on $X$) that is uniformly bounded on every compact subset of $C(X,\R)$ (resp.~$X$) is a normal family. Thus, we only need to verify the equicontinuity and the uniform boundedness on compact subsets of $C(X,\R)$ (resp.~$X$) when showing the normality of a given family below. 

\begin{proof}[\bf Proof of Theorem~\ref{t: zero-temperature limit}]
    (i) The normality directly follows from Lemma~\ref{l: normality and zero limit of eigenfunction}. Now assume that $\beta_n^{-1}\log u_{\beta_n\phi}$ uniformly converges to $v\in C(X,\R)$ with $\beta_n\to+\infty$ as $n\to+\infty$. Following the argument of (\ref{ineq: difference between the two operators}) in the proof of Lemma~\ref{l: zero tem limit of pressure}, we have that for all $u\in C(X,\R)$, 
    \begin{equation*}
        \Normbig{\cL_{\phi}(u)-\beta^{-1}\log (\cR_{\beta\phi}(\exp(\beta u)))}_{C^0}\leq \beta^{-1}\log N ,
    \end{equation*}
    where $N=\sup\bigl\{\card \bigl(T^{-1}(x)\bigr): x\in X\bigr\}<+\infty$ (cf.~Lemma~\ref{l:expanding systems}). Thus, by Lemmas~\ref{l: projective estimate} and~\ref{l: zero tem limit of pressure}, 
    \begin{equation*}
        \begin{aligned}
            \cL_\phi(v)
            &=\lim_{n\to+\infty}\cL_\phi\bigl(\beta_n^{-1}\log u_{\beta_n\phi}\bigr)
            =\lim_{n\to+\infty}\beta_n^{-1}\log (\cR_{\beta_n\phi}(u_{\beta_n\phi}))\\
            &=\lim_{n\to+\infty}\bigl(\beta_n^{-1}P(T,\beta_n\phi)+\beta_n^{-1}\log u_{\beta_n\phi}\bigr)
            =\energy\otimes v,
        \end{aligned}
    \end{equation*}
    i.e., $v\in\es{\cL_\phi}{C(X,\R)}{\energy}$.

    \smallskip
    (ii) It follows from (\ref{ineq: regularity of functionals at positive temperature}) that $\bigl\{l_{\beta}^m|_{C(X,\R)}\bigr\}_{\beta\in(1,+\infty)}$ is uniformly bounded on every bounded subset of $C(X,\R)$, equicontinuous, and consequently normal. By (\ref{eq: l_beta at infinity}), if $\bigl\{l_{\beta_k}^m|_{C(X,\R)}\bigr\}_{k\in\N}$ is pointwise convergent as $k\to+\infty$, then so is $\bigl\{l_{\beta_k}^m\bigr\}_{k\in\N}$.

Now suppose $l\:C(X,\R)\cup\{\tinftyf,\,\taif\}\to\R$ is the pointwise limit of a convergent subsequence $\bigl\{l_{\beta_k}^m\bigr\}_{k\in\N}$ with $\beta_k\to+\infty$ as $k\to+\infty$. Adopting an argument that is verbatim the same as in the proof of \cite[Theorem~E~(i)]{LS26}, we see that $l$ is tropical linear.

Since every tropical linear functional on $C(X,\R)\cup\{\tinftyf,\,\taif\}$ is tropical continuous (cf.~\cite[Proposition~2.7]{LS26}), we see that $l$ is tropically continuous, and thus by Proposition~\ref{p: dual space equals completion}, has a unique density $b$ in $\density(X)$ . Note that for each $\beta>0$, $l_\beta^m(\tmif)=\beta^{-1}\log 1=0$ since $m_{\beta \phi}$ is a probability measure. It follows that $l(\tmif)=0$ and consequently $b\notin\{\tinftyf,\,\taif\}$.

Now we show that $b\in\density(X)\smallsetminus\{\tinftyf,\,\taif\}$ is in $\esbig{\cL_{\phi}^{\ostar}}{\density(X)}{\energy}$. Since $\energy\in\R$ (cf.~(\ref{eq: defi of ergodic supremum})), it follows that $\energy\otimes b\in\density(X)$. Note that $\cL_{\phi}^{\ostar}(b)\in\density(X)$ and that for every $b_1\:X\to\Rm$, there exists a unique $b_2\in\density(X)$ equivalent to $b_1$ (cf.~Remark~\ref{r: equivalence of densities}). Thus, it suffices to show that $\cL_{\phi}^{\ostar}(b)$ is equivalent to $\energy\otimes b$, i.e., $l(\cL_\phi(f))=l(f)\otimes\energy$ for all $f\in C(X,\R)$ (cf.~Remark~\ref{r: equivalence of densities}).

Fix $f\in C(X,\R)$. It follows from (\ref{ineq: regularity of functionals at positive temperature}) and the definitions of the two operators ((\ref{eq: defi of the Ruelle operator}) and~(\ref{eq: defi of the Bousch operator})) that for all $k\in\N$,
\begin{equation}\label{ineq:difference between Ruelle and Bousch operator}
 \Absbig{l_{\beta_k}^m(\cL_{\phi}(f))-l_{\beta_k}^m\bigl(\beta_k^{-1}\log \bigl(\cR_{\beta_k \phi} \bigl( e^{\beta_k f}\bigr)\bigr)\bigr)}
 \leq\Normbig{\cL_{\phi}(f)-\beta_k^{-1}\log\bigl(\cR_{\beta_k \phi}\bigl(e^{\beta_k f}\bigr)\bigr)}_{C^0}
 \leq\beta_k^{-1}\log N,
\end{equation}
where $N$ is the constant in Lemma~\ref{l:expanding systems}. Note that
\begin{align}
l_{\beta_k}^m\bigl(\beta_k^{-1}\log\bigl(\cR_{\beta_k\phi}\bigl(e^{\beta_k f} \bigr)\bigr)\bigr)
&=\frac{1}{\beta_k}\log\int \! \cR_{\beta_k \phi} \bigl(e^{\beta_k f} \bigr) \, \mathrm{d} m_{\beta_k \phi}
=\frac{1}{\beta_k}\log\int \! e^{P(T,\beta_k\phi)}\cdot e^{\beta_k f} \, \mathrm{d} m_{\beta_k\phi} \notag\\
&=\frac{P(T,\beta_k \phi)}{\beta_k}\otimes\frac{1}{\beta_k}\log\int \! e^{\beta_k f} \, \mathrm{d} m_{\beta_k\phi}
=\frac{P(T,\beta_k\phi)}{\beta_k}\otimes l_{\beta_k}^m(f),\label{eq:eigenmeasure_functional at positive temperature}
\end{align}
where the second equality holds since $m_{\beta\phi}$ is the eigenmeasure of  $\cR_{\beta\phi}^*$ with eigenvalue $e^{P(T, \beta\phi)}$. Recall $\lim\limits_{k\to+\infty}\beta_k=+\infty$,
$\lim\limits_{k\to+\infty}l_{\beta_k}^m(f)=l(f)$, $\lim\limits_{k\to+\infty}l_{\beta_k}^m(\cL_\phi(f))=l(\cL_\phi(f))$, and $\lim\limits_{\beta\to+\infty}\frac1\beta P(T,\beta\phi)=\energy$ (by Lemma~\ref{l: zero tem limit of pressure}). Combining (\ref{ineq:difference between Ruelle and Bousch operator}) and~(\ref{eq:eigenmeasure_functional at positive temperature}) and letting $k\to+\infty$, we conclude that $l(\cL_\phi(f))=l(f)\otimes\energy$, and~(ii) is now verified.

\smallskip
\smallskip

(iii) It follows from (\ref{ineq: regularity of functionals at positive temperature}) that $\bigl\{ l_{\beta}^\mu|_{C(X,\R)} \bigr\}_{\beta\in(1,+\infty)}$ is normal. By (\ref{eq: l_beta at infinity}), if $\bigl\{l_{\beta_k}^\mu|_{C(X,\R)}\bigr\}_{k\in\N}$ is pointwise convergent as $k\to+\infty$, then so is $\bigl\{l_{\beta_k}^\mu\bigr\}_{k\in\N}$.

Now suppose $\widehat{l}$ is the pointwise limit of a convergent subsequence $\bigl\{l_{\beta_k}^\mu\bigr\}_{k\in\N}$ with $\beta_k\to+\infty$ as $k\to+\infty$. By Lemma~\ref{l: normality and zero limit of eigenfunction}, (ii), and Remark~\ref{r: l_beta at infinity}, we take a subsequence $\{\beta'_k\}_{k\in\N}$ of $\{\beta_k\}_{k\in\N}$ such that as $k\to+\infty$, $\bigl\{{\beta'_k}^{-1}\log u_{\beta'_k \phi}\bigr\}_{k\in\N}$ uniformly converges to $v$ in $\es{\cL_\phi}{C(X,\R)}{\energy}$ and $\bigl\{l_{\beta'_k}^{m}\bigr\}_{k\in\N}$ pointwise converges to $l$. Adopting an argument that is verbatim the same as in the proof of \cite[Theorem~E~(ii)]{LS26}, we see that $\widehat{l}(f)=l(v\otimes f)$ for all $f\in C(X,\R)\cup\{\tinftyf,\,\taif\}$.

%Note that for all $\beta>0$ and $f\in C(X,\R)$, $l_{\beta}^{\mu}(f)=l_{\beta}^m \bigl( f+\beta^{-1}\log u_{\beta \phi} \bigr)$ since $\mu_{\beta \phi}=u_{\beta \phi}\cdot m_{\beta \phi}$. Thus, for all $k\in\N$ and $f\in C(X,\R)$,
%\begin{equation*}
	%\begin{aligned}
		%\Absbig{l(v\otimes f)-l_{\beta'_k}^{\mu}(f)}
		%&=\Absbig{l(v+f)-l_{\beta'_k}^m \bigl( f+{\beta'_k}^{-1}\log u_{\beta'_k \phi} \bigr) }\\&\leq\Absbig{l(v+f)-l_{\beta'_k}^m(v+f)}+\Absbig{l_{\beta'_k}^m(v+f)-l_{\beta'_k}^m \bigl( f+{\beta'_k}^{-1} \log u_{\beta'_k \phi} \bigr) }\\&\leq\Absbig{l(v+f)-l_{\beta'_k}^m(v+f)}+\Normbig{v-{\beta'_k}^{-1}\log u_{\beta'_k \phi}}_{C^0},
	%\end{aligned}
%\end{equation*}
%where the last inequality follows from (\ref{ineq: regularity of functionals at positive temperature}). By the choice of $\{\beta'_k\}_{k\in\N}$, it follows that $l(v\otimes f)=\widehat{l}(f)$ for all $f\in C(X,\R)$. Since $v\in C(X,\R)$, it follows from (\ref{eq: l_beta at infinity}) that $l(v\otimes f)=\widehat{l}(f)$ for all $f\in C(X,\R)\cup\{\tinftyf,\,\taif\}$. 

By (ii), $l$ is tropically continuous linear and its density $b$ in $\density(X)$ is in $\esbig{\cL_\phi^{\ostar}}{\density(X)}{\energy}\smallsetminus\{\tinftyf,\,\taif\}$. By Definition~\ref{d: defi of lu}, Proposition~\ref{p: dual space equals completion}, and Remark~\ref{r: equivalence of densities}, we see that $\widehat{l}(f)=\oplusu{x\in X}(f(x)\otimes v(x)\otimes b(x))$ for all $f\in C(X,\R )\cup\{\tinftyf,\,\taif\}$, i.e., $\widehat{l}$ is a tropically continuous linear functional and $v\otimes b$ is its density in $\density(X)$. Since $v\in C(X,\R)$ and $b\notin\{\tinftyf,\,\taif\}$, we see that $v\otimes b\notin\{\tinftyf,\,\taif\}$. Now (iii) is verified.
\end{proof}

Finally, we check Lemma~\ref{l: normality and zero limit of eigenfunction}. Let $(\lambda,\,\gamma,\,\xi)$ be the expanding constants of $T$ (cf.~Remark~\ref{r: expanding parameter and long pull backs}). Denote for each modulus of continuity $\omega$ and each constant $K>0$, 
\begin{equation*}
    H^\omega_{K}(X,\R)\coloneqq\bigl\{\phi\in H^\omega(X,\R):\phi>0,\,\phi(x)/\phi(y)\leq e^{K\omega(d(x,y))}\text{ for all }x,\,y\in X\text{ with }d(x,y)\leq \xi\bigr\}.
\end{equation*}

\begin{proof}[Proof of Lemma~\ref{l: normality and zero limit of eigenfunction}]
    Let $(\lambda,\gamma,\xi)$ be expanding constants of $T$, where $\lambda>1$, $\gamma>0$, and $\xi>0$. 
    
    We first check the equicontinuity. Note that if $\exp\phi\in H^{\omega}_{K_0}(X,\R)$ for some $K_0>0$, then $\exp(\beta\phi)\in H^{\omega}_{\beta K_0}(X,\R)$ for all $\beta>0$, and by \cite[Theorem~1]{FJ01a}, $u_{\beta\phi}\in H^{\assocmodu}_{\beta K_0}(X,\R)$ for all $\beta>0$. 
    Note that $\phi\in H^\omega(X,\R)$ implies that $\exp\phi\in H^\omega_{\abs{\phi}_\omega}(X,\R)$. We conclude that for all $\beta>0$, $u_{\beta\phi}\in H^{\assocmodu}_{\beta\abs{\phi}_\omega}(X,\R)$, i.e., $\abs{\log u_{\beta \phi}}_{\assocmodu,\xi}\leq \beta \abs{\phi}_{\omega}$.

   Now it suffices to verify the (uniform) boundedness. Since $X$ is compact and $T$ is transitive, there exists $N_{\xi}\in\N$ such that for all $x,\,y\in X$, there exists an integer $m$ satisfying $0\leq m\leq N_{\xi}$ and $T^m(B(x,\xi))\cap B(y,\xi)\neq\emptyset$. Recall that $N=\sup\bigl\{\card \bigl(T^{-1}(x)\bigr):x\in X\bigr\}<+\infty$ (cf.~Lemma~\ref{l:expanding systems}). Next, adopting an argument that is verbatim the same as in the proof of ~\cite[Proposition~4.4.3]{PU10} with $\psi\in H^\omega(X,\R)$ in place of $\phi$, $\exp(\abs{\psi}_\omega\assocmodu(\xi))$ in place of $C_1$, $N$ in place of $D$, and $N_\xi$ in place of $N$, we get the conclusion that for all $x,\,y\in X$, $\psi\in H^\omega(X,\R)$, and $n\in\N$,
   \begin{equation*}
       \frac{\sum_{z\in T^{-n}(x)}\exp(S_n\psi(z))}{\sum_{z\in T^{-n}(y)}\exp(S_n\psi(z))}\leq \exp(2\abs{\psi}_\omega\assocmodu(\xi))\exp(2N_\xi\norm{\psi}_{C^0}) N^{N_\xi}.
   \end{equation*}
   Note that for all $\beta>0$, $\beta\phi\in H^\omega(X,\R)$. Hence, there exist constants $K_1,\,K_2>0$ depending only on $T$, $\phi$, and $\omega$ that for all $\beta>0$,
   $\frac{\cR_{\beta\phi}^n(\mathbbold{1})(x)}{\cR_{\beta\phi}^n(\mathbbold{1})(y)}\leq \exp(\beta K_1+K_2)$.
   It then follows from \cite[Theorem~1~(4)]{FJ01a} that for all $\beta>0$,
   $u_{\beta\phi}(x)\leq \exp(\beta K_1+K_2)u_{\beta\phi}(y)$.
   Since $\int\! u_{\beta\phi}\,\mathrm{d}m_{\beta\phi}=1$, we see that $\exp(-\beta K_1-K_2)\leq u_{\beta\phi}\leq \exp(\beta K_1+K_2)$, and consequently for all $\beta>1$,
   $\Normbig{\frac{1}{\beta}\log u_{\beta\phi}}_{C^0}\leq K_1+K_2$.
\end{proof}

\subsection{Representation and uniqueness of eigenfunction and eigendensity}\label{ss: representation and uniqueness}

In this subsection, we establish the representations of tropical eigenobjects (Proposition~\ref{p: representation}) and generic zero-temperature large deviation principles of equilibrium states (Proposition~\ref{p: uniquely maximizing implies the large deviation principle}) for Dini potentials, proving Theorem~\ref{t: Dini quanta no freezing}~(ii).

\begin{prop}[Representation of tropical eigenfunctions and eigendensities]\label{p: representation}
	Let $T\:X\rightarrow X$ be a transitive expanding covering map on a compact metric space $(X,d)$, $\omega\in\sD$, and $\phi\in H^\omega(X,\R)$. Let $S_\phi(\,\cdot\,,\,\cdot\,) \: X\times X\rightarrow\R$ be the Ma\~{n}\'{e} potential associated with $\phi$, and $\Omega_\phi$ be the Aubry set with respect to $\phi$. Then the following hold:
\begin{enumerate}[label=\rm{(\roman*)}]	
	\smallskip
	\item Suppose $v$ is in $\es{\cL_\phi}{C(X,\R)}{\energy}$. Then $v(y)=\oplusu{x\in\Omega_\phi}(v(x)\otimes S_\phi(x,y))$ for every $y$ in $X$.
    
	\smallskip
	\item Suppose $b$ is in $\esbig{\cL_\phi^{\ostar}}{\density(X)}{\energy}$. Then $b(\,\cdot\,)$ is equivalent to $\oplusu{y\in\Omega_\phi}(S_\phi(\,\cdot\,,y)\otimes b(y))$. Moreover, if $b\neq\tinftyf$, then $b(x)=\oplusu{y\in\Omega_\phi}(S_\phi(x,y)\otimes b(y))$ for all $x\in X$.
    
	\smallskip
	\item The entries of $\{S_\phi(\,\cdot\,,y)\}_{y\in\Omega_\phi}$ are tropical eigendensities of $\cL_\phi^{\ostar}$ associated with eigenvalue $\energy$, and the entries of $\{S_\phi(x,\,\cdot\,)\}_{x\in\Omega_\phi}$ are tropical eigenfunctions of $\cL_\phi$ in $C(X,\R)$ associated with eigenvalue $\energy$. If $\phi$ has a unique maximizing measure, then both $\es{\cL_\phi}{C(X,\R)}{\energy}$ and $\esbig{\cL_\phi^{\ostar}}{\density(X)}{\energy}$ are of tropical dimension one. 
\end{enumerate}
\end{prop}

The definitions of the Ma\'{n}\~{e} potential and the Aubry set are reviewed in Definitions~\ref{d: Mane potential} and~\ref{d: Aubry set}, respectively. The notion of \emph{tropical dimension one} is given in Subsection~\ref{ss: backgrounds for TTF}.

\begin{prop}[Generic zero-temperature large deviation principle]\label{p: uniquely maximizing implies the large deviation principle}
	Let $T\:X\to X$ be a mixing expanding map on a compact metric space $(X,d)$ and $\omega\in\sD$. Then for each $\phi\in H^{\omega}(X,\R)$ that has a unique maximizing measure, the family of equilibrium states $\{\mu_{\beta \phi}\}_{\beta\in(1,+\infty)}$ satisfies the large deviation principle as $\beta\to +\infty$ with rate function $-(v_0\otimes b_0)$, for some $v_0\in \es{\cL_\phi}{C(X,\R)}{\energy}$ and $b_0\in \esbig{\cL_\phi^{\ostar}}{\density(X)}{\energy}$. Moreover, if $d_{\omega}(\,\cdot\,,\,\cdot\,)\=\omega(d(\,\cdot\,,\,\cdot\,))$ is a metric on $X$, then the above property holds for a generic subset of $H^{\omega}(X,\R)$.
\end{prop}

Here a \emph{generic} subset means a countable intersection of open and dense subsets. The notion of zero-temperature large deviation principle is reviewed in Subsection~\ref{ss: backgrounds for TTF}. 

To prove these two propositions, we state some auxiliary results from Proposition~\ref{p: Mane_potential_properties} to Proposition~\ref{p: condition_for_uniqueness}, whose proofs are deferred to Appendix~\ref{ss: Appendix A}.

\begin{definition}[Aubry set]\label{d: Aubry set}
Let $T\:X \rightarrow  X$ be an open continuous distance-expanding map. For a continuous potential $\phi\in C(X,\R)$, we call $x\in X$ an \emph{Aubry point} if for every $\epsilon>0$, there exists $y\in X$ and $n\in\N$ such that \begin{equation*}
	d(x,y)\leq\epsilon,\quad d(T^n(y),x)\leq\epsilon,\quad\text{ and } \quad\Absbig{S_n\toverline{\phi}(y)}\leq\epsilon.
\end{equation*}
The collection of all Aubry points in $X$ is called the \emph{Aubry set} and denoted by $\Omega_\phi$. 
\end{definition}

\begin{definition}[Ma\~{n}\'{e} potential]\label{d: Mane potential}
Let $T\:X \rightarrow  X$ be an open continuous distance-expanding map and $\omega\in\sD$. For a potential $\phi\in H^\omega(X,\R)$, the \emph{Ma\~{n}\'{e} potential} associated with $\phi$ is the function $S_\phi$ defined on $X\times X$ given by 
\begin{equation*}
S_\phi(x,y)\coloneqq\lim_{\epsilon\to0^+}\oplusu{n\in\N}\oplusu{\substack{d(z,x)\leq\epsilon\\ d(T^n(z),y)\leq\epsilon}}S_n\toverline{\phi}(z).
\end{equation*}
\end{definition}

\begin{prop}[Properties of the Ma\~{n}\'{e} potential] \label{p: Mane_potential_properties}
Let $T\:X \rightarrow  X$ be an open continuous distance-expanding map, $\omega\in\sD$, and $\phi\in H^\omega(X,\R)$. Then the following hold for all $x,\,y,\,z\in X$:
\begin{enumerate}[label=\rm{(\roman*)}]	
\smallskip
\item $u(x)\otimes S_\phi(x,y)\leq u(y)$ for all $u\in \es{\cL_\phi}{C(X,\Rminus)}{\energy}$.

\smallskip
\item $S_\phi(x,y)\otimes b(y)\leq b(x)$ for all $b\in\esbig{\cL_\phi^{\ostar}}{\density(X)}{\energy}$.

\smallskip
\item $S_\phi(x,z)\geq S_\phi(x,y)\otimes S_\phi(y,z)$.
\smallskip
\item $S_\phi(x,x)\leq 0$ and $S_\phi(x,x)=0$ if and only if $x\in\Omega_\phi$.
\smallskip
\item Assume, in addition, that $x\in\Omega_\phi$. Then $S_\phi(x,\,\cdot\,)$ is in $\es{\cL_\phi}{C(X,\Rminus)}{\energy}$, and 
$
    S_\phi(x,z)\leq S_\phi(x,w)+\abs{\phi}_{\omega}\assocmodu(d(z,w))
$
for all $z,\,w\in X$ satisfying $d(z,w)<\xi$, where $\xi>0$ is the constant in Lemma~\ref{l:expanding systems}. Moreover, if $T$ is transitive, then $S_\phi(x,\,\cdot\,)\in H^{\assocmodu}(X,\R)$.
\end{enumerate}
\end{prop}

\begin{lemma} [Ma\~n\'e potential as eigendensity] \label{l: mane potential gives eigenmeasures}
Let $T\:X \rightarrow  X$ be an open continuous distance-expanding map, $\omega\in\sD$, and $\phi\in H^\omega(X,\R)$. The Ma\~{n}\'{e} potential satisfies \begin{equation}\label{eq: Mane potential (_,y) almost eigendensity}
	S_\phi(T(x),y)=S_\phi(x,y)-\toverline{\phi}(x)
\end{equation}
 for all $x,\,y\in X$ with $T(x)\neq y$. Moreover, for every $y\in\Omega_\phi$, (\ref{eq: Mane potential (_,y) almost eigendensity}) holds for every $x\in X$, i.e., $S_\phi(\,\cdot\,,y)$ is in $\esbig{\cL_\phi^{\ostar}}{\density(X)}{\energy}$.
\end{lemma}

\begin{prop} [Sufficient condition for uniqueness]\label{p: condition_for_uniqueness}
Let $T\:X \rightarrow  X$ be a transitive expanding covering map, $\omega\in\sD$, and $\phi\in H^\omega(X,\R)$. Assume that $\phi$ has a unique maximizing measure. Then 
\begin{equation*} 
   \tdim\es{\cL_\phi}{C(X,\Rminus)}{\energy}
   =1
   = \tdim\esbig{\cL_\phi^{\ostar}}{\density(X)}{\energy}.
\end{equation*}
\end{prop}

We conclude this subsection with the following proofs of Propositions~\ref{p: representation} and~\ref{p: uniquely maximizing implies the large deviation principle}, leading to the proof of Theorem~\ref{t: Dini quanta no freezing}~(ii).

\begin{proof}[Proof of Proposition~\ref{p: representation}]
    (i) \& (ii) We adopt the proof from that in the case of H\"{o}lder potentials and defer it to Appendix. 

    \smallskip
    (iii) It follows from Proposition~\ref{p: Mane_potential_properties}~(v) and Lemma~\ref{l: mane potential gives eigenmeasures} that for all $x\in\Omega_\phi$, $S_\phi(x,\,\cdot\,)\in\es{\cL_\phi}{C(X,\R)}{\energy}$ and $S_\phi(\,\cdot\,,x)\in\esbig{\cL_\phi^{\ostar}}{\density(X)}{\energy}$. It follows from Proposition~\ref{p: Mane_potential_properties}~(iv) that $S_\phi(x,x)=0$ for all $x\in\Omega_\phi$. We conclude that for all $x\in\Omega_\phi$, $S_\phi(x,\,\cdot\,)\neq\taif$ and $S_\phi(\,\cdot\,,x)\notin\{\tinftyf,\,\taif\}$.
    Now the first part of (iii) is verified.
    
    For the remaining part, we observe that $\tdim\es{\cL_\phi}{C(X,\Rminus)}{\energy}
   =1
   = \tdim\esbig{\cL_\phi^{\ostar}}{\density(X)}{\energy}$ follows from Proposition~\ref{p: condition_for_uniqueness}. Now we verify $\tdim\es{\cL_\phi}{C(X,\R)}{\energy}=1$.
   
   Since $\es{\cL_\phi}{C(X,\Rminus)}{\energy}$ is of tropical dimension one, there exists $u\in\es{\cL_\phi}{C(X,\Rminus)}{\energy}$ such that $\es{\cL_\phi}{C(X,\Rminus)}{\energy}\smallsetminus\{\taif,\,\tinftyf,\,+\infty\otimes u\}=\{a\otimes u:a\in\R\}$. Now let $v\coloneqq v_{\tmif}$ be the function constructed in Proposition~\ref{p: main construction}. By Proposition~\ref{p: main construction}~(ii)(iii), $v\in\es{\cL_\phi}{C(X,\R)}{\energy}\subseteq\es{\cL_\phi}{C(X,\Rminus)}{\energy}\smallsetminus\{\taif,\,\tinftyf,\,+\infty\otimes u\}=\{a\otimes u: a\in\R\}$. It follows that there exists $a_0\in\R$ such that $a_0\otimes u=v$, and consequently, $\es{\cL_\phi}{C(X,\R)}{\energy}\subseteq\{a\otimes v: a\in\R\}$. Since $v\in\es{\cL_\phi}{C(X,\R)}{\energy}$, it follows from the tropical linearity of $\cL_\phi$ that $a\otimes v\in\es{\cL_\phi}{C(X,\R)}{\energy}$ for all $a\in \R$. We conclude that $\es{\cL_\phi}{C(X,\R)}{\energy}=\{a\otimes v:a\in\R\}$ and $\tdim\es{\cL_\phi}{C(X,\R)}{\energy}=1$.
\end{proof}

\begin{proof}[Proof of Proposition~\ref{p: uniquely maximizing implies the large deviation principle}]
    Note that for all $b\in\density(X)\smallsetminus\{\tinftyf,\,\taif\}$, $\oplusu{x\in X}b(x)\in\R$. Let $\phi$ be an $\omega$-continuous function with a unique maximizing measure. It follows from Proposition~\ref{p: condition_for_uniqueness} that $\tdim\esbig{\cL_\phi^{\ostar}}{\density(X)}{\energy}=1$. Thus, there exists a unique $b_0\in\esbig{\cL_\phi^{\ostar}}{\density(X)}{\energy}$ such that $\oplusu{x\in X}b_0(x)=0$. Note that for all $\beta>0$, $l_{\beta}^m(\tmif)= \beta^{-1}\log 1=0$. It then follows from Theorem~\ref{t: zero-temperature limit}~(ii) that every pointwise convergent subsequence $\bigl\{l_{\beta_n}^m\bigr\}_{n\in\N}$ of the family $\bigl\{l_{\beta}^m\bigr\}_{\beta\in(1,+\infty)}$ with $\beta_n\to+\infty$ as $n\to+\infty$ must converge to some tropical linear functional $l$ whose density $b$ in $\density(X)$ is in $\esbig{\cL_\phi^{\ostar}}{\density(X)}{\energy}$ and $\oplusu{x\in X}b(x)=l(\tmif)=0$, which implies that $b=b_0$. Recall from Theorem~\ref{t: zero-temperature limit}~(ii) that the family $\bigl\{l_\beta^m|_{C(X,\R)}\bigr\}_{\beta\in(1,+\infty)}$ is normal. By Remark~\ref{r: l_beta at infinity}, we conclude that as $\beta\to+\infty$, $l_\beta^m$ must pointwise converge to $l$ whose density in $\density(X)$ is $b_0$. 
    
    Since $\phi$ has a unique maximizing measure, it follows from Proposition~\ref{p: representation}~(iii) that $\tdim\es{\cL_\phi}{C(X,\R)}{\energy}=1$. Thus, there exists a unique $v_0\in\es{\cL_\phi}{C(X,\R)}{\energy}$ such that $\oplusu{x\in X}(v_0(x)\otimes b_0(x))=l(v_0)=0$. Now adopting an argument that is verbatim the same as in the proof of \cite[Corollary~4.3]{LS26}, we see that $\lim_{\beta\to+\infty}\frac1\beta\log u_{\beta\phi}=v_0$ in $C(X,\R)$ and $\lim_{\beta\to+\infty}l_{\beta}^{\mu}(f)=l(v_0+f)$ for all $f\in C(X,\R)$.

 Since $b_0$ is the density of $l$ in $\density(X)$, by Remark~\ref{r: equivalence of densities}, we see that $\lim_{\beta\to+\infty}l_{\beta}^{\mu}(f)=\oplusu{x\in X}(f(x)\otimes(v_0(x)\otimes b_0(x)))$
for all $f\in C(X,\R)$. It then follows from (\ref{eq: defi of l_beta}) and Remark~\ref{r: two LDP} that $\{\mu_{\beta \phi}\}_{\beta\in(1,+\infty)}$ satisfies the large deviation principle as $\beta\to+\infty$ with rate function $-(v_0\otimes b_0)$.

Now further assume that $d_{\omega}(\,\cdot\,,\,\cdot\,)\=\omega(d(\,\cdot\,,\,\cdot\,))$ is a metric on $X$. By the Stone--Weistrass theorem, we see that $H^{\omega}(X,\R)$ is a dense subset of $C(X,\R)$. By \cite[Theorem~3.2]{Je06}, we see that for a generic $\phi\in H^\omega(X,\R)$, $\phi$ has a unique maximizing measure. The proof is now complete.
\end{proof}

\begin{proof}[\bf Proof of Theorem~\ref{t: Dini quanta no freezing}~(ii)]
     It follows from Proposition~\ref{p: uniquely maximizing implies the large deviation principle} that, provided that $\phi$ has a unique maximizing measure, $\{\mu_{\beta\phi}\}_{\beta\in (0,+\infty)}$ satisfies the large deviation principle with rate function $-(v_0\otimes b_0)$ for some $v_0\in \es{\cL_\phi}{C(X,\R)}{\energy}$ and $b_0\in \esbig{\cL_\phi^{\ostar}}{\density(X)}{\energy}$. By Theorem~\ref{t: existence}~(iv), we see that $v_0\otimes b_0$ above is an invariant tropical Gibbs density. If $d_{\omega}(\,\cdot\,,\,\cdot\,)\=\omega(d(\,\cdot\,,\,\cdot\,))$ is a metric on $X$, it follows from Proposition~\ref{p: uniquely maximizing implies the large deviation principle} that the above property holds for a generic subset of $H^{\omega}(X,\R)$. 
\end{proof}

\begin{rem}\label{r: generic for symbolic}
    Here we explain the proof of the first theorem on page 1. Note that statements~(i) and~(iii) follow from Theorem~\ref{t: Dini quanta no freezing}~(i) and Theorem~\ref{t: non TPO}, respectively. When $\omega\in\cD$ vanishes only at $0$, it follows from \cite[Theorem~9]{PT14} that $d_{\omega}(\,\cdot\,,\,\cdot\,)\=\omega(d_2(\,\cdot\,,\,\cdot\,))$ is a metric on $\Sigma$ as $d_2$ is an ultrametric on $\Sigma$, and thus statement~(ii) directly follows from Theorem~\ref{t: Dini quanta no freezing}~(ii). Now assume that $\omega\in\cD$ and $\omega(t)=0$ for some $t\in (0,+\infty)$. As $\omega(\diam_{d_2}\Sigma)>0$, we see that $t\in (0,\diam_{d_2}\Sigma)$ and $H^\omega(\Sigma,\R)$ consists of locally constant potentials depending on first $n$ symbols for some $n\in\N$. By \cite[Theorem~A~(a)]{WY19}, there exists an open and dense subset $U$ of $H^{\omega}(\Sigma,\R)$ such that every $\phi\in U$ has a unique maximizing measure that is supported on a periodic orbit. Thus, statement~(ii) follows from Theorem~\ref{t: Dini quanta no freezing}~(ii). 
\end{rem}

\section{Failure of Typical Periodic Optimization}\label{s: failure of TPO}

In this section, we use tools from optimal transport and investigate fine properties of the Fibonacci subshift to prove Theorems~\ref{t: energy transport estimate} and~\ref{t: non TPO}. Throughout this section, we let $(\Sigma,\sigma)$ be the full shift on $\Sigma\=\{0,\,1\}^\N$ equipped with the metric $d\=d_2$, and denote by $\mu_K$ be the unique measure in $\cM(\Sigma,\sigma)$ supported on the Fibonacci subshift $K$. We establish the following result, which in particular implies Theorem~\ref{t: energy transport estimate} (cf.~Remark~\ref{r: range of the technical conditions}):

\setcounter{thml}{7}
\begin{thml}\label{t: energy transport estimate: general condition}
    %Let $(\Sigma,\sigma)$ be the full shift on $\Sigma=\{0,\,1\}^\N$ equipped with the metric $d_2$. 
    Suppose a modulus of continuity $\omega$ satisfies the scale doubling condition and the stable condition. Then there exists a constant $C>0$ such that for each $m\in\cM(\Sigma,\sigma)$ and each $\omega$-continuous function $g\: \Sigma\to\R$ with $\norm{g}_{\omega}\leq 1$,
    \begin{equation*}
        \int_{\Sigma}\! g\,\mathrm{d}\mu_K-\int_{\Sigma}\! g\,\mathrm{d}m\leq C\int_{\Sigma}\!\omega(d(x,K))\,\mathrm{d}m(x).
    \end{equation*} 
\end{thml}

The precise definitions of the two technical conditions are given in Definition~\ref{d: two technical conditions}. 

To prove Theorem~\ref{t: energy transport estimate: general condition}, we first transform the above problem into controlling the differences between some periodic measure $\nu_p$ and $\mu_K$ on cylinder sets using tools from optimal transport in Subsection~\ref{ss: Wasserstein}. We then construct a measure $\eta_z$ supported on $K$, which is close to $\nu_p$ in Subsection~\ref{ss: accident decomposition}, and use flows on the Rauzy graphs to estimate the cylinder discrepancies between $\eta_z$ and $\mu_K$ (cf.~Definition~\ref{d: D_n}) in Subsection~\ref{ss: comparing flows on Rauzy graphs}. The proof of Theorem~\ref{t: non TPO} is also presented in Subsection~\ref{ss: comparing flows on Rauzy graphs}. Basic facts of the Fibonacci subshift and its Rauzy graphs are reviewed in Subsection~\ref{ss: Fibonacci subshift and Rauzy graph}.

\subsection{Wasserstein reduction to cylinder discrepancies}\label{ss: Wasserstein}

A central object in optimal transport is the Wasserstein distance, which dates back to the works of Gini, Kantorovich, Wasserstein, Mallows, and Tanaka~\cite{Gi14,Ka42,Ma72,Ta73,Va69}. For more backgrounds on optimal transport, see Villani \cite{Vi03,Vi09}. While the $1$-Wasserstein distance, in the form of the supremum over Lipschitz test functions (cf.~(\ref{eq: W_omega in terms of H^omega functions})), naturally appears in the study of the Typically Periodic Optimization Conjecture as a useful metric on the space of Borel probability measures (see e.g.~\cite{Bou08,HJXZ26a}), here we adopt an equivalent definition in terms of the transport form of the $1$-Wasserstein distance for the proof of Theorem~\ref{t: energy transport estimate: general condition}. For connections between Wasserstein distances and other topics in analysis, see Figalli and Glaudo \cite{FG23}. For a different line of work at the interface between ergodic optimization and optimal transport, see Lopes--Mengue \cite{LM12}, Contreras--Lopes--Oliveira \cite{CLO14}, and Lopes--Oliveira--Thieullen \cite{LOT15}. 

Generally, for each $p\in[1,+\infty)$, the $p$-Wasserstein distance between probability measures is obtained by minimizing the transport cost with respect to $d^p$ and taking the infimum to the power $1/p$, where $d$ is the underlying metric. Here we focus on the case $p=1$, while replacing the underlying metric $d$ by a new metric $d_{\omega}(\,\cdot\,,\,\cdot\,)\=\omega(d(\,\cdot\,,\,\cdot\,))$ on $\Sigma$ for some modulus of continuity $\omega$ that vanishes only at $0$. Note that a modulus of continuity $\omega$ satisfying the scale doubling condition vanishes only at $0$ (cf.~Definition~\ref{d: two technical conditions}).

\begin{definition}\label{d: Wasserstein}
    For each modulus of continuity $\omega$ that vanishes only at $0$, we define the \emph{$\omega$-Wasserstein distance} $W_{\omega}\:\cP(\Sigma)\times\cP(\Sigma)\to[0,+\infty)$ by 
    \begin{equation*}
        W_{\omega}(\mu,\nu)\=\inf_{\pi}\int_{\Sigma\times \Sigma}\!\omega(d(x,y))\,\mathrm{d}\pi(x,y)
    \end{equation*}
    for all Borel probability measures $\mu,\,\nu$ in $\cP(\Sigma)$, where the infinum ranges over all Borel probabilty measures $\pi$ on $\Sigma\times\Sigma$ whose marginal distributions are $\mu$ and $\nu$, respectively.
\end{definition}

Note that for each modulus of continuity $\omega$ that vanishes only at $0$, $d_{\omega}(x,y)\=\omega(d(x,y))$ defines a new metric on $\Sigma$, and by the Kantorovich duality formula (cf.~\cite[Remark~6.5]{Vi09}), it follows that 

\begin{equation}\label{eq: W_omega in terms of H^omega functions}
    W_{\omega}(\mu,\nu)=\sup\biggl\{\int_\Sigma\! f\,\mathrm{d}\mu-\int_\Sigma\! f\,\mathrm{d}\nu:f\in H^{\omega}(\Sigma,\R),\,\abs{f}_{\omega}\leq 1\biggr\}.
\end{equation}

Denote the $(\omega,K)$-plateau potential $f_{K,\omega}(\,\cdot\,)\=-\omega(d(\,\cdot\,,K))\in H^{\omega}(\Sigma,\R)$ for each modulus of continuity $\omega$\footnote{Here, without assuming that $\omega$ vanishes only at $0$, $f_{K,\omega}\in H^{\omega}(\Sigma,\R)$ follows from the monotonicity of $\omega$ and the fact that $d$ is an ultrametric on $\Sigma$, i.e., $d(x,y)\leq\max\{d(x,z),\,d(y,z)\}$ for all $x,\,y,\,z\in \Sigma$.}. Since periodic measures are weak$^*$ dense in the set $\cM(\Sigma,\sigma)$ of invariant measures (cf.~\cite{Sig74}), to prove Theorem~\ref{t: energy transport estimate: general condition}, it suffices to prove the following result, whose proof is given in Subsection~\ref{ss: comparing flows on Rauzy graphs}.

\begin{prop}[Transport gap]
\label{p: transport gap}
Under the hypotheses of Theorem~\ref{t: energy transport estimate: general condition}, there exists a constant $C>0$ such that for every periodic measure $\nu_p\in\cM(\Sigma,\sigma)$,
\begin{equation}\label{ineq: transport gap}
W_{\omega}(\nu_p,\mu_K)\leq
C \int_\Sigma\! -f_{K,\omega}\,\mathrm{d}\nu_p.
\end{equation}
\end{prop}

\begin{proof}[\bf Proof of Theorem~\ref{t: energy transport estimate: general condition} admitting Proposition~\ref{p: transport gap}]
    Let $C>0$ be the constant in Proposition~\ref{p: transport gap}. Note that a modulus of continuity $\omega$ satisfying the scale doubling condition vanishes only at $0$. Since periodic measures are weak$^*$ dense in $\cM(\Sigma,\sigma)$ (cf.~\cite{Sig74}), we see from Proposition~\ref{p: transport gap} and (\ref{eq: W_omega in terms of H^omega functions}) that $W_{\omega}(\nu,\mu_K)\leq C\int_{\Sigma}\!-f_{K,\omega}\,\mathrm{d}\nu$ for all $\nu\in\cM(\Sigma,\sigma)$. The proof is now complete with (\ref{eq: W_omega in terms of H^omega functions}).
\end{proof}

To establish Proposition~\ref{p: transport gap}, we recall the calculation for $W_{\omega}$ (cf.~\cite[Lemma~3.1]{Kl15}) in Proposition~\ref{p: multiscale formula}. We first prepare the following notations. For all $n\in\N$ and $w=\{w_i\}_{i=1}^n\in\{0,\,1\}^n$, denote $[w]\=\{x\in\Sigma: x_i=w_i\text{ for all }1\leq i\leq n\}$.

\begin{definition}[Cylinder discrepancy]\label{d: D_n}
For $\mu,\,\nu\in\cP(\Sigma)$ and $n\in\N$, define
\begin{equation*}
D_n(\mu,\nu)\=
\sum_{w\in\{0,\,1\}^n}
\abs{
\mu([w])-\nu([w])
}.
\end{equation*}
\end{definition}

As $(\Sigma,d_{\omega})$ is a compact ultrametric space for a modulus of continuity $\omega$ that vanishes only at $0$, the following multiscale formula is an immediate consequence of \cite[Lemma~3.1]{Kl15}.

\begin{prop}[Multiscale formula for $W_\omega$]
\label{p: multiscale formula}
Let $\omega$ be a modulus of continuity that vanishes only at $0$. Then for all $\mu,\,\nu\in\cP(\Sigma)$, 
\begin{equation*}
W_\omega(\mu,\nu)=\frac12\sum_{n=1}^{+\infty}
\bigl(\omega\bigl(2^{-n+1}\bigr)-\omega(2^{-n})\bigr)D_n(\mu,\nu).
\end{equation*}
\end{prop}

\subsection{Fibonacci subshift and its Rauzy graphs}\label{ss: Fibonacci subshift and Rauzy graph}
Here we review the Fibonacci subshift and its Rauzy graphs. Serving as a model for quasicrystals, the Fibonacci subshift arises naturally in a variety of research contexts, including low-complexity symbolic dynamics \cite{MH40}, the geometry and diffraction of aperiodic tillings \cite{CS03,Hof95}, and the spectral theory of quasiperiodic Schr\"{o}dinger operators \cite{DGY16}. Introduced by Rauzy in \cite{Ra83}, Rauzy graphs provide finite representations of the language of the subshift and serve as a useful tool for studying recurrence, complexity, and related dynamical properties (see e.g.\ \cite{Bru22}). 

In particular, probability measures supported on the Fibonacci subshift induces compatible circulations on Rauzy graphs of different scales (cf.~Proposition~\ref{p: cycle properties}). Thus to estimate the cylinder discrepancy (cf.~Definition~\ref{d: D_n}), we define the stable coordinate and establish a comparison lemma (Lemma~\ref{l: comparison}) for circulations on the Rauzy graphs. 

Consider the Fibonacci substitution
$\vartheta\:\Sigma\to\Sigma$ generated by $0\mapsto01$ and $1\mapsto 0$. Let $\mathbf f=\{f_i\}_{i=1}^{+\infty}\in\Sigma$ be the unique one-sided fixed point of $\vartheta$ satisfying $f_1=0$, and $K\=\overline{
\{\sigma^n(\mathbf f):n\in\N_0\}
}$ be the Fibonacci subshift. For each $n\in\N$, denote 
\begin{equation*}
\fL_n(K)\=\{w\in\{0,\,1\}^n:\text{there exists }x\in K\text{ such that }x_i=w_i\text{ for all }1\leq i\leq n\}.
\end{equation*}
The union $\fL(K)\=\bigcup_{n=1}^{+\infty}\fL_n(K)$ is the \emph{Fibonacci language}. 

The following result is well known (cf.~\cite[Theorem~1.6.9]{Fo02} and~\cite[Example 2.1.1]{Lo02}). 

\begin{prop}[Basic properties of the Fibonacci shift]
\label{p: fibonacci-basic}
The Fibonacci subshift $K$ is minimal and uniquely ergodic. Moreover, for each $n\in\N$, $\card(\fL_n(K))=n+1$.
\end{prop}

The following notion is standard in symbolic dynamics and combinatorics
on words (cf.~\cite{Cas97,Fo02,Lo02}).
\begin{definition}[Special words]
A word $w\in\fL(K)$ is called
\begin{enumerate}
    \smallskip
    \item \emph{right-special} if both $w0$ and $w1$ belong to
    $\fL(K)$;
    \smallskip
    \item \emph{left-special} if both $0w$ and $1w$ belong to
    $\fL(K)$;
    \smallskip
    \item \emph{bi-special} if it is both left-special and right-special.
\end{enumerate}
\end{definition}

Since $\card(\fL_n(K))=n+1$ for all $n\in\N$ (cf.~Proposition~\ref{p: fibonacci-basic}), it is straightforward to see that for all $n\in\N$, there is a unique left-special (resp.~right-special) word in $\fL_n(K)$. If the two special words in $\fL_n(K)$ coincide, then this word is bi-special and $n$ is called a \emph{bi-special scale}. Let $\{F_n\}_{n\in\N_0}$ be the Fibonacci sequence defined by $F_0\=0$, $F_1\=1$, and $F_{n+1}\=F_{n}+F_{n-1}$ for all $n\in\N$. It turns out that bi-special scales are related to the Fibonacci sequence. The result below follows from \cite[Proposition~2.2]{BL15}.

\begin{prop}[Bi-special scales]
\label{p: bi-special scales}
The set of bi-special words is $\{f_1f_2\ldots f_{F_m-2}:m\in\N,\,m\geq 4\}$.
\end{prop}

Now we introduce the notion of \emph{Rauzy graph}, which is a standard tool in combinatorics on words (cf.~\cite{Ra83,Fo02}). For a subshift of the one‑sided full shift $(\Sigma,\sigma)$, the associated Rauzy graphs are subgraphs of the de Bruijn--Good graphs for the full shift $\Sigma$. For a use of de Bruijn--Good graphs (cf.~\cite{Br46,Goo46}) for the Hunt--Ott conjecture, see \cite{BZ16,DLZ24}.

\begin{definition}[Rauzy graph]\label{d: Rauzy graph}
For each $n\in\N$, the \emph{Rauzy graph} of the Fibonacci subshift $K$ at scale $n$, denoted by $R_n(K)$, is the directed graph whose vertex and edge sets are
\begin{equation*}
V(R_n(K))=\fL_{n}(K)\quad\text{ and }\quad
E(R_n(K))=\fL_{n+1}(K),
\end{equation*}
respectively. An edge $e = x_1x_2 \ldots x_{n+1} \in \fL_{n+1}(K)$ is directed from the vertex $x_1x_2 \ldots x_{n}\in\fL_{n}(K)$ to the vertex $x_2x_3\ldots x_{n+1}\in\fL_{n}(K)$.
\end{definition}

Note that a left-special (resp.~right-special) word corresponds to a vertex
with indegree (resp.~outdegree) greater than one in the Rauzy graph. From now on, we focus on the Rauzy graphs at bi-special scales. Denote $q_k\=F_k-2$ for all $k\in\N_0$ and $\Gamma_m\= R_{q_m}(K)$ for all integers $m\geq 4$.  

For a directed graph $\Gamma=(V,E)$, let $C_1(\Gamma; \R)$ and $C_0(\Gamma; \R)$ denote the real vector spaces freely generated by its directed edges and vertices, respectively. We use the following coefficient notation: for $H\in C_1(\Gamma;\R)$, $H=\sum_{e\in E}H(e)e$, where $H(e)\in\R$ for all $e\in E$; for $P\in C_0(\Gamma;\R)$, $P=\sum_{v\in V}P(v)v$, where $P(v)\in\R$ for all $v\in V$. The boundary operator
$\partial\: C_1(\Gamma; \mathbb{R}) \to C_0(\Gamma; \mathbb{R})$
is a linear operator given by
$\partial e\=t(e)-o(e)$ for all directed edges $e\in E$, where $o(e),\,t(e)$ denote the origin and terminus of the edge $e$, respectively.

\begin{definition}[Circulation space]
For each integer $m\geq 4$, the \emph{circulation space} (or first homology group) of $\Gamma_m$ is
\begin{equation*}
\sZ_m\=\ker( \partial \colon C_1(\Gamma_m;\R) \to C_0(\Gamma_m;\R) ).
\end{equation*}
\end{definition}

The following result is a special case of \cite[Theorem~1]{AR91}.
\begin{prop}[Structure of $\Gamma_m$]
\label{p: structure of Gamma_m}
Let $m\geq 4$ be an integer. The Rauzy graph $\Gamma_m$ consists of two directed return loops based at the unique bi-special vertex. We denote the longer loop by $C_m^L$ and the shorter one by $C_m^S$. Their lengths are $F_{m-1}$ and $F_{m-2}$, respectively. Consequently, 
$\mathscr{Z}_m = \operatorname{span}\bigl\{ C_m^L,\, C_m^S \bigr\}$ and $\dim \mathscr{Z}_m = 2$.
\end{prop}

Now for each integer $m\geq 4$ and each $H\in\sZ_m$, we can define its loop coordinate $\bigl(x_m^L(H),x_m^S(H)\bigr)\in\R^2$ with respect to the basis $\bigl\{C_m^L,\,C_m^S\bigr\}$, i.e., 
\begin{equation}\label{eq: defi of loop coordinate}
H=x_m^L(H)C_m^L+x_m^S(H)C_m^S.
\end{equation}

Let $m\geq 4$ be an integer. We consider the graph homomorphism $\pi_m$ from $\Gamma_{m+1}$ to $\Gamma_{m}$ given by 
\begin{equation*}
\pi_m(x_1x_2\ldots x_{q_{m+1}})\=x_1x_2\ldots x_{q_m}
\end{equation*}
for all $x_1x_2\ldots x_{q_{m+1}}\in\fL_{q_{m+1}}(K)=V(\Gamma_{m+1})$. This map naturally induces a linear pushforward $(\pi_m)_\#$ from $C_1(\Gamma_{m+1}; \mathbb{R})$ to $C_1(\Gamma_m; \mathbb{R})$. Since $(\pi_m)_\#$ commutes with the boundary operator, it sends circulations to circulations and restricts to a linear map between the circulation spaces.

\begin{prop}[Projections of circulations]
\label{p: loops-under-projection}
Let $m\geq 4$ be an integer. The pushforward $(\pi_m)_\#$ acts on the two generating loops
$C_{m+1}^L,\, C_{m+1}^S \in \mathscr{Z}_{m+1}$ by
\begin{equation}\label{eq: projection of cycles}
(\pi_m)_\#\bigl(C_{m+1}^L\bigr) = C_m^L + C_m^S\quad\text{ and }
\quad
(\pi_m)_\#\bigl(C_{m+1}^S\bigr) = C_m^L.
\end{equation}
\end{prop}
\begin{proof}
    Let $(\pi_m)_\#\bigl(C_{m+1}^L\bigr) = aC_m^L + bC_m^S$ and $(\pi_m)_\#\bigl(C_{m+1}^S\bigr) =a' C_m^L+b'C_m^S$, where $a,\,b,\,a',\,b'\in\R$. It immediately follows from the definition of $\pi_m$ that $a,\,b,\,a',\,b'\in\N_0$. Now computing the length of the loops, we see that $F_m=aF_{m-1}+bF_{m-2}$ and $F_{m-1}=a'F_{m-1}+b'F_{m-2}$. If $m\geq 5$, it follows from $F_m\leq 2F_{m-1}$, $\gcd(F_{m},F_{m-2})=\gcd(F_{m-1},F_{m-2})=1$, and $F_{m-2}\geq 2$ that the unique solution is $(a,b,a',b')=(1,1,1,0)$. If $m=4$, then it is straightforward to check that (\ref{eq: projection of cycles}) holds by drawing $\Gamma_m=R_{1}(K)$ and $\Gamma_{m+1}=R_3(K)$.
\end{proof}

Let $m\geq 4$ be an integer. For the invariant measure $\mu_K$, we consider the chain $\mu_m\in C_1(\Gamma_m;\R)$ given by $\mu_m(e)\=\mu_K([e])$ for all $e\in E(\Gamma_m)=\fL_{q_m+1}(K)$. It follows from the invariance of $\mu_K$ that $\partial\mu_m(v)=\mu_K([0v])+\mu_K([1v])-\mu_K([v0])-\mu_K([v1])=\mu_K\bigl(\sigma^{-1}([v])\bigr)-\mu_K([v])=0$, i.e., $\mu_m\in\sZ_m$. Since $\mu_K$ is supported on $K$, we see that $(\pi_m)_{\#}(\mu_{m+1})=\mu_m$. It then follows from Proposition~\ref{p: loops-under-projection} that 
\begin{equation}\label{eq: inductive formula of loop coordinates}
    \begin{pmatrix}
x_m^L(\mu_m)\\
x_m^S(\mu_m)
\end{pmatrix}=
\begin{pmatrix}
    1 & 1\\
    1 & 0
\end{pmatrix}
\begin{pmatrix}
    x_{m+1}^L(\mu_{m+1})\\
x_{m+1}^S(\mu_{m+1})
\end{pmatrix}.
\end{equation}
Note that eigenvalues of $\begin{pmatrix}
    1 & 1\\
    1 & 0
\end{pmatrix}$ are $\frac{1+\sqrt{5}}{2}$ and $\frac{1-\sqrt{5}}{2}$, and the associated eigenspaces are spanned by $\begin{pmatrix}
    (1+\sqrt{5})/2\\
    1
\end{pmatrix}$ and $\begin{pmatrix}
    (1-\sqrt{5})/2\\
    1
\end{pmatrix}$. Since $\begin{pmatrix}
    a\\
    b
\end{pmatrix}=\frac{2a+(\sqrt{5}-1)b}{2\sqrt{5}}\begin{pmatrix}
    (1+\sqrt{5})/2\\
    1
\end{pmatrix}+\frac{-2a+(\sqrt{5}+1)b}{2\sqrt{5}}\begin{pmatrix}
    (1-\sqrt{5})/2\\
    1
\end{pmatrix}$ for all $a,\,b\in\R$, we define the unstable and stable coordinates
$u_m(H),\,s_m(H)$ for each $H\in\sZ_m$
by 
\begin{equation}\label{eq: defi of two coordinates}
u_m(H)\=\frac{2x_m^L(H)+(\sqrt{5}-1)x_m^S(H)}{2\sqrt{5}}\quad\text{ and }\quad s_m(H)\=\frac{-2x_m^L(H)+(1+\sqrt{5})x_m^S(H)}{2\sqrt{5}}.
\end{equation}
It then follows from (\ref{eq: inductive formula of loop coordinates}) that 
\begin{equation}\label{eq: eigendirection coordinate for mu_K}
    u_m(\mu_m)=\frac{1+\sqrt{5}}{2}u_{m+1}(\mu_{m+1})\quad\text{ and }\quad s_m(\mu_m)=\frac{1-\sqrt{5}}{2}s_{m+1}(\mu_{m+1}).
\end{equation}
By the Perron--Frobenius theorem\footnote{Note that for each integer $n\geq 4$, the loop coordinates $x_n^L(\mu_n)$ and $x_n^S(\mu_n)$ are positive. Thus, for each integer $m\geq 4$, $\begin{pmatrix}
    1 &1\\
    1 &0
\end{pmatrix}^k\begin{pmatrix}
    x_{m+k}^L(\mu_{m+k})\\
    x_{m+k}^S(\mu_{m+k})
\end{pmatrix}$ converges to $\begin{pmatrix}
    (1+\sqrt{5})/2\\
    1
\end{pmatrix}$ as $k\to+\infty$ under the Hilbert projective metric.} (cf.~\cite[Theorem~1]{Bi57}), we see that $s_m(\mu_m)=0$. We conclude the above discussions in the following proposition.

\begin{prop}[Circulation properties]\label{p: cycle properties}
Let $m\geq 4$ be an integer. The following hold:

\begin{enumerate}
\smallskip
\item The chain $\mu_m$ is in $\sZ_m$, $(\pi_{m})_{\#}(\mu_{m+1})=\mu_m$, and $s_m(\mu_m)=0$.
\smallskip
\item For each $H\in \sZ_{m+1}$, $s_m((\pi_m)_{\#}(H))=\frac{1-\sqrt{5}}{2}s_{m+1}(H)$.
\end{enumerate}
\end{prop}
\begin{proof}
    Here (i) is already proved in the discussions above. For (ii), let $H=x_{m+1}^L(H)C^L_{m+1}+x_{m+1}^S(H)C_{m+1}^S\in \sZ_{m+1}$, it follows from Proposition~\ref{p: loops-under-projection} that $(\pi_m)_{\#}(H)=(x_{m+1}^L(H)+x_{m+1}^S(H))C_m^L+x_{m+1}^L(H)C_m^S$, i.e., $x_m^L((\pi_m)_{\#}(H))=x_{m+1}^L(H)+x_{m+1}^S(H)$ and $x_m^S((\pi_m)_{\#}(H))=x_{m+1}^L(H)$. This, together with (\ref{eq: defi of two coordinates}), yields (ii).
\end{proof}

For each integer $m\geq 4$ and each $H=\sum_{e\in E(\Gamma_m)}H(e)e\in C_1(\Gamma_m;\R)$, 
we define the mass and norm by 
\begin{equation}\label{eq: defi of mass and norm}
    \mathfrak{m}_m(H)\=\sum_{e\in E(\Gamma_m)}H(e)\quad\text{ and }\quad\norm{H}_m\=\sum_{e\in E(\Gamma_m)}\abs{H(e)}.
\end{equation}
We conclude this subsection with the following comparison lemma.

\begin{lemma}[Comparison]
\label{l: comparison}
There exists a constant $c_1>0$ such that for each integer $m\geq 4$ and each $H\in\sZ_m$,
\begin{equation*}
\norm{H-\mu_m}_m\leq c_1(\abs{\mathfrak{m}_m(H)-1}+q_m
\abs{s_m(H)}).
\end{equation*}
\end{lemma}
\begin{proof}
   Since $\mu_K$ is a probability measure, we see that $\mathfrak{m}_m(\mu_m)=1$ for each integer $m\geq 4$. By Proposition~\ref{p: cycle properties}~(i), $s_m(\mu_m)=0$ for each integer $m\geq 4$. Thus, it suffices to prove the following claim.

    \smallskip
    \emph{Claim.} There exists a constant $c_1>0$ such that for each integer $m\geq 4$ and each $H\in\sZ_m$,
    \begin{equation*}\label{ineq: estimate of norm H by mass and stable coordinate}
        \norm{H}_m\leq c_1(\abs{\mathfrak{m}_m(H)}+q_m\abs{s_m(H)}).
    \end{equation*}
    \smallskip
    Let $m\geq 4$ be an integer and $H\in\sZ_m$. Recall that $H=x_m^L(H)C_m^L+x_m^S(H)C_m^S$ (cf.~(\ref{eq: defi of loop coordinate})). It then follows from Proposition~\ref{p: structure of Gamma_m} that $\norm{H}_m=F_{m-1}\Absbig{x_m^L(H)}+F_{m-2}\Absbig{x_m^S(H)}$ and $\mathfrak{m}_m(H)=F_{m-1}x_{m}^L(H)+F_{m-2}x_{m}^S(H)$.
    
    If $x_m^L(H)x_{m}^S(H)\geq 0$, then $\norm{H}_m=\abs{\mathfrak{m}_m(H)}$. If $x_{m}^L(H)x_{m}^S(H)<0$, then by (\ref{eq: defi of two coordinates}), 
    \begin{equation*}
        \begin{aligned}
        \norm{H}_m\leq 2q_m\bigl(\Absbig{x_m^L(H)}+\Absbig{x_m^S(H)}\bigr)\leq 2q_m(\sqrt{5}+2)\abs{s_m(H)}.
        \end{aligned}
    \end{equation*}
    We conclude that the claim holds with $c_1\=4+2\sqrt{5}$ and the proof is now complete.
\end{proof}

\subsection{Accident decomposition and the approximating measure}\label{ss: accident decomposition}

In the remaining parts of this section, we fix some periodic point $p\in \Sigma$ of prime period $N\in\N$ for the proof of Proposition~\ref{p: transport gap}. Then the corresponding periodic measure $\nu_p$ is equal to $\frac{1}{N}\sum_{i=0}^{N-1}\delta_{\sigma^i(p)}$, where $\delta_x$ stands for the Dirac measure at $x\in\Sigma$. In this subsection, we construct the measure $\eta_z$ close to $\nu_p$ and supported on the Fibonacci subshift $K$, with the notion of accident introduced by Bruin and Leplaideur in \cite{BL13,BL15}. 

\begin{definition}[Shadow length and accidents]\label{d: shadow length and accident}
For each $x\in\Sigma$, we define the shadow length
\begin{equation*}
\ell_K(x)\=\sup\{n\in\N:x_1x_2\ldots x_{n}\in\fL_n(K)\}.
\end{equation*}
For the periodic point $p$, denote $\ell_i\=\ell_K\bigl(\sigma^i(p)\bigr)$ for all $i\in\Z/N\Z$. Sine $K$ is $\sigma$-invariant, $l_{i}\geq l_{i-1}-1$ for all $i\in\Z/N\Z$. We say that there is an \emph{accident} at $\sigma^i(p)$ for some $i\in Z/N\Z$ if $l_i\geq l_{i-1}$.
\end{definition}

It is straightforward to see that there exists $i\in\Z/N\Z$ so that we have an accident at $\sigma^i(p)$. Without loss of generality, we may assume that there is an accident at $p$. We then have the following decomposition of the periodic orbit.

\begin{definition}[Accident times for the periodic orbit]\label{d: accident times}
Let $0=B_0<B_1<\cdots<B_k=N$ for some $k\in\N$ be the accident times for the periodic orbit $\bigl\{\sigma^i(p):i\in\Z/N\Z\bigr\}$, inductively defined by $B_0\=0$ and $B_{j+1}\=\min\bigl\{N\geq i>B_j:\text{ there is an accident at }\sigma^{i}(p)\bigr\}$.
\end{definition}

For each $j\in \Z/k\Z$, denote
\begin{equation}\label{eq: defi of Dj bj rj}
    D_j\=\ell_{B_j},\quad b_j\=B_{j+1}-B_j\in\N,\quad\text{ and }\quad r_j\=D_j-b_j\in\N_0.
\end{equation} 
It follows that for each $j\in\Z/k\Z$ and each integer $t\in [0,b_j)$, $\ell_{B_j+t}=D_j-t=r_j+b_j-t$. Now we can calculate the right-hand side of (\ref{ineq: transport gap}): for each modulus of continuity $\omega$,
\begin{equation}\label{ineq: estimate of bonus}
    \int\!-f_{K,\omega}\,\mathrm{d}\nu_p=\frac1N\sum_{i=0}^{N-1}\omega\bigl(2^{-\ell_i}\bigr)=\frac1N\sum_{j=0}^{k-1}(\assocsum{2}{D_j}(1)-\assocsum{2}{r_j}(1)).
\end{equation}

Now we review the following results from \cite{BL15}, which connect accidents to bi-special words and providing estimates on $b_j$.
\begin{prop}[Accident residual is bi-speical]\label{p: accident residual is bi-special}
Suppose $j\in\Z/k\Z$. If $r_j\in\N$, then the residual word $\bigl(\sigma^{B_{j+1}}(p)\bigr)\bigl|_{r_j}$ is
bi-special. Consequently, there exists a unique integer $n_j\geq 3$ such that $r_j=F_{n_j}-2=q_{n_j}$.
\end{prop}
\begin{proof}
    The first half immediately follows from \cite[Lemma~2.9]{BL15}. For the second half, the case when $r_j\in\N$ follows from Proposition~\ref{p: bi-special scales} and we get $n_j=3$ when $r_j=0$.
\end{proof}

\begin{prop}[Accident gap estimate]
\label{p: accident gap}
There exists a constant $c_3>0$ such that
for each $j\in\Z/k\Z$, $b_j\geq c_3(r_{j-1}+1)$.
\end{prop}
\begin{proof}
    Let $n_a\geq 3$, $a\in\Z/k\Z$, be the integers from Proposition~\ref{p: accident residual is bi-special}. Fix $j\in \Z/k\Z$. Assume that $r_{j-1}\geq 3$. First consider the case where $r_j\geq 3$. It follows from the proof of \cite[Lemma~4.2]{BL15} that $b_j\geq F_{n_{j-1}-1}+F_{n_j-2}$. For the Fibonacci sequence $\{F_n\}_{n\in\N_0}$, there exists a constant $c>0$ such that $F_{n+1}\leq cF_n$ for all $n\in\N$. It follows that $b_j\geq \frac{1}{c}F_{n_{j-1}}=\frac{1}{c}(r_{j-1}+2)$. Now consider the case where $r_j\leq 2$. By Definition~\ref{d: accident times}, $l_{B_j}\geq l_{B_{j}-1}$, i.e., $b_j+r_j\geq r_{j-1}+1$. As $r_{j-1}\geq 3$ and $r_j\leq 2$, we see that $b_j\geq r_{j-1}-1\geq \frac{1}{2}(r_{j-1}+1)$. If $r_{j-1}\leq 2$, then $b_j\geq 1\geq \frac{1}{3}(r_{j-1}+1)$. We take $c_3\=\min\bigl\{\frac{1}{c},\,\frac{1}{3}\bigr\}\in (0,+\infty)$ and the proof is now complete.
\end{proof}

We conclude this subsection with the construction of the measure $\eta_z$.

\begin{definition}[Construction of $\eta_z$]\label{d: defi of eta_z}
For each $j\in\Z/k\Z$, we fix some $y^{j}\in K$ such that $d\bigl(\sigma^{B_j}(p),y^j\bigr)=d\bigl(\sigma^{B_j}(p),K\bigr)=2^{-D_j}$. Then define $z_{B_j+t}\=\sigma^t\bigl(y^{j}\bigr)$ for each $j\in\Z/k\Z$ and each integer $t\in [0,b_j)$. It follows that $d\bigl(\sigma^i(p),z_i\bigr)=d\bigl(\sigma^i(p),K\bigr)=2^{-l_i}$ for all $i\in\Z/N\Z$. Define $\eta_z\=\frac1N\sum_{i=0}^{N-1}\delta_{z_i}$.
\end{definition}

\begin{prop}[Difference between $\nu_p$ and $\eta_z$]
\label{p: difference of nu_p and eta_z}
For each modulus of continuity $\omega$ that vanishes only at $0$, $W_{\omega}(\nu_p,\eta_z)\le\frac1N\sum_{j=0}^{k-1}(\assocsum{2}{D_j}(1)-\assocsum{2}{r_j}(1))$.
\end{prop}
\begin{proof}
    Consider $\pi\=\frac1N\sum_{i=0}^{N-1}\delta_{(\sigma^i(p),z_i)}$. Note that $\pi$ is a Borel probability measures on $\Sigma\times\Sigma$ with marginal distributions $\nu_p$ and $\eta_z$. It then follows from Definition~\ref{d: Wasserstein} and~(\ref{ineq: estimate of bonus}) that 
    \begin{align*}
            W_{\omega}(\nu_p,\eta_z)&\leq \int_{\Sigma\times\Sigma}\!\omega(d(x,y))\,\mathrm{d}\pi(x,y)=\frac{1}{N}\sum_{i=0}^{N-1}\omega\bigl(d\bigl(\sigma^i(p),z_i\bigr)\bigr)\\
            &=\frac1N\sum_{i=0}^{N-1}\omega\bigl(2^{-l_i}\bigr)=\frac1N\sum_{j=0}^{k-1}(\assocsum{2}{D_j}(1)-\assocsum{2}{r_j}(1)).\qedhere
    \end{align*}
\end{proof}

\subsection{Comparing flows on Rauzy graphs}\label{ss: comparing flows on Rauzy graphs}

By Propositions~\ref{p: difference of nu_p and eta_z} and $\ref{p: multiscale formula}$, it suffices to estimate $W_\omega(\eta_z,\mu_K)$ and $D_n(\eta_z,\mu_K)$ for the proof of Proposition~\ref{p: transport gap}. To this end, we consider the circulations induced by $\eta_z$ on Rauzy graphs $\{\Gamma_m\}_{m\geq 4}$.

Let $e_{i,m}\=z_i|_{q_m+1}\in\fL_{q_m+1}(K)=E(\Gamma_m)$. In general, the chain $\sum_{i=0}^{N-1}e_{i,m}$ may not be in $\sZ_m$. In order to apply Proposition~\ref{p: structure of Gamma_m}, we first modify $\sum_{i=0}^{N-1}e_{i,m}\in C_1(\Gamma_m;\R)$ to get a circulation. We need the following definition, which captures when the terminus $t(e_{i,m})$ may not be equal to the origin $o(e_{i+1,m})$.

\begin{definition}[Bad windows]\label{d: bad windows}
For each integer $m\geq 4$, we define $U_m\=\{i\in\mathbb Z/N\mathbb Z:
\ell_i\le q_m\}$.
\end{definition}

Note that if $i\notin U_m$ for some $i\in\Z/N\Z$ and $m\geq 4$ is an integer, then 
\begin{equation}\label{eq: good edge vertex cancelling}
z_{i+1}|_{q_m}=\sigma z_i|_{q_m}\quad\text{ and }\quad t(e_{i,m})=o(e_{i+1,m}).
\end{equation}

\begin{prop}[Geometry of the bad windows]
\label{p: geometry of bad window}
Let $m\geq 4$ be an integer and $j\in\Z/k\Z$. 
Then $U_m\cap[B_j,B_{j+1})=[B_{j+1}-w_{j,m},B_{j+1})$, where $w_{j,m}\=\min\{(q_m-r_j)_+,\,b_j\}$. Moreover, $\card U_m=\sum_{j=0}^{k-1}w_{j,m}$.
\end{prop}
\begin{proof}
    By Definitions~\ref{d: shadow length and accident} and~\ref{d: accident times}, we see that for each integer $i\in [B_j,B_{j+1})$, $l_i=D_j-(i-B_j)=r_j+b_j+B_j-i=r_j+(B_{j+1}-i)$. It then follows from Definition~\ref{d: bad windows} that $U_m\cap [B_j,B_{j+1})=\{i\in \N_0\cap[B_j,B_{j+1}):r_j+B_{j+1}-i\leq q_m\}=[B_{j+1}-w_{j,m},B_{j+1})$. Since $\{[B_j,B_{j+1})\}_{j\in\Z/k\Z}$ forms a partition of $\N_0\cap[0,N)$, we see that $\card U_m=\sum_{j=0}^{k-1}w_{j,m}$.
\end{proof}

Suppose $m\geq 4$ is an integer. Let $\kappa_m$ denote the number of connected components of $U_m$ in the cyclic time space $\mathbb Z/N\mathbb Z$ if $\emptyset\subsetneq U_m\subsetneq\Z/N\Z$ and set $\kappa_m\=0$ if $U_m=\emptyset$ or $U_m=\Z/N\Z$. Now consider the good edge chain
\begin{equation}\label{eq: defi of good edge chain}
F_m\=\frac1N\sum_{i\in (\Z/N\Z)\smallsetminus U_m}e_{i,m}.
\end{equation}

\begin{prop}[Boundary repair]
\label{p: boundary repair}
Let $m\geq 4$ be an integer. There exists a chain $Q_m\in C_1(\Gamma_m;\mathbb R)$
such that $\partial Q_m=\partial F_m$
and $\norm{Q_m}_m\leq (q_m+1)\kappa_m/N$.
\end{prop}
\begin{proof}
    It follows from (\ref{eq: good edge vertex cancelling}) and Proposition~\ref{p: geometry of bad window} that $\partial F_m$ consists of $\kappa_m$ pairs of sources and sinks with weight $1/N$. For each such pair of source and sink, we pick a route on $\Gamma_m$ from the source to the sink, and let $Q_m$ be $1/N$ times of the sum of these $\kappa_m$ routes. It follows that $\partial Q_m=\partial F_m$. Since $\card (E(\Gamma_m))=\card(\fL_{q_m+1}(K))=q_m+2$, we see that the length of each route is bounded by $q_m+1$ (cf.~Proposition~\ref{p: structure of Gamma_m}), and $\norm{Q_m}_m\leq (q_m+1)\kappa_m/N$.
\end{proof}

From now on, we fix some $Q_m$ in Proposition~\ref{p: boundary repair} for each integer $m\geq 4$ and consider 
\begin{equation}\label{eq: defi of H_m}
    H_m\=F_m-Q_m\in\sZ_m.
\end{equation}

\begin{prop}[Estimating cylinder discrepancy with $H_m$]
\label{p: cylinder discrepancy estimated through H_m}
Let $c_1$ be the constant from Lemma~\ref{l: comparison}. Then for each integer $m\geq 4$,
\begin{equation}\label{ineq: cyclinder discrepancy with H_m}
D_{q_m+1}(\eta_z,\mu_K)\leq (c_1+1)\bigl(N^{-1}\card U_m+N^{-1}(q_m+1)\kappa_m+q_m\abs{s_m(H_m)}\bigr).
\end{equation}
\end{prop}
\begin{proof}
    Let $m\geq 4$ be an intger. Denote $\eta_z^m\=\sum_{e\in E(\Gamma_m)}\eta_z([e])e=\frac{1}{N}\sum_{i=0}^{N-1}e_{i,m}\in C_1(\Gamma_m;\R)$. Note that both $\eta_z$ and $\mu_K$ are supported on $K$. It then follows from Definition~\ref{d: D_n} and~(\ref{eq: defi of mass and norm}) that $D_{q_m+1}(\eta_z,\mu_K)=\norm{\eta_z^m-\mu_m}_m$.  We see that 
    \begin{equation}\label{ineq: norm of eta_z minus mu_m}
        \begin{aligned}
            \norm{\eta_z^m-\mu_m}_m&\leq\norm{\eta_z^m-F_m}_m+\norm{F_m-H_m}_m+\norm{H_m-\mu_m}_m\\
            &\leq N^{-1}\card U_m+N^{-1}(q_m+1)\kappa_m+\norm{H_m-\mu_m}_m\\
            &\leq N^{-1}\card U_m+N^{-1}(q_m+1)\kappa_m+c_1(\abs{\mathfrak{m}_m(H_m)-1}+q_m\abs{s_m(H_m)}),
        \end{aligned}
    \end{equation}
    where the second inequality follows from (\ref{eq: defi of good edge chain}) and Proposition~\ref{p: boundary repair}, and the third inequality follows from Lemma~\ref{l: comparison}. It also follows from Proposition~\ref{p: boundary repair} and Lemma~\ref{l: comparison} that  
    \begin{equation*}
        \abs{\mathfrak{m}_m(H_m)-1}\leq \abs{\mathfrak{m}_m(Q_m)}+\abs{\mathfrak{m}_m(F_m)-1}\leq N^{-1}(q_m+1)\kappa_m+N^{-1}\card U_m.
    \end{equation*}
   This, together with (\ref{ineq: norm of eta_z minus mu_m}), yields (\ref{ineq: cyclinder discrepancy with H_m}).
\end{proof}

To state the next result, we denote 
\begin{equation}\label{eq: defi of Deltam}
    \Delta_m\=\begin{cases}
        \omega\bigl(2^{-1-q_{m-1}}\bigr)-\omega\bigl(2^{-1-q_m}\bigr) &\text{if }m\in\N\text{ and }m\geq5,\\
        \omega(1)-\omega\bigl(2^{-1-q_4}\bigr) &\text{if }m=4.
    \end{cases}
\end{equation}
\begin{prop}[Reduction to bi-sepcial scales]
\label{p: reduction to bi-special scales}
Let $\omega$ be a modulus of continuity that vanishes only at $0$. Then for all Borel probability measures $\mu,\,\nu$ on $\Sigma$,
\begin{equation*}
W_{\omega}(\mu,\nu)\leq
\frac{1}{2}\sum_{m=4}^{+\infty}\Delta_mD_{q_{m}+1}(\mu,\nu).
\end{equation*}
\end{prop}
\begin{proof}
    Let $\mu,\,\nu$ be two probability measures on $\Sigma$. By Definition~\ref{d: D_n}, we see that $\{D_n(\mu,\nu)\}_{n\in\N}$ is a nondecreasing sequence. Denote $\omega_{n}\=\omega(2^{-n})$ for all $n\in\N_0$. It then follows from Proposition~\ref{p: multiscale formula} and~(\ref{eq: defi of Deltam}) that 
    \begin{align*}
        2W_\omega(\mu,\nu)&\leq D_{q_4+1}(\mu,\nu)\sum_{n=1}^{q_4+1}(\omega_{n-1}-\omega_n)+\sum_{m= 5}^{+\infty}D_{q_{m}+1}(\mu,\nu)\sum_{n=q_{m-1}+2}^{q_{m}+1}(\omega_{n-1}-\omega_n)\\
        &=(\omega_0-\omega_{q_4+1})D_{q_4+1}(\mu,\nu)+\sum_{m=5}^{+\infty}(\omega_{q_{m-1}+1}-\omega_{q_m+1})D_{q_{m}+1}(\mu,\nu)=\sum_{m=4}^{+\infty}\Delta_mD_{q_{m}+1}(\mu,\nu).\qedhere
    \end{align*}
\end{proof}

With Propositions~\ref{p: cylinder discrepancy estimated through H_m} and Propositions~\ref{p: reduction to bi-special scales}, We can estimate $W_{\omega}(\eta_z,\mu_K)$ through the weighted sum of the three terms on the right-hand side of (\ref{ineq: cyclinder discrepancy with H_m}). To this end, we introduce the following conditions.

\begin{definition}[Two technical conditions]\label{d: two technical conditions}
    Let $c_3>0$ be the constant in Proposition~\ref{p: accident gap}. A modulus of continuity $\omega$ is said to satisfy the \emph{scale doubling} condition if there exists a constant $\gamma_1>1$ such that for each integer $n\geq 2$,
    \begin{equation}\label{ineq: scale doubling condition}
        \omega\bigl(2^{-{q_{n}-1}}\bigr)<\gamma_1\omega\bigl(2^{-q_{n+1}-1}\bigr),
    \end{equation}
    and is said to satisfy the \emph{stable} condition if there exists a constant $\gamma_2>0$ such that for all integers $m\geq 4$,
    \begin{equation}\label{ineq: stable condition}
        \sum_{k=4}^m\biggl(\frac{\sqrt{5}-1}{2}\biggr)^{m-k}q_k\Delta_k\leq \gamma_2 q_m\Delta_m.
    \end{equation}
\end{definition}

\begin{rem}[Range of the technical conditions]\label{r: range of the technical conditions}
    For the family of moduli $\{\omega_\alpha\}_{\alpha\in(0,+\infty)}$ defined by (\ref{eq: defi of the omegaalpha family}), it is straightforward to check that $\omega_\alpha$ satisfies the scale doubling condition for all $\alpha\in(0,+\infty)$, and satisfies the stable condition if and only if $\alpha\in (0,2)$. Note that $\omega_\alpha$ satisfies the Dini condition if and only if $\alpha\in(1,+\infty)$. 
\end{rem}

Now we estimate the weighted sum of each of the three terms on the right-hand side of (\ref{ineq: cyclinder discrepancy with H_m}) in Lemmas~\ref{l: bad window estimates},~\ref{l: bad block estimates}, and Proposition~\ref{p: stable coordinate estimates II}, respectively.

\begin{lemma}[Bad window estimates]
\label{l: bad window estimates}
Let $\omega$ be a modulus of continuity satisfying the scale doubling condition. Then there exists a constant $c_4>0$ independent of the periodic point $p$ such that
\begin{equation}\label{ineq: bad window estimates}
\sum_{m= 4}^{+\infty}\Delta_m\cdot \card U_m\leq c_4\sum_{j=0}^{k-1}(\assocsum{2}{D_j}(1)-\assocsum{2}{r_j}(1)).
\end{equation}
\end{lemma}
\begin{proof}
   By Proposition~\ref{p: geometry of bad window}, it suffices to show that there exist a constant $c_4>0$ such that for each $j\in\Z/k\Z$,
    \begin{equation}\label{ineq: bad window estimates-single j}
        \sum_{m=4}^{+\infty}\Delta_mw_{j,m}\leq c_4(\assocsum{2}{D_j}(1)-\assocsum{2}{r_j}(1)).
    \end{equation}
    Denote $\omega_n\=\omega(2^{-n})$ for $n\in\N_0$ and $M(a,b)\=\{m\in\N: m\geq 4,\, a\leq q_m\leq b\}$ for $a,\,b\in[0,+\infty]$.

    Now fix $j\in\Z/k\Z$. Recall that $w_{j,m}=\min\{(q_m-r_j)_+,\,b_j\}$ for each integer $m\geq 4$. It follows that 
    \begin{equation}\label{eq: calculating w_j,m}
        \sum_{m= 4}^{+\infty}\Delta_m w_{j,m}=\sum_{m\in M(r_j+b_j+1,+\infty)}\Delta_m b_j+\sum_{m\in M(r_j+1,r_j+b_j)}\Delta_m(q_m-r_j).
    \end{equation}
    We now treat the two parts in the right-hand side of (\ref{eq: calculating w_j,m}) separately. 
    
    Since $q_4=1$, we see that $m'\=\min M(r_j+b_j+1,+\infty)\geq 5$. Let $\gamma_1>1$ be the constant from (\ref{ineq: scale doubling condition}). It then follows from (\ref{ineq: scale doubling condition}), (\ref{eq: defi of Dj bj rj}), and the monotonicity of $\omega$ that
    \begin{equation}\label{ineq: bad window-m big part}
        \begin{aligned}
        \sum_{m\in M(r_j+b_j+1,+\infty)}\Delta_mb_j
        &=\sum_{m=m'}^{+\infty}(\omega_{q_{m-1}+1}-\omega_{q_m+1})b_j
        =b_j\omega_{q_{m'-1}+1}
        \leq \gamma_1b_j\omega_{q_{m'}+1}\\
        &\leq\gamma_1b_j\omega_{r_j+b_j}
        =\gamma_1b_j\omega_{D_j}
        \leq\gamma_1(\assocsum{2}{D_j}(1)-\assocsum{2}{r_j}(1)).
        \end{aligned}
    \end{equation}

    Consider $m''\=\min M(r_j+1,+\infty)$. Note that if $m''\geq m'$, then the second part of the right-hand side of (\ref{eq: calculating w_j,m}) vanishes so that (\ref{ineq: bad window estimates}) follow from (\ref{ineq: bad window-m big part}). Now suppose $m''<m'$. It follows that 
    \begin{equation*}
        \sum_{m\in M(r_j+1,r_j+b_j)}\Delta_m(q_m-r_j)=\Delta_{m''}(q_{m''}-r_j)+\sum_{m=m''+1}^{m'-1}\Delta_m(q_m-r_j),
    \end{equation*}
    where the second part of the right-hand side vanishes if $m'=m''+1$.  For the first item, if $m''\geq 5$, then  
   \begin{equation*}
        \begin{aligned}
       \Delta_{m''}(q_{m''}-r_j)
       &=(\omega_{q_{m''-1}+1}-\omega_{q_{m''}+1})(q_{m''}-r_j)
       \leq (\gamma_1-1)\omega_{q_{m''}+1}(q_{m''}-r_j)\\
       &\leq(\gamma_1-1)\sum_{n=r_j+1}^{q_{m''}}\omega_n\leq (\gamma_1-1)(\assocsum{2}{D_j}(1)-\assocsum{2}{r_j}(1)),
       \end{aligned}
   \end{equation*}
   where the first inequality follows from the scale doubling condition (cf.~(\ref{ineq: scale doubling condition})), the second inequality follows from the monotonicity of $\omega$, and the last inequality follows from the fact that $m''<m'$ (cf.~(\ref{eq: defi of Dj bj rj})). If $m''=4$, then it follows from the definition of $m''$ that $r_j=0$, and we see from (\ref{eq: defi of Deltam}) and (\ref{ineq: scale doubling condition}) that $\Delta_{m''}(q_{m''}-r_j)=\Delta_4(1-0)=\omega_0-\omega_2\leq\omega_0\leq \gamma_1\omega_1\leq\gamma_1\sum_{n=r_j+1}^{r_j+b_j}\omega_n=\gamma_1(\assocsum{2}{D_j}(1)-\assocsum{2}{r_j}(1))$. 
   
   For the remaining items, 
    \begin{equation}\label{ineq: bad window- remaining item for m modest part}
        \begin{aligned}
        \sum_{m=m''+1}^{m'-1}
        &\Delta_m(q_m-r_j)\leq\sum_{m=m''+1}^{m'-1}\Delta_mq_m
        =\sum_{m=m''+1}^{m'-1}q_m(\omega_{q_{m-1}+1}-\omega_{q_m+1})\\
        &\leq(\gamma_1-1)\sum_{m=m''+1}^{m'-1}q_m\omega_{q_m+1}
        \leq (\gamma_1-1)\sum_{m=m''+1}^{m'-1}3\sum_{n=q_{m-1}+1}^{q_m}\omega_n
        =3(\gamma_1-1)\sum_{n=q_{m''}+1}^{q_{m'-1}}\omega_n\\
        &=3(\gamma_1-1)(\assocsum{2}{q_{m'-1}}(1)-\assocsum{2}{q_{m''}}(1))
        \leq 3(\gamma_1-1)(\assocsum{2}{D_j}(1)-\assocsum{2}{r_j}(1)),
        \end{aligned}
    \end{equation}
   where the third inequality follows from the fact that $q_m\leq 3(q_m-q_{m-1})$ for each integer $m\geq 4$. We conclude that 
   \begin{equation}\label{ineq: bad window-m modest part}
    \sum_{m\in M(r_j+1,r_j+b_j)}\Delta_m(q_m-r_j)\leq (4\gamma_1-3)(\assocsum{2}{D_j}(1)-\assocsum{2}{r_j}(1)).
   \end{equation}
   
    Combining (\ref{eq: calculating w_j,m}), (\ref{ineq: bad window-m big part}), and~(\ref{ineq: bad window-m modest part}), we see that (\ref{ineq: bad window estimates-single j}) holds with $c_4\=5\gamma_1-3$.
\end{proof}

\begin{lemma}[Bad block estimates] \label{l: bad block estimates}
Let $\omega$ be a modulus of continuity satisfying the scale doubling condition. Then there exists a constant $c_5>0$ independent of the periodic point $p$ such that
\begin{equation}\label{ineq: bad block estimate}
\sum_{m= 4}^{+\infty}\Delta_mq_m\kappa_m\leq c_5\sum_{j=0}^{k-1}(\assocsum{2}{D_j}(1)-\assocsum{2}{r_j}(1)).
\end{equation}
\end{lemma}
\begin{proof}
    Let $m\geq 4$ be an integer. We may assume that $\emptyset\subsetneq U_m\subsetneq \Z/N\Z$ (otherwise $\kappa_m=0$). Now $U_m$ has $\kappa_m$ connected components. Due to Proposition~\ref{p: geometry of bad window}, we say that a connected component $I\subseteq U_m$ is of type one if there exists $j\in \Z/k\Z$ such that $0<w_{j,m}<b_j$ and $U_m\cap[B_j,B_{j+1})\subseteq I$, otherwise we say that $I$ is of type two. Denote the number of type one (resp.~two) components by $\kappa_m^1$ (resp.~$\kappa_m^2$). It follows that 
    \begin{equation}\label{ineq: single kappa 1}
    \kappa_m^1\leq \card\{j\in\Z/k\Z:0<w_{j,m}<b_j\}=\card\{j\in\Z/k\Z: r_j<q_m<r_j+b_j\}.
    \end{equation}

    Now consider a component $I\subseteq U_m$ of type two. It follows that for all $j\in \Z/k\Z$, $I\cap[B_j,B_{j+1})$ is either $\emptyset$ or the whole interval $[B_j,B_{j+1})$. Since $\emptyset\subsetneq U_m\subsetneq \Z/N\Z$, we see that there exists a unique $j\in\Z/k\Z$ such that $[B_j,B_{j+1})\subseteq I$ and $[B_{j-1},B_{j})\cap I=\emptyset$, i.e., $r_j+b_j\leq q_m\leq r_{j-1}$. It follows that 
    \begin{equation}\label{ineq: single kappa 2}
    \kappa_m^2\leq \card\{j\in\Z/k\Z:r_j+b_j\leq q_m\leq r_{j-1}\}.
    \end{equation}
    
     Denote $\omega_n\=\omega(2^{-n})$ for $n\in\N_0$ and $M(a,b)\=\{m\in\N: m\geq 4,\, a\leq q_m\leq b\}$ for $a,\,b\in[0,+\infty]$. For the contribution of both components, it follows from (\ref{ineq: single kappa 1}) and~(\ref{ineq: single kappa 2}) that
    \begin{equation}\label{ineq: kappa_1}
        \begin{aligned}
        \sum_{m=4}^{+\infty}\Delta_mq_m\kappa_m^1
        &\leq\sum_{m =4}^{+\infty}\card\{j\in\Z/k\Z: r_j<q_m<r_j+b_j\}\cdot q_m\Delta_m\\
        &=\sum_{j\in\Z/k\Z}\sum_{m\in M(r_j+1,r_j+b_j-1)}q_m\Delta_m  
        \end{aligned}
    \end{equation}
and
    \begin{equation}\label{ineq: kappa_2}
        \begin{aligned}
            \sum_{m=4}^{+\infty}\kappa_m^2q_m\Delta_m&\leq \sum_{m=4}^{+\infty}\card\{j\in\Z/k\Z: r_j+b_j\leq q_m\leq r_{j-1}\}\cdot q_m\Delta_m\\
            &=\sum_{j\in\Z/k\Z}\sum_{m\in M( r_j+b_j,r_{j-1})} q_m\Delta_m.
        \end{aligned}
    \end{equation}
    Now fix $j\in\Z/k\Z$. Let $\gamma_1>1$ be the constant from (\ref{ineq: scale doubling condition}). We first deal with the contribution of type one components. Consider $m''\=\min (M(r_j+1,+\infty))$ and $m'\=\min (M(r_j+b_j+1,+\infty))$ as in the proof of Lemma~\ref{l: bad window estimates}. If $q_{m''}>r_j+b_j-1$, then $M(r_j+1,r_j+b_j-1)=\emptyset$ and $\sum_{m\in M(r_j+1,r_j+b_j-1)}q_m\Delta_m=0$. Now assume $q_{m''}\leq r_j+b_j-1$. Then by (\ref{ineq: bad window- remaining item for m modest part}), we see that
    \begin{equation*}
        \sum_{m\in M(r_j+1,r_j+b_j-1)}q_m\Delta_m\leq\sum_{m=m''}^{m'-1}\Delta_mq_m\leq \Delta_{m''}q_{m''}+3(\gamma_1-1)(\assocsum{2}{D_j}(1)-\assocsum{2}{r_j}(1)).
    \end{equation*}
    Note that for each integer $n$ satisfying $r_j+1\leq n\leq q_{m''}$, $\omega_n\geq\omega_{q_{m''}}\geq\omega_{q_{m''}+1}$. It then follows from $q_{m''}\leq r_j+b_j-1$ that 
    \begin{equation*}
        \begin{aligned}
            \assocsum{2}{D_j}(1)-\assocsum{2}{r_j}(1)&\geq \sum_{n=r_j+1}^{q_{m''}}\omega_n\geq (q_{m''}-r_j)\omega_{q_{m''}+1}\geq\frac{q_{m''}-r_j}{\gamma_1^2-1}\Delta_{m''},
        \end{aligned}
    \end{equation*}
    where the last inequality follows from (\ref{ineq: scale doubling condition}) and (\ref{eq: defi of Deltam}). Note that by Proposition~\ref{p: accident residual is bi-special}, there exists an integer $n_j\geq 3$ such that $r_j=q_{n_j}$. As $q_{m''}>r_j=q_{n_j}$, it follows from the property of the Fibonacci sequence that $q_{m''}-r_j\geq\frac{1}{3}q_{m''}$. 
    
    We conclude that 
    \begin{equation}\label{ineq: kappa1 contribution}
        \sum_{m\in M(r_j+1,r_j+b_j-1)}q_m\Delta_m\leq3(\gamma_1-1)(\gamma_1+2)(\assocsum{2}{D_j}(1)-\assocsum{2}{r_j}(1)).
    \end{equation}
    
    We then turn to the contribution of type two components. Let $c_3>0$ be the constant in Proposition~\ref{p: accident gap}. It follows from Proposition~\ref{p: accident gap} that $M(r_j+b_j,r_{j-1})\subseteq M(c_3(r_{j-1}+1),r_{j-1})$. Since the sequence $\{q_m\}_{m\geq 4}$ grows geometrically, it follows that there exists a constant $c''>0$ only dependent on $c_3$ such that $\card (M(r_j+b_j, r_{j-1}))\leq c''$. Note that by Proposition~\ref{p: accident gap}, we have for each $m\in M(r_j+b_j,r_{j-1})$,
    \begin{equation*}
         \assocsum{2}{D_{j}}(1)-\assocsum{2}{r_{j}}(1)=\sum_{n=r_j+1}^{r_j+b_j}\omega_n\geq b_j\omega_{q_{m}}\geq c_3(r_{j-1}+1)\omega_{q_m}\geq c_3q_m\omega_{q_m}. 
    \end{equation*}
    We conclude that
    \begin{equation}\label{ineq: kappa2 contribution}
        \sum_{m\in M(r_j+b_j,r_{j-1})} q_m\Delta_m\leq \frac{c''}{c_3}(\assocsum{2}{D_{j}}(1)-\assocsum{2}{r_{j}}(1)).
    \end{equation}
    
    Combining (\ref{ineq: kappa_1}), (\ref{ineq: kappa_2}), (\ref{ineq: kappa1 contribution}), and~(\ref{ineq: kappa2 contribution}), we see that (\ref{ineq: bad block estimate}) holds with $c_5\=3\bigl(\gamma_1^2+\gamma_1-2\bigr)+c''/c_3$.
\end{proof}

Let $m\geq 4$ be an integer. To estimate $s_m(H_m)$, we investigate the difference between $(\pi_m)_\#(H_{m+1})$ and $H_m$. Recall Definition~\ref{d: bad windows}.  Note that $U_m\subseteq U_{m+1}$. Set $G_m\=F_m-(\pi_{m})_\#(F_{m+1})$. By (\ref{eq: defi of good edge chain}), we see that 
\begin{equation}\label{eq: norm of Gm}
\norm{G_m}_m=N^{-1}\card(U_{m+1}\smallsetminus U_m).
\end{equation}
Recall (\ref{eq: defi of H_m}). We then have $H_m=(\pi_{m})_\#(H_{m+1})+\cE_m$, where $\cE_m\=G_m+(\pi_{m})_\#(Q_{m+1})-Q_m\in\sZ_m$.

\begin{prop}[Stable coordinate estimates I]
For each integer $m\geq 4$, we have
\begin{align}
    s_m(H_m)&=\bigl(\bigl(1-\sqrt{5}\bigr)\big/2\bigr)s_{m+1}(H_{m+1})+s_m(\cE_m)\label{eq: induc of s_m(H_m)}\quad\text{ and }\\
\abs{s_m(\cE_m)}&\leq 5 \bigl((NF_m)^{-1}\card(U_{m+1}\smallsetminus U_m)+
N^{-1}\kappa_m+N^{-1}\kappa_{m+1}\bigr).\label{ineq: control of stable coordinate of the tail}
\end{align}
\end{prop}
\begin{proof}
    Let $m\geq 4$ be an integer. Equality (\ref{eq: induc of s_m(H_m)}) immediately follows from Proposition~\ref{p: cycle properties}~(ii) and that $H_m=(\pi_{m})_\#(H_{m+1})+\cE_m$. Note that $\abs{s_m(H)}\leq F_{m-2}^{-1}\norm{H}_m$ for each $H\in\sZ_m$ (cf.~(\ref{eq: defi of two coordinates}), (\ref{eq: defi of mass and norm}), and Proposition~\ref{p: structure of Gamma_m}). Thus, $\abs{s_m(\cE_m)}\leq F_{m-2}^{-1}\norm{\cE_m}_m\leq F_{m-2}^{-1}(\norm{G_m}_m+\norm{Q_{m+1}}_{m+1}+\norm{Q_m}_m)$. This, together with Proposition~\ref{p: boundary repair}, (\ref{eq: norm of Gm}), and that $F_m\leq F_{m+1}\leq 5F_{m-2}$, yields (\ref{ineq: control of stable coordinate of the tail}).
\end{proof}

\begin{prop}[Stable coordinate estimates II]
\label{p: stable coordinate estimates II}
Let $\omega$ be a modulus of continuity satisfying the scale doubling condition and the stable condition. Then there exists a constant $c_6>0$ such that
\begin{equation}\label{ineq: stable coordinate estimates II}
\sum_{m= 4}^{+\infty}\Delta_mq_m\abs{s_m(H_m)}\leq\frac{c_6}{N}
\sum_{j=0}^{k-1}(\assocsum{2}{D_j}(1)-\assocsum{2}{r_j}(1)).
\end{equation}
\end{prop}
\begin{proof}
     Note that for sufficiently large $m\in\N$, $U_m=\Z/N\Z$ (cf.~Definition~\ref{d: bad windows}), and consequently $F_m=Q_m=H_m=G_m=0$. It then follows from (\ref{eq: induc of s_m(H_m)}) that 
    \begin{equation*}
        \abs{s_m(H_m)}\leq\sum_{l=0}^{+\infty}\frac
    {(\sqrt{5}-1)^l}{2^l}\abs{s_{m+l}(\cE_{m+l})},
    \end{equation*}
    and thus, 
    \begin{equation*}
    \begin{aligned}
    \sum_{m= 4}^{+\infty}\abs{s_m(H_m)}q_m\Delta_m
    &\leq\sum_{m=4}^{+\infty}\sum_{l=0}^{+\infty}\frac
    {(\sqrt{5}-1)^l}{2^l}q_m\Delta_m\abs{s_{m+l}(\cE_{m+l})}
    =\sum_{m=4}^{+\infty}\abs{s_{m}(\cE_m)}\sum_{k=4}^m\frac
    {(\sqrt{5}-1)^{m-k}}{2^{m-k}}q_k\Delta_k.
    \end{aligned}
    \end{equation*} 
    Let $\gamma_2>0$ be the constant from (\ref{ineq: stable condition}). By (\ref{ineq: stable condition}), we see that 
    \begin{equation*}
    \sum_{m=4}^{+\infty}\abs{s_m(H_m)}q_m\Delta_m\leq \gamma_2\sum_{m=4}^{+\infty}\abs{s_m(\cE_m)}q_m\Delta_m.
    \end{equation*}
     Now applying (\ref{ineq: control of stable coordinate of the tail}) and Lemma~\ref{l: bad block estimates}, we see that it suffices to control $\sum_{m= 4}^{+\infty}\Delta_m\cdot\card U_{m+1}$ and $\sum_{m=4}^{+\infty}q_m\Delta_m\kappa_{m+1}$. To apply Lemmas~\ref{l: bad window estimates} and~\ref{l: bad block estimates} for the two sums, respectively, we only need to prove the following result:
     
     \smallskip
     \emph{Claim.} There exists a constant $\oc>0$ independent of the periodic point $p$ such that for each integer $m\geq 4$, $\Delta_m\leq \oc\Delta_{m+1}$.
     
     \smallskip
     Let $\gamma_2>0$ be the constant from (\ref{ineq: stable condition}). It follows that for each integer $m\geq 4$, $\gamma_2q_{m+1}\Delta_{m+1}\geq \frac{\sqrt{5}-1}{2}q_{m}\Delta_m$. Since $q_{m+1}\leq 3q_m$ for all $m\in\N$ with $m\geq 4$, we see that the claim holds with $\oc\=\frac{3\sqrt{5}+3}{2}\gamma_2$, and~(\ref{ineq: stable coordinate estimates II}) holds with $c_6\=5\gamma_2(\oc c_4+c_5+\oc c_5)$, where $c_4$ and $c_5$ are the constants from Lemmas~\ref{l: bad window estimates} and~\ref{l: bad block estimates}, respectively.
\end{proof}

\begin{proof}[Proof of Proposition~\ref{p: transport gap}]
    Let $p$ be some periodic point of prime period $N\in\N$. Combining Propositions~\ref{p: cylinder discrepancy estimated through H_m}, \ref{p: reduction to bi-special scales}, and~\ref{p: stable coordinate estimates II}, and Lemmas~\ref{l: bad window estimates} and~\ref{l: bad block estimates}, we see that
    \begin{equation*}
        W_{\omega}(\eta_z,\mu_K)\leq N^{-1} (c_1+1)(c_4+2c_5+c_6)\sum_{j=0}^{k-1}(\assocsum{2}{D_j}(1)-\assocsum{2}{r_j}(1)),
    \end{equation*}
    where $c_1$, $c_4$, $c_5$, $c_6$ are the constants in Lemmas~\ref{l: comparison}, \ref{l: bad window estimates}, and~\ref{l: bad block estimates}, and Proposition~\ref{p: stable coordinate estimates II}, respectively. By Proposition~\ref{p: difference of nu_p and eta_z} and (\ref{ineq: estimate of bonus}), we conclude that (\ref{ineq: transport gap}) holds with $C\=1+(c_1+1)(c_4+2c_5+c_6)$, completing the proof. 
\end{proof}

\begin{proof}[\bf Proof of Theorem~\ref{t: non TPO}]
    Let $\omega$ be a modulus of continuity satisfying the scale doubling condition, the stable condition, and the Dini condition. By Remark~\ref{r: range of the technical conditions}, we see that such $\omega$ exists. Let $C>0$ be the constant from Theorem~\ref{t: energy transport estimate: general condition}. Then for each constant $\epsilon\in [0,1/C)$, each $g\in H^{\omega}(\Sigma,\R)$ with $\norm{g}_{\omega}\leq 1$, and each invariant measure $\nu\in\cM(\Sigma,\sigma)\smallsetminus\{\mu_K\}$,
    \begin{equation*}
        \begin{aligned}
        &\int_\Sigma\! (f_{K,\omega}+\epsilon g)\,\mathrm{d}\mu_K-\int_\Sigma\!(f_{K,\omega}+\epsilon g)\,\mathrm{d}\nu =\int_\Sigma\! -f_{K,\omega}\,\mathrm{d}\nu+\epsilon\biggl(\int_\Sigma\! g\,\mathrm{d}\mu_K-\int_\Sigma\! g\,\mathrm{d}\nu\biggr)\geq (1-\epsilon C)\int_\Sigma\! -f_{K,\omega}\,\mathrm{d}\nu>0,
        \end{aligned}
    \end{equation*}
    where the first inequality follows from Theorem~\ref{t: energy transport estimate: general condition} and the last inequality follows from the unique ergodicity of $K$ (cf.~Proposition~\ref{p: fibonacci-basic}). We conclude that for the open subset $U\=\{h\in H^{\omega}(\Sigma,\R):\norm{h-f_{K,\omega}}_{\omega}<1/C\}$ of $H^\omega(\Sigma,\R)$, $\mu_K$ is the unique $h$-maximizing measure for each $h\in U$.
\end{proof}

\begin{rem}\label{r: range of thmB}
    As can be seen from the above proof, the conclusion of Theorem~\ref{t: non TPO} holds for all moduli of continuity satisfying the scale doubling condition, the stable condition, and the Dini condition. 
\end{rem}

\section{Non-freezing for weak log-log Dini potentials}\label{s: non-freezing for log Dini}

In this section, we prove Theorem~\ref{t: non-freezing for log Dini} after the preparations in Lemmas~\ref{l: a general non-freezing criterion} to \ref{l: maximizing measure not much weight on nonmaximizing block}. In the proof of Theorem~\ref{t: non-freezing for log Dini}, we construct a sequence of invariant measures $\{\mu_{n_k}\}_{k\in\N}$ for which finer estimates hold so that Lemma~\ref{l: a general non-freezing criterion} applies. 

\begin{lemma}[A general non-freezing criterion]\label{l: a general non-freezing criterion}
Let $T\: X\to X$ be a continuous map on some compact metric space $(X,d)$, $\phi\: X\to\R$ be a continuous function, and $\mu\in\cM_{\max}(T,\phi)$. Assume that there exists a sequence $\{\mu_n\}_{n\in\N}$ in $\cM(X,T)$ such that $h_{\mu_n}(T)>h_{\mu}(T)$ for all sufficiently large $n\in\N$ and \footnote{Here we adopt the convention that $\frac{a}{0}\coloneqq+\infty$ for $a\in(0,+\infty)$ to include the case where $\int_X\!\phi\,\mathrm{d}\mu=\int_X\!\phi\,\mathrm{d}\mu_n$ for some $n\in\N$.}
\begin{equation} \label{eq: a general non-freezing criterion}
    \limsup_{n\to+\infty}\frac{h_{\mu_n}(T)-h_{\mu}(T)}{\int_{X}\!\phi\,\mathrm{d\mu-\int_X\!\phi\,\mathrm{d}\mu_n}}=+\infty.
\end{equation}
Then $\mu$ is not a frozen state for $\phi$.
\end{lemma}

\begin{proof}
	Fix $\beta>0$. It follows that there exists $n\in\N$ such that $h_{\mu_n}(T)-h_{\mu}(T)> \beta\bigl(\int_X\!\phi\,\mathrm{d}\mu-\int_X\!\phi\,\mathrm{d}\mu_n\bigr)$, i.e., $h_{\mu_n}(T)+\beta\int_X\!\phi\,\mathrm{d}\mu_n>h_{\mu}(T)+\beta\int_X\!\phi\,\mathrm{d}\mu$, which means that $\mu$ is not an equilibrium state for $\beta\phi$. We conclude that $\mu$ is not a frozen state for $\phi$.
\end{proof}

Lemmas~\ref{l: associated modulus rate} and~\ref{l: subseqeunce for log Dini} below provide an interpretation of the weak log-log Dini condition (cf.~(\ref{eq: key threshold})). 
\begin{lemma}\label{l: associated modulus rate}
    Let $\omega$ be a modulus of continuity and $\lambda>1$. Then for each $z>0$,
    \begin{equation}\label{eq: liminf from continuous r to discrete i}
         \liminf_{r\to 0^+}\frac{\int_r^1\!\frac{\omega(t)}{t}\,\mathrm{d}t}{\log(\log(r^{-1}))}=\liminf_{r\to 0^+}\frac{\int_r^z\!\frac{\omega(t)}{t}\,\mathrm{d}t}{\log(\log (r^{-1}))}=\log\lambda\cdot\liminf_{n\to+\infty}\frac{\assocsum{\lambda}{n}(z)}{\log n}.
    \end{equation}
\end{lemma}

\begin{proof}
	Fix $z>0$. It is straightforward to see that the first equality holds. Now we prove the second. Since $\omega$ is nondecreasing, we see that for each $i\in\N$, 
	\begin{equation}\label{ineq: interval integration for omega(t) over t}
		\omega\bigl(z\lambda^{-i}\bigr)\cdot\log \lambda\leq \int_{z\lambda^{-i}}^{z\lambda^{1-i}}\!\frac{\omega(t)}{t}\,\mathrm{d}t\leq \omega\bigl(z\lambda^{1-i}\bigr)\cdot\log \lambda.
	\end{equation}
    For each $n\in\N$, summing (\ref{ineq: interval integration for omega(t) over t}) over $i=1,\,2,\ldots,\,n$, we see from (\ref{eq: defi of associated modulus}) that  
    $\assocsum{\lambda}{n}(z) \cdot \log\lambda
    \leq\int_{z\lambda^{-n}}^{z}\!\frac{\omega(t)}{t}\,\mathrm{d}t
    \leq (\assocsum{\lambda}{n-1}(z)+\omega(z))\cdot\log\lambda$.
    Thus, for each integer $n>1+\frac{1+\log z}{\log\lambda}$ and each $r\in\bigl(z\lambda^{-n},z\lambda^{1-n}\bigr)$, 
    \begin{equation*}
        \frac{\assocsum{\lambda}{n-1}(z)\cdot\log\lambda}{\log(\log(\lambda^{n}/z))}
    \leq\frac{\int_{r}^z\!\frac{\omega(t)}{t}\,\mathrm{d}t}{\log(\log (r^{-1}))}
    \leq \frac{(\assocsum{\lambda}{n-1}(z)+\omega(z))\cdot\log\lambda}{\log(\log(\lambda^{n-1}/z))}.
    \end{equation*} 
    Now (\ref{eq: liminf from continuous r to discrete i}) follows. 
\end{proof}

\begin{lemma}\label{l: subseqeunce for log Dini}
    Let $\omega$ be a modulus of continuity satisfying the weak log-log Dini condition, $\lambda>1$, and $z>0$. Assume that $\{a_k\}_{k\in\N}$ is a sequence of strictly increasing positive integers satisfying $\delta\=\inf_{k\in\N}\frac{\log a_k}{\log a_{k+1}}>0$. Then 
    \begin{equation}\label{eq: subsequence for log Dini}
        \liminf_{k\to+\infty}\frac{\assocsum{\lambda}{a_k}(z)}{\log a_k}=0=\liminf_{n\to+\infty}\frac{\assocsum{\lambda}{n}(z)}{\log n}.
    \end{equation}
\end{lemma}
\begin{proof}
    As $\omega$ satisfies the weak log-log condition (cf.~(\ref{eq: key threshold})), the second equality in (\ref{eq: subsequence for log Dini}) follows from Lemma~\ref{l: associated modulus rate}. To prove the first equality in (\ref{eq: subsequence for log Dini}), we argue by contradiction. Assume that $\liminf_{k\to+\infty}\frac{\assocsum{\lambda}{a_k}(z)}{\log a_k}>0$. Then there exists a constant $c>0$ such that $\assocsum{\lambda}{a_k}(z)>c\log a_k$ for all $k\in\N$. As the sequence $\{a_k\}_{k\in\N}$ is strictly increasing, we see that for each $k\in\N$ and each integer $n\in [a_k,a_{k+1})$, 
    \begin{equation*}
        \frac{\assocsum{\lambda}{n}(z)}{\log n}\geq \frac{\assocsum{\lambda}{a_k}(z)}{\log a_{k+1}}>\frac{c\log a_k}{\log a_{k+1}}\geq c\delta>0,
    \end{equation*}
    which leads to a contradiction with the second equality in (\ref{eq: subsequence for log Dini}). Thus, the assumption does not hold, and the first equality in (\ref{eq: subsequence for log Dini}) holds.
\end{proof}

The following results, Lemmas~\ref{l: birkhoff sum controlled by asscosum} and~\ref{l: maximizing measure not much weight on nonmaximizing block}, provide a quantitative estimate on the weight that a maximizing measure assigns to some non-maximizing orbit, which is a key ingredient in the proof of Theorem~\ref{t: non-freezing for log Dini}. 

\begin{lemma}\label{l: birkhoff sum controlled by asscosum}
    Let $(X,\sigma)$ be a mixing $1$-step subshift of finite type on a finite alphabet $A$ equipped with the metric $d_{\lambda}$ for some $\lambda>1$, $s\in\N_0$, $\omega$ be a modulus of continuity, and $\phi\in H^{\omega}(X,\R)$ be an $\omega$-continuous function. Denote $M\=s\norm{\phi-Q(\sigma,\phi)}_{C^0}$. Assume that $\sigma^{s+1}([a])=X$ for all $a\in A$. Then for all $x\in X$ and $n\in\N$,
    \begin{equation*}
        S_n^{\sigma}\phi(x)-nQ(\sigma,\phi)\leq \abs{\phi}_{\omega}\cdot\assocsum{\lambda}{n}(1)+M.
    \end{equation*}
\end{lemma}

\begin{proof}
    Let $x=x_1x_2\ldots\in X$ and $n\in\N$, where $x_k\in A$ for all $k\in\N$. Since $\sigma^{s+1}([a])=X$ for all $a\in A$, there exists a word $c\in\fL_s(X)$ such that $x_ncx_1\in\fL(X)$. Thus, consider the periodic point $y\=x_1x_2\ldots x_ncx_1x_2\ldots x_nc\ldots\in X$. It immediately follows that $S_{n+s}^{\sigma}\phi(y)\leq (n+s)Q(\sigma,\phi)$ and $\abs{S_n^{\sigma}\phi(x)-S_{n}^{\sigma}\phi(y)}\leq \abs{\phi}_{\omega}\cdot\assocsum{\lambda}{n}(1)$. Combining the two inequalities, we see that $S_n^{\sigma}\phi(x)-nQ(\sigma,\phi)\leq \abs{\phi}_{\omega}\cdot\assocsum{\lambda}{n}(1)+M$. 
\end{proof}

To state the following lemma, we first linearly extend the function $n\mapsto\assocsum{\lambda}{n}(1)$ by 
\begin{equation*}
\assocsum{\lambda}{t}(1)\=(n+1-t)\assocsum{\lambda}{n}(1)+(t-n)\assocsum{\lambda}{n+1}(1)
\end{equation*}
for all $n\in\N_0$ and $t\in [n,n+1)$. Note that the function $\assocsum{\lambda}{\,\cdot\,}(1)\:[0,+\infty)\to[0,+\infty)$ is nondecreasing and concave. 
\begin{lemma}\label{l: maximizing measure not much weight on nonmaximizing block}
    Under the hypotheses of Lemma~\ref{l: birkhoff sum controlled by asscosum}, let $\mu\in\cM_{\max}(\sigma,\phi)$ be a $\phi$-maximizing measure. Assume that $y=y_1y_2\ldots\in X$ is a periodic point of period $L$, where $L\in\N$, and $a\=\frac{1}{L}S_L^{\sigma}\phi(y)<Q(\sigma,\phi)$. If $\mu([y_1y_2\ldots y_L])>0$, then 
    \begin{equation}\label{ineq: key estimate for nonmax block lemma}
    L(Q(\sigma,\phi)-a)\leq \abs{\phi}_{\omega}\assocsum{\lambda}{L}(1)+M+\abs{\phi}_{\omega}\assocsum{\lambda}{L/\mu([y_1y_2\ldots y_L])}(1).
    \end{equation}
    Consequently, if $L\leq N$ and $L(Q(\sigma,\phi)-a)>2\abs{\phi}_{\omega}\assocsum{\lambda}{N}(1)+M$ for some $N\in\N$, then 
    \begin{equation}\label{ineq: maximizing measure has controlled weight on nonmaximizing orbits}
    \mu([y_1y_2\ldots y_L])<L/N.
    \end{equation}
\end{lemma}

\begin{proof}

Denote $C\=[y_1y_2\ldots y_L]$, $\zeta\=\sigma^L$, and $U\=\bigcup_{n=0}^{+\infty}\zeta^{-n}(C)$. If $\mu(C)=0$, then (\ref{ineq: maximizing measure has controlled weight on nonmaximizing orbits}) naturally holds. Now assume that $\mu(C)>0$ and denote $\mu_C\=\frac{1}{\mu(C)}\mu|_{C}$. In the following, we first induce on $C$ and prove that the integral of the normalized potential $G$ with respect to $\mu_C$ vanishes. This, together with Lemma~\ref{l: birkhoff sum controlled by asscosum} yields (\ref{ineq: key estimate for nonmax block lemma}).

For $x\in X$, we define the return time $\tau_C(x)\=\inf\{n\in\N:\zeta^nx\in C\}$. Applying Kac's formula (cf.~\cite[Theorem~3.6]{Kre95}), we see that
\begin{equation}\label{eq: kac formula}
    \int_C\!\tau_C(x)\,d\mu(x)=\mu(U),
\end{equation}
and $\tau_C\in L^1(\mu_C)$. Now consider the return map $F\:C\to C$ given by $F(x)\= \zeta^{\tau_C(x)}(x)$ for $x\in C$ satisfying $\tau_C(x)<+\infty$. As $\mu$ is $\zeta$-invariant, by Poincar\'{e}'s recurrence theorem, $F$ is defined for $\mu$-a.e.\ $x\in C$. It is straightforward to check that $\mu_C$ is $F$-invariant (cf.~\cite[Theorem~3.7]{Kre95}).   

Set $R\=L\tau_C$ and $G(x):=\sum_{j=0}^{\tau_C(x)-1}(S_L^{\sigma}\phi-LQ(\sigma,\phi))\bigl(\zeta^j(x)\bigr)=\bigl(S_{R(x)}^{\sigma}\phi\bigr)(x)-R(x)Q(\sigma,\phi)$ for $x\in C$. It follows that $\abs{G}\leq L\norm{\phi-Q(\sigma,\phi)}_{C^0}\tau_C$, and thus $G\in L^1(\mu_C)$. For each $n\in\N$, denote $\tau_n(x)\=\sum_{i=0}^{n-1}\tau_C\bigl(F^i(x)\bigr)$ for $x\in C$. Since $\tau_C,\,G\in L^1(\mu_C)$ and $\mu_C$ is $F$-invariant, by Birkhoff's theorem (cf.~\cite[Theorem~4.5.5]{BrSt02}), the limits $\overline{\tau}(x)\=\lim_{n\to+\infty}\tau_n(x)/n$ and $\overline{G}(x)\=\lim_{n\to+\infty}\frac{1}{n}\sum_{i=0}^{n-1}G\bigl(F^i(x)\bigr)$ exist for $\mu_C$-a.e.\ $x\in C$ and $\overline{\tau},\,\overline{G}\in L^1(\mu_C)$. 

Similarly, $\phi_1(x)\=\lim_{n\to+\infty}\frac{1}{n}S_n^{\sigma}\phi(x)$ exists for $\mu$-a.e.\ $x\in X$, $\phi_1\in L^1(\mu)$, and $\int\!\phi_1\,\mathrm{d}\mu=\int\!\phi\,\mathrm{d\mu}=Q(\sigma,\phi)$. By \cite[Proposition~2.1]{Je06}, $\phi_1(x)\leq Q(\sigma,\phi)$ for $\mu$-a.e.\ $x\in X$. It then follows from $\int\!\phi_1\,\mathrm{d}\mu=Q(\sigma,\phi)$ that $\phi_1(x)=Q(\sigma,\phi)$ for $\mu$-a.e.\ $x\in X$. In other words, for $\mu$-a.e.\ $x\in X$, 
\begin{equation}\label{eq: Birkhoff average zero wrt maximizing measure}
\lim_{n\to+\infty}\frac{1}{n}(S_n^{\sigma}\phi(x)-nQ(\sigma,\phi))=0.
\end{equation}
We then calculate $\overline{G}$. Note that $\sum_{i=0}^{n-1}G\bigl(F^i(x)\bigr)=S_{L\tau_n(x)}^{\sigma}\phi(x)-L\tau_n(x)Q(\sigma,\phi)$ for $x\in C$ and $n\in\N$. Since $\tau_n\geq n$ for each $n\in\N$, we have $L\tau_n\to+\infty$. This, together with (\ref{eq: Birkhoff average zero wrt maximizing measure}) and $\overline{\tau}\in L^1(\mu_C)$, yields that for $\mu_C$-a.e.\ $x\in C$,
\begin{equation*}
    \overline{G}(x)=\lim_{n\to+\infty}\frac{1}{n}\sum_{i=0}^{n-1}G\bigl(F^i(x)\bigr)=\lim_{n\to+\infty}L\cdot\frac{\tau_n(x)}{n}
    \cdot\frac{S_{L\tau_n(x)}^{\sigma}\phi(x)-L\tau_n(x)Q(\sigma,\phi)}{L\tau_n(x)}=L\cdot\overline{\tau}(x)\cdot 0=0
\end{equation*}
Again by Birkhoff's theorem (cf.~\cite[Theorem~4.5.5]{BrSt02}), we see that $\int\!G\,\mathrm{d}\mu_C=\int\!\oG\,\mathrm{d}\mu_C=0$.

Now we apply Lemma~\ref{l: birkhoff sum controlled by asscosum} to estimate $G$. As $x\in C$ agrees with the periodic point $y$ in its first $L$ coordinates and $\phi\in H^\omega(X,\R)$, we see that for each $x\in C$,
\begin{equation}\label{ineq: S_L for x in C}
    S_L^{\sigma}\phi(x)\leq S_L^{\sigma}\phi(y)+\abs{\phi}_{\omega}\cdot\assocsum{\lambda}{L}(1)=La+\abs{\phi}_{\omega}\cdot\assocsum{\lambda}{L}(1).
\end{equation}
Since $R=L\tau_C\geq L$, it follows from (\ref{ineq: S_L for x in C}) and Lemma~\ref{l: birkhoff sum controlled by asscosum} that for each $x\in C$,
\begin{equation*}
    G(x)= S_{R(x)}^{\sigma}\phi (x)-R(x)Q(\sigma,\phi)\leq L(a-Q(\sigma,\phi))+\abs{\phi}_{\omega}\cdot\assocsum{\lambda}{L}(1)+\abs{\phi}_{\omega}\cdot\assocsum{\lambda}{R(x)-L}(1)+M.
\end{equation*}
Thus, 
\begin{equation}\label{ineq: estimate int G wrt muC}
    \begin{aligned}
    0=\int\!G\,\mathrm{d}\mu_C \leq L(a-Q(\sigma,\phi))+\abs{\phi}_{\omega}\cdot\assocsum{\lambda}{L}(1)
    +M+\abs{\phi}_{\omega}\int\!\assocsum{\lambda}{R(x)-L}(1)\,\mathrm{d}\mu_C.
    \end{aligned}
\end{equation}
As the function $\assocsum{\lambda}{\,\cdot\,}(1)\:[0,+\infty)\to[0,+\infty)$ is nondecreasing and concave, by Jensen's inequality and~(\ref{eq: kac formula}), we see from (\ref{ineq: estimate int G wrt muC}) that
\begin{equation*}
\begin{aligned}
    L(Q(\sigma,\phi)-a)
    &\leq \abs{\phi}_{\omega}\cdot\assocsum{\lambda}{L}(1)+M+\abs{\phi}_{\omega}\assocsum{\lambda}{\int\!(R(x)-L)\,\mathrm{d}\mu_C}(1)\\
    &=\abs{\phi}_{\omega}\cdot\assocsum{\lambda}{L}(1)+M+\abs{\phi}_{\omega}\assocsum{\lambda}{L\mu(U)/\mu(C)-L}(1)
    \leq \abs{\phi}_{\omega}\cdot\assocsum{\lambda}{L}(1)+M+\abs{\phi}_{\omega}\assocsum{\lambda}{L/\mu(C)}(1),
\end{aligned}
\end{equation*}
which proves (\ref{ineq: key estimate for nonmax block lemma}). Finally, (\ref{ineq: key estimate for nonmax block lemma}) together with the fact that $\assocsum{\lambda}{\,\cdot\,}(1)$ is nondecreasing yields (\ref{ineq: maximizing measure has controlled weight on nonmaximizing orbits}).
\end{proof}

We recall the following notions that will be used in the proof of Theorem~\ref{t: non-freezing for log Dini}. For a function $p\:S\to[0,1]$ on some finite set $S$ satisfying $\sum_{s\in S}p(s)=1$, we denote the \emph{entropy} of $p$ by 
\begin{equation*}
    H(p)\=-\sum_{s\in S}p(s)\log p(s). 
\end{equation*}
For a $\sigma$-invariant Borel probability measure $\mu$ on some subshift of finite type $(X,\sigma)$, we denote for each $n\in\N$,
\begin{equation*}
    H_n(\mu)\=-\sum_{w\in\fL_n(X)}\mu([w])\log\mu([w]).
\end{equation*}
As is well known, $\{H_n(\mu)\}_{n\in\N}$ is subadditive and $h_\mu(\sigma)=\inf\{H_n(\mu)/n : n\in\N\}$. Now we are ready to prove Theorem~\ref{t: non-freezing for log Dini}.

\begin{proof}[\bf Proof of Theorem~\ref{t: non-freezing for log Dini}]
    In this proof, we first set up a symbolic coding of $X$ from a Markov partition. Essentially, we need to deal with the case where there exists some non-maximizing periodic orbit for the potential $\phi$. Since the modulus of continuity $\omega$ satisfies the weak log-log condition, we apply Lemma~\ref{l: subseqeunce for log Dini} to find good scales $\{n_k\}_{k\in\N}$ for which Lemma~\ref{l: maximizing measure not much weight on nonmaximizing block} applies, which is the most delicate part of the proof. We then construct $\mu_n\in\cM(X,T)$ from some $\phi$-maximizing measure $\mu$ by sparsely inserting the word $J_n$ from the non-maximizing orbit into words of length $2n$ (cf.~(\ref{eq: defi of Nn ln Jn}) and (\ref{eq: defi of w(u,j)})). The use of separated positions for insertions (cf.~(\ref{eq: defi of In})) is in the spirit of the classical marker constructions (cf.~\cite{Boy83,Kri82}). More recently, Chazottes, Kucherenko, and Quas employed a similar but different argument to show that every equilibrium state for a continuous potential on $A^{Z^d}$ of summable variations is fully supported (cf.~\cite[Theorem~4.1]{CKQ25}). Finally, we estimate $h_{\mu_n}(T)-h_{\mu}(T)$ and $\int_X\!\phi\,\mathrm{d}\mu-\int_X\!\phi\,\mathrm{d}\mu_n$. Indeed, we have a good control for the entropy difference along the sequence $\{n_k\}_{k\in\N}$ (cf.~(\ref{ineq: entropy difference})) so that Lemma~\ref{l: a general non-freezing criterion} applies to $\{\mu_{n_k}\}_{k\in\N}$.

    Now we start the proof. Assume that $(\lambda,\gamma,\xi)$ be expanding constants of $T\:X\to X$. Take a Markov partition $\cR=\{R_1,\,R_2,\,\ldots,\,R_k\}$ of $X$ for $T$, where $k\in\N$ (cf.~\cite[Theorem~4.5.2]{PU10}). Consider the (symbolic) coding $\pi\:(X',d_{\lambda})\to (X,d)$ given by the Markov partition $\cR$, where $X'$ is a $1$-step subshift of finite type on the alphabet $\{1,\,2,\,\ldots,\,k\}$.  It is well known that $\pi$ is surjective and Lipschitz continuous (cf.~\cite[Theorem~4.5.7]{PU10}), and the induced map $\pi_{*}\: \cM(X',\sigma)\to\cM(X,T)$ is surjective. Denote the Lipschitz constant of $\pi$ by $\Lip\pi$. Since $(X,T)$ is mixing, $(X',\sigma)$ is also mixing. Thus, assume that for $s\in\N_0$, $\sigma^{s+1}([i])=X'$ for all $i\in\{1,\,2,\,\ldots,\,k\}$.

    Let $\phi\in H^{\omega}(X,\R)$ be an $\omega$-continuous potential. If $\int\!\phi\,\mathrm{d}\mu=\energy$ for all $\mu\in\cM(X,T)$, then the equilibrium state for $\beta\phi$ is equal to the measure of maximal entropy for $X$ for each $\beta\in(0,+\infty)$, and thus $\phi$ does not admit a freezing phase transition. Now assume that there exists $\mu\in\cM(X,T)$ such that $\int\!\phi\,\mathrm{d}\mu<\energy$. Denote $\widetilde{\phi}\=\phi\circ\pi$. As $\pi_*\:\cM(X',\sigma)\to\cM(X,T)$ is surjective, we see that $Q\bigl(\sigma,\widetilde{\phi}\bigr)=\energy$ and there exists $\mu'\in\cM(X',\sigma)$ such that $\int\!\widetilde{\phi}\,\mathrm{d}\mu'<\energy$. Since periodic measures are weak$^*$ dense in $\cM(X',\sigma)$ (cf.~\cite{Sig74}), there exists a periodic point $y\in X'$ of periodic $l\in\N$ such that $a\=\frac{1}{l}S_l^{\sigma}\widetilde{\phi}(y)<\energy$.

    Now to show that $\phi$ does not admit a freezing phase transition, it suffices to show that for each $\phi$-maximizing measure $\mu$, $\mu$ is not a frozen state for $\phi$. Fix $\mu\in\cM_{\max}(T,\phi)$ and $\mu'\in\pi_*^{-1}(\mu)$. For each $n\in\N$, set
    \begin{align}
    N_n\=n^3, \quad 
    l_n&\=l\big\lceil(\log n+1)^{3/4}\cdot(\assocsum{\lambda}{N_n}(\Lip\pi)+1)^{1/4}\big\rceil\in\N, \quad
    J_n\=y_1y_2\ldots y_{l_n}\in\fL(X'), \quad\text{ and}\label{eq: defi of Nn ln Jn}\\
    I_n&\=\{k(l_n+2s+1):k\in\N,\,k(l_n+2s+1)<2n\}.\label{eq: defi of In}
    \end{align}
    Since $\lim_{t\to 0^+}\omega(t)=\omega(0)=0$, we see that $\lim_{n\to+\infty}l_n/n=0$ and $I_n\neq\emptyset$ for all sufficiently large $n\in\N$. As $\omega$ is weak log-log Dini (cf.~(\ref{eq: key threshold})), it follows from Lemma~\ref{l: associated modulus rate} that $\liminf_{n\to+\infty}\frac{\assocsum{\lambda}{n}(\max{1,\,\Lip\pi})}{\log n}=0$. Applying Lemma~\ref{l: subseqeunce for log Dini} to the sequence $\{N_{n+1}\}_{n\in\N}$, we see that there exists a strictly increasing sequence $\{n_k\}_{k\in\N}$ of positive integers such that 
    \begin{equation}\label{eq: choice of nk}
        \lim_{k\to+\infty}\assocsum{\lambda}{N_{n_k}}(\max\{1,\,\Lip\pi\})\big/\log n_k=0,
    \end{equation}
    and that for each $k\in\N$, $I_{n_k}\neq\emptyset$, $l_{n_k}\leq N_{n_k}$ and $l_{n_k}\bigl(Q\bigl(\sigma,\widetilde{\phi}\bigr)-a\bigr)>2\abs{\phi}_{\omega}\assocsum{\lambda}{N_{n_k}}(\Lip\pi)+s\Normbig{\widetilde{\phi}-Q\bigl(\sigma,\widetilde{\phi}\bigr)}_{C^0}$.
    Define $\overline{\omega}(t)\=\omega(t\cdot\Lip\pi)$ for all $t\in [0,+\infty)$. Note that $\widetilde{\phi}$ is $\overline{\omega}$-continuous. Applying Lemma~\ref{l: maximizing measure not much weight on nonmaximizing block} to $\widetilde{\phi}$ and $\overline{\omega}$, we see that for each $k\in\N$, $\mu'([J_{n_k}])<l_{n_k}/N_{n_k}$.

    In the following, we construct $\mu_n$ for all $n\in\N$ satisfying $I_n\neq\emptyset$ and apply Lemma~\ref{l: a general non-freezing criterion} to $\{\mu_{n_k}\}_{k\in\N}$. Now fix $n\in\N$ for which $I_n\neq\emptyset$.

    For each $u=u_1u_2\ldots u_{2n}\in\fL_{2n}(X')$ and each $j\in I_n$, we fix a choice of $c_j^{-}(u)$ and $c_j^{+}(u)$ in $\fL_{s}(X')$ such that $u_jc_j^{-}(u)J_nc_{j}^+(u)u_{j+1}\in\fL(X')$ and thus define a new word $w(u,j)\in \fL_{2n+l_n+2s}(X')$ by
    \begin{equation} \label{eq: defi of w(u,j)}
    	w(u,j)\=u_1u_2\ldots u_jc_j^{-}(u)J_nc_j^{+}(u)u_{j+1}u_{j+2}\ldots u_{2n}.
    \end{equation} 
Denote $B_n\=\{w(u,j):u\in\fL_{2n}(X'),\,j\in I_n\}$. Consider the probability $p_n$ on $\fL_{2n}(X')\times I_n$ given by $p_n(u,j)\=\mu'([u])/\card I_n$ for all $u\in\fL_{2n}(X')$ and $j\in I_n$, and the pushforward $p_n'$ of $p_n$ on $B_n$ through $(u,j)\mapsto w(u,j)$. 

    We then construct an $\sigma$-invariant measure $\nu_n$ with $p_n'$. Write $B_n=\{b_1,\,b_2,\,\ldots,\,b_{\card B_n}\}$. Fix a choice of $C_{ij}\in\fL_s(X')$ such that $b_iC_{ij}b_j\in\fL(X')$. Then consider $\pi_n\:\{1,\,2,\,\ldots,\,\card  B_n\}^{\N}\to X'$ given by 
    \begin{equation*}
       \pi_n\bigl( \{x_k\}_{k=1}^{+\infty}\bigr)\=b_{x_1}C_{x_1x_2}b_{x_2}C_{x_2x_3}b_{x_3}\ldots
    \end{equation*}
    and denote the image of $\pi_n$ by $\im(\pi_n)$. It is straightforward to check that $\bigl(\im(\pi_n),\sigma^{2n+l_n+3s}\bigr)$ is topologically conjugate to the full shift on $\{1,\,2,\,\ldots,\,\card  B_n \}^{\N}$. Thus, we consider the Bernoulli measure $\nu_n^b$ on $\{1,\,2,\,\ldots,\,\card B_n \}^{\N}$ induced by $p_n'$ and define
    \begin{equation*}
        \nu_n\=\frac{\sum_{k=1}^{2n+l_n+3s}\nu_n^b\circ\pi_n^{-1}\circ\sigma^{-k}}{2n+l_n+3s}.
    \end{equation*}
    As $\nu_n^b$ is $\sigma^{2n+l_n+3s}$-invariant, we see that $\nu_n$ is $\sigma$-invariant. 
    
    Take $\mu_n\=\nu_n\circ\pi^{-1}$. We then estimate $h_{\mu_n}(T)-h_{\mu}(T)$ and $\int_X\!\phi\,\mathrm{d}\mu-\int_X\!\phi\,\mathrm{d}\mu_n$. 
    
   We first deal with the entropy difference. As is well known, $\sup_{x\in X}\card\bigl(\pi^{-1}(x)\bigr)<+\infty$ (cf.~\cite[Theorem~4.5.7]{PU10}), and $h_{\nu}(\sigma)=h_{\nu\circ \pi^{-1}}(T)$ for each $\nu\in\cM(X',\sigma)$ (cf.~\cite[Theorem~2.1]{LW77}). Thus, 
   \begin{equation*}
   h_{\mu_n}(T)
   =h_{\nu_n}(\sigma)=h_{\nu_n^b\circ\pi_n^{-1}}\bigl(\sigma^{2n+l_n+3s}\bigl)\big/(2n+l_n+3s)
   =H(p_n')/(2n+l_n+3s).
   \end{equation*}
   As $p_n'$ is the pushforward of $p_n$ through $\fL_{2n}(X')\times I_n\to B_n$, $(u,j)\mapsto w(u,j)$, we define for each $b\in B_n$, the number $m(b)$ of preimages of $b$. For each $w\in\fL_{2n}(X')$, let $m'(w)$ be the number of $J_n$'s in $w$, i.e., \begin{equation*}
   m'(w)\=\card\{j \in \N : j\leq 2n+1-l_n,\, w_jw_{j+1}\ldots w_{j+l_n-1}=J_n\}.
   \end{equation*}
   By (\ref{eq: defi of In}), we see that $m(w(u,j))\leq 1+m'(u)$ for $(u,j)\in\fL_{2n}(X')\times I_n$. It follows that 
   \begin{equation*}
   \begin{aligned}
   H(p_n')&\geq H(p_n)-\sum_{b\in B_n}p_n'(b)\log (m(b))\\
   &\geq H_{2n}(\mu')+\log(\card I_n)-\sum_{u\in \fL_{2n}(X')}\sum_{j\in I_n}p_n(u,j)\log (1+m'(u))\\
   &=H_{2n}(\mu')+\log(\card I_n)-\sum_{u\in\fL_{2n}(X')}\mu'([u])\log (1+m'(u))\\
   &\geq H_{2n}(\mu')+\log(\card I_n)-\log\Bigl(1+\sum_{u\in\fL_{2n}(X')}\mu'([u])m'(u)\Bigr),
   \end{aligned}
   \end{equation*}
   where the last inequality follows from Jensen's inequality. As $\mu'$ is $\sigma$-invariant, we see that 
   \begin{equation*}
   \sum_{u\in\fL_{2n}(X')}\mu'([u])m'(u)=\int_{X'}\!\sum_{j=1}^{2n+1-l_n}\chi_{\{x\in X':x_jx_{j+1}\ldots x_{j+l_n-1}=J_n\}}\,\mathrm{d}\mu'=(2n+1-l_n)\mu'([J_n]),
   \end{equation*}
   where $\chi_{S}\:X'\to\R$ denotes the characteristic function of some subset $S$ of $X'$. Recall that for each $k\in\N$, $\mu'([J_{n_k}])<l_{n_k}/N_{n_k}=l_{n_k}\big/n_k^3$. We conclude that for each $k\in\N$,
   \begin{equation}\label{ineq: entropy difference}
        \begin{aligned}
       h_{\mu_{n_k}}(T)-h_{\mu}(T)&\geq \frac{2n_kh_{\mu'}(\sigma)+\log (\card I_{n_k})-\log (1+3l_{n_k}n_k^{-2})}{2n_k+l_{n_k}+3s}-h_{\mu}(T)\\
       &\geq \frac{-(l_{n_k}+3s)h_{\mu}(T)+\log n_k-\log (l_{n_k}+2s+1)-\log(1+3l_{n_k}n_k^{-2})}{2n_k+l_{n_k}+3s}.
       \end{aligned}
   \end{equation}
    It then follows from (\ref{eq: defi of Nn ln Jn}) and (\ref{eq: choice of nk}) that for all sufficiently large $k\in\N$, $h_{\mu_{n_k}}(T)>h_{\mu}(T)$.
    
    For the energy, we have
    \begin{equation}\label{eq: mun energy in terms of integral with Bernoulli}
        \begin{aligned}
            \int_X\!\phi\,\mathrm{d}\mu_n=\int_{X'}\!\widetilde{\phi}\,\mathrm{d}\nu_n=\frac{1}{2n+l_n+3s}\int_{X'}\!S_{2n+l_n+3s}^{\sigma}\widetilde{\phi}\,\mathrm{d}\bigl(\nu_n^b\circ\pi_n^{-1}\bigr).
        \end{aligned}
    \end{equation}
     Recall that $\pi\:(X',d_\lambda)\to(X,d)$ is Lipschitz. It follows that for all $x,\,y\in X'$, $d(\pi(x),\pi(y))\leq \Lip\pi\cdot d_{\lambda}(x,y)$. Denote $z\=\Lip\pi$. It follows that for $x\in [w(u,j)]\subseteq X'$ and $y\in[u]\subseteq X'$, where $u\in\fL_{2n}(X')$ and $j\in I_n$,
    \begin{equation*}
        \begin{aligned}
        \Absbig{S_{2n+l_n+3s}^{\sigma}\widetilde{\phi}(x)-S_{2n}^{\sigma}\widetilde{\phi}(y)}
        &\leq (l_n+3s)\norm{\phi}_{C^0}+\abs{\phi}_{\omega}\cdot\sup_{j\in I_n}\biggl(\sum_{k=1}^j\omega\bigl(z\lambda^{-k}\bigr)+\sum_{k=1}^{2n-j}\omega\bigl(z\lambda^{-k}\bigr)\biggr)\\
        &\leq (l_n+3s)\norm{\phi}_{C^0}+2\abs{\phi}_{\omega}\assocsum{\lambda}{n}(z).
        \end{aligned}
    \end{equation*}
    Thus, for $x\in[w(u,j)]\subseteq X'$ where $u\in\fL_{2n}(X')$ and $j\in I_n$,
    \begin{equation*}
        \Absbigg{\mu'([u])S_{2n+l_n+3s}^{\sigma}\widetilde{\phi}(x)-\int_{[u]}\!S_{2n}^{\sigma}\widetilde{\phi}(y)\,\mathrm{d}\mu'(y)}\leq \mu'([u])((l_n+3s)\norm{\phi}_{C^0}+2\abs{\phi}_{\omega}\assocsum{\lambda}{n}(z)).
    \end{equation*}
     This, together with $\sum_{j\in I_n}p_n(u,j)=\mu'([u])$ for all $u\in\fL_{2n}(X')$, yields
     \begin{equation*}
         \begin{aligned}
             &\Absbigg{\int_{X'}\!  S_{2n+l_n+3s}^{\sigma}\widetilde{\phi}\,\mathrm{d}\bigl(\nu_n^b\circ\pi_n^{-1}\bigr)-\int_{X'}\!S_{2n}^{\sigma}\widetilde{\phi}\,\mathrm{d}\mu'}\\
             &\qquad=\Absbigg{\sum_{\substack{(u,j)\in\fL_{2n}(X')\times I_n\\ p_n(u,j)>0}}\biggl(\frac{p_n(u,j)}{p_n'(w(u,j))}\int_{[w(u,j)]}\!S_{2n+l_n+3s}^{\sigma}\widetilde{\phi}\,\mathrm{d}\bigl(\nu_n^b\circ\pi_n^{-1}\bigr)-\frac{p_n(u,j)}{\mu'([u])}\int_{[u]}\!S_{2n}^{\sigma}\widetilde{\phi}\,\mathrm{d}\mu'\biggr)}\\
             &\qquad\leq \sum_{u\in\fL_{2n}(X')}\sum_{j\in I_n}p_n(u,j)((l_n+3s)\norm{\phi}_{C^0}+2\abs{\phi}_{\omega}\assocsum{\lambda}{n}(z))=(l_n+3s)\norm{\phi}_{C^0}+2\abs{\phi}_{\omega}\assocsum{\lambda}{n}(z).
         \end{aligned}
     \end{equation*}
     This, together with (\ref{eq: mun energy in terms of integral with Bernoulli}) and $\mu'\circ\pi^{-1}=\mu$, yields that
     \begin{equation}\label{ineq: energy difference}
       \Absbigg{\int_X\!\phi\,\mathrm{d}\mu_n-\int_X\!\phi\,\mathrm{d}\mu}\leq \frac{2(l_n+3s)\norm{\phi}_{C^0}+2\abs{\phi}_{\omega}\assocsum{\lambda}{n}(z)}{2n}.
     \end{equation}
    Now combining (\ref{ineq: entropy difference}) and~(\ref{ineq: energy difference}), we see that 
    \begin{align}\label{ineq: entropy energy comparison in proof}
            &\limsup_{k\to+\infty}\frac{h_{\mu_{n_k}}(T)-h_{\mu}(T)}{\int_{X}\!\phi\,\mathrm{d}\mu-\int_X\!\phi\,\mathrm{d}\mu_{n_k}}\\
            &\geq\limsup_{k\to+\infty}\frac{2n_k}{2n_k+l_{n_k}+3s}\frac{-(l_{n_k}+3s)h_{\mu}(T)+\log n_k-\log (l_{n_k}+2s+1)-\log (1+3l_{n_k}n_k^{-2})}{2(l_{n_k}+3s)\norm{\phi}_{C^0}+2\abs{\phi}_{\omega}\assocsum{\lambda}{n_k}(z)}
            =+\infty, \notag
    \end{align}
    where the last equality follows from the fact that $\lim_{k\to+\infty}\assocsum{\lambda}{n_k}(\max\{1,\,z\})/\log n_k=0=\lim_{k\to+\infty}\frac{l_{n_k}}{\log n_k}$, which is a consequence of (\ref{eq: choice of nk}) and the monotonicity of $\assocsum{\lambda}{\,\cdot\,}(\max\{1,\,z\})$. Recall that $h_{\mu_{n_k}}>h_{\mu}$ for all sufficiently large $k\in\N$. It now follows from (\ref{ineq: entropy energy comparison in proof}) and Lemma~\ref{l: a general non-freezing criterion} that $\mu$ is not a frozen state for $\phi$. The proof is now complete.
\end{proof}

\section{Freezing: construction for a general modulus}\label{s: freezing for non log Dini}

In this section, we prove Theorem~\ref{t: freezing for not log Dini} by constructing a subshift of positive entropy. A general freezing criterion is established in Subsection~\ref{ss: freezing criterion for positive h}. The specific construction and the proof of Theorem~\ref{t: freezing for not log Dini} are given in Subsection~\ref{ss: construction of the critical subshift}.

\subsection{A freezing criterion for subshift of positive entropy}\label{ss: freezing criterion for positive h}

Let $b_K(m)$ denotes the number of minimal forbidden words of length $m$ for some subshift $K$ and some $m\in\N$. The following result relates the growth of $\{\card(\fL_m(K))\}_{m\in\N}$ and $\{b_K(m)\}_{m\in\N}$ to freezing phase transitions. 

\begin{lemma}\label{l: freezing criterion for positive entropy subshift}
    Let $A$ be a finite alphabet, $K$ be a proper subshift of $\bigl(A^{\N},\sigma\bigr)$, and $h\=h_{\topo}(\sigma|_K)$. Assume that $h>0$ and that there exists constants $N_1\geq 1$, $N_2,\,D>0$, and $ \gamma\in [0,h)$ such that for each $m\in\N$,
    \begin{equation}\label{ineq: complexity condition}
        \card(\fL_{m}(K))\leq N_1(m+1)^De^{hm}\quad\text{ and }\quad b_K(m)\leq N_2e^{\gamma m}.
    \end{equation}
    Then for each modulus of continuity $\omega$ not satisfying the weak log-log Dini condition in (\ref{eq: key threshold}), there exists a constant $B>0$ such that for each number $\beta>B$, $\fE(\sigma,-\beta\omega(d_2(\,\cdot\,,K)))=\MME(K)$, and consequently, the plateau potential $-\omega(d_2(\,\cdot\,,K))$ admits a freezing phase transition. 
\end{lemma}

\begin{proof}

 Let $\omega$ be a modulus of continuity not satisfying the weak log-log Dini condition. It follows from Lemma~\ref{l: associated modulus rate} that there exists a constant $c>0$ such that for each $N\in\N$,
\begin{equation}\label{ineq: assoc sum growth lower}
 \assocsum{2}{N}(1)\geq c\log(N+1).
\end{equation}
Since $\lim_{t\to 0^+}\omega(t)=\omega(0)=0$, we see that $\lim_{N\to+\infty}\assocsum{2}{N}(1)/N=0$.

In the remainder of the proof, we decompose each word $w$ on $A$ into bad phrases recording successive failures of membership in $\fL(K)$ and ending in minimal forbidden words. This left-to-right decomposition has multiplicities controlled by (\ref{ineq: complexity condition}) and corresponding energy costs controlled by (\ref{ineq: assoc sum growth lower}). After normalization by $e^{hN}$, the resulting phrase weights form summable kernels $Q_\beta$ satisfying $\lim_{\beta\to+\infty}Q_{\beta}=0$, which yields freezing through a calculation of partition functions.

Now we decompose some word $w$ on $A$ recursively from left to right. Recall that $\epsilon$ denotes the empty word. For the remaining suffix $v$ of $w$, the decomposition stops if $v\in\fL(K)\cup\{\epsilon\}$ and otherwise we consider the shortest illegal prefix of $v$ and its shortest forbidden suffix $f$ so that the illegal prefix can be written as $zf$ and $v=zfv'$, where $v'$ is the remaining suffix. It is straightforward to see that $z\in\fL(K)\cup\{\epsilon\}$ and $f$ is minimal forbidden. Then we repeat this decomposition for $v'$. Since $\abs{w}<+\infty$, the decomposition process will stop and $w$ is decomposed as a unique finite sequence of bad phrases $z_if_i$ followed by at most one terminal phrase in $\fL(K)$.

For some bad phrase $zf$ of $w$, set $r\=|z|\in\N_0$ and $m\=|f|\in\N$. Suppose that the phrase $zf$ begins at the $s$-th coordinate of $w$. Then for each integer $t\in [0,r]$, the
word beginning at $(s+t)$-th coordinate of $w$ and ending with the phrase $zf$ is illegal and has length
$m+r-t$.  Thus, for each $x\in[w]$ and each integer $t\in[0,r]$, $\omega\bigl(d\bigl(\sigma^{s+t-1}(x),K\bigr)\bigr)\geq \omega\bigl(2^{1+t-m-r}\bigr)$. Consequently, the bad phrase $zf$ contributes the energy cost
\begin{equation}\label{eq: phrase energy cost}
 S(m,r)\=\sum_{t=0}^{r}\omega\bigl(2^{1+t-m-r}\bigr)=\assocsum{2}{m+r-1}(1)-\assocsum{2}{m-1}(1)+\omega\bigl(2^{1-m}\bigr).
\end{equation}

Note that for a fixed pair $(m,r)\in\N\times\N_0$, it follows from (\ref{ineq: complexity condition}) that\footnote{Here we adopt $\fL_0(K)\coloneqq\{\epsilon\}$, where $\epsilon$ is the empty word, to include the $r=0$ case.}
\begin{equation}\label{ineq: phrase-count}
\begin{aligned}
 &\card\{zf:z\in\cL_r(K),\, |f|=m,\,f\text{ is minimal forbidden}\}\\
 &\qquad=\card(\cL_r(K))\cdot b_K(m)
 \leq N_1N_2(r+1)^D e^{hr+\gamma m}.
\end{aligned}
\end{equation}
Set $\delta\=h-\gamma>0$. Consider for each $\beta>0$, the normalized phrase kernel
\begin{equation}\label{eq:kernel}
 Q_\beta\=N_1N_2\sum_{m=1}^{+\infty}e^{-\delta m}
 \sum_{r=0}^{+\infty}(r+1)^D e^{-\beta S(m,r)}.
\end{equation}
By~\eqref{ineq: assoc sum growth lower} and~\eqref{eq: phrase energy cost}, we see that $S(m,r)\geq c\log(m+r)-\assocsum{2}{m-1}(1)+\omega\bigl(2^{1-m}\bigr)$ for all $m\in\N$ and $r\in\N_0$.

Fix $\beta_*>0$ such that $\beta_*c>D+1$.  There is a constant
$C_*\in (0,+\infty)$ such that for each $m\in\N$,
\begin{align*}
 \sum_{r=0}^{+\infty}\frac{(r+1)^D}{ e^{\beta_*S(m,r)}}
 \leq e^{\beta_*(\assocsum{2}{m-1}(1)-\omega(2^{1-m}))}
 \sum_{r=0}^{+\infty}\frac{(r+1)^D}{(m+r)^{\beta_*c}}\leq C_*e^{\beta_*(\assocsum{2}{m-1}(1)-\omega(2^{1-m}))}.
\end{align*}
Recall that $\lim_{N\to+\infty}\assocsum{2}{N}(1)/N=0$. It follows that there exists $m_*\in\N$ such that for all intgers $m\geq m_*$,
\begin{equation*}
 e^{-\delta m+\beta_*(\assocsum{2}{m-1}(1)-\omega(2^{1-m}))}\leq e^{-\delta m/2}.
\end{equation*}
Hence $Q_{\beta_*}<+\infty$. For $m\in\N$, $r\in\N_0$, and $\beta\geq\beta_*$, consider
\begin{equation*}
 F_\beta(m,r)\=N_1N_2e^{-\delta m}(r+1)^D e^{-\beta S(m,r)}.
\end{equation*}
It is straightforward to see that for each pair of $(m,r)\in\N\times\N_0$, $F_{\beta}(m,r)$ is nonincreasing with respect to  $\beta$ and $\lim_{\beta\to+\infty}F_{\beta}(m,r)=0$. It then follows from the dominated convergence
theorem that
\begin{equation*}
 \lim_{\beta\to+\infty}Q_\beta=0.
\end{equation*}

Choose $B\geq\beta_*$ such that $Q_B<1$. Now we combine the decomposition of words discussed earlier with the kernel $Q_\beta$ to analyze the set $\fE(\sigma,-\beta\omega(d_2(\,\cdot\,,K)))$ of equilibrium states. Let $N\in\N$ and  $\beta\geq B$. Consider the $N$-th partition function
\begin{equation*}
 Z_N(\beta)\=\sum_{w\in A^N}
 \exp\biggl(\beta\sup_{x\in[w]}\sum_{j=0}^{N-1}
 \bigl(-\omega\bigl(d_2\bigl(\sigma^j(x),K\bigr)\bigr)\bigr)\biggr).
\end{equation*}
After multiplying $e^{-hN}$, we see that bad phrases $zf$ with $(\abs{z},\abs{f})=(r,m)$ contributes at most $N_1N_2e^{-\delta m}(r+1)^De^{-\beta S(m,r)}$ (cf.~\ref{eq:kernel}), while the terminal legal word $u$ with $\abs{u}=l$ contributes at most $e^{-hl}\card(\cL_l(K))\leq N_1(N+1)^D$ (cf.~(\ref{ineq: complexity condition})).  Summing over the number $k$ of bad phrases, we see from (\ref{ineq: phrase-count}) and~(\ref{eq:kernel}) that
\begin{equation*}
\begin{aligned}
 e^{-hN}Z_N(\beta)&\leq N_1(N+1)^D\sum_{k=0}^{+\infty}Q_\beta^k\notag\leq (1-Q_\beta)^{-1}N_1(N+1)^D.
\end{aligned}
\end{equation*}
It follows that $P(\sigma,-\beta\omega(d_2(\,\cdot\,,K)))\leq h$. Meanwhile, if $\nu\in\MME(K)$, then
$\int\!\omega(d_2(\,\cdot\,,K))\,d\nu=0$, and the variational principle implies
$P(\sigma,-\beta\omega(d_2(\,\cdot\,,K)))\geq h$. We conclude that $P(\sigma,-\beta\omega(d_2(\,\cdot\,,K)))=h$ and $\MME(K)\subseteq \fE(\sigma,-\beta\omega(d_2(\,\cdot\,,K)))$.

Finally, consider $\mu\in\fE(\sigma,-\beta\omega(d_2(\,\cdot\,,K)))$ for some $\beta>B$ and $I_\mu\=\int\!\omega(d_2(\,\cdot\,,K))\,d\mu$. Since 
\begin{equation*}
P(\sigma,-\beta\omega(d_2(\,\cdot\,,K)))=h=P(\sigma,-B\omega(d_2(\,\cdot\,,K))),
\end{equation*}
we see that $h_\mu(\sigma)-\beta I_\mu=h\geq h_\mu(\sigma)-B I_{\mu}$. Since $I_\mu\geq0$ and $\beta>B$, we conclude that $I_\mu=0$ and $\mu\in\MME(K)$. Hence, $\MME(K)=\fE(\sigma,-\beta\omega(d_2(\,\cdot\,,K)))$ for all $\beta>B$. Further consider 
\begin{equation*}
\beta_c\=\inf\{\beta\in(0,+\infty):\fE(\sigma,-\beta\omega(d_2(\,\cdot\,,K)))=\MME(K)\}.
\end{equation*}
As $K$ is a proper subshift of $A^{\N}$, we see that $\beta_c>0$. Note that if 
\begin{equation*}
\fE(\sigma,-\beta_1\omega(d_2(\,\cdot\,,K)))=\fE(\sigma,-\beta_2\omega(d_2(\,\cdot\,,K)))=\MME(K)
\end{equation*}
for some $0<\beta_1<\beta_2$, then $\fE(\sigma,-\beta\omega(d_2(\,\cdot\,,K)))=\MME(K)$ for all $\beta\in[\beta_1,\beta_2]$. Thus, $-\omega(d_2(\,\cdot\,,K))$ admits a freezing phase transition. 
\end{proof}

\subsection{Construction}\label{ss: construction of the critical subshift}

Here we prepare several technical results for the proof of Proposition~\ref{p: the freezing subshift}, where we inductively delete periodic orbits. Lemmas~\ref{l: finite-presentation} and~\ref{l: entropy-isolation} are used to verify the complexity conditions in Lemma~\ref{l: freezing criterion for positive entropy subshift}. 

The following lemma is the key ingredient in the inductive construction process. While statement~(ii) follows from \cite[Theorem~3]{Li89}, we need to verify statment~(i) so that the inductive construction process can continue. 

\begin{lemma}[Periodic block deletion]\label{l: deletion of periodic orbit}
Let $X$ be a mixing subshift of finite type on the alphabet $A$. Assume that $h_X\=h_{\topo}(\sigma|_X)>0$ and $z\in X$ is a periodic point with least period $p\in\N$. There exist a constant $C(X,z)>0$, $\rho\in\Z/p\Z$, a residual class modulo $p$, and $m_0\in\N$
such that for every integer $m\geq m_0$ in that residue class, the word $w_m\=z_\rho z_{\rho+1}\ldots z_{\rho+m-1}$
has the following properties:
\begin{enumerate}
    \smallskip
    \item The subshift $Y_m\=\{x\in X:w_m\text{ does not occur in }x\}$ is a mixing subshift of finite type on the alphabet $A$ and $\fL(Y_m)=\{u\in\fL(X):w_m\text{ is not a factor of }u\}$.
    \smallskip
    \item $0\leq h_X-h_{\topo}(\sigma|_{Y_m})\leq C(X,z)\exp(-h_{X}m)$.
\end{enumerate}
\end{lemma}

\begin{proof}
Assume that $X$ is a $r$-step subshift of finite type for some $r\in\N$ and fix an integer $q\geq\max\{r,p\}$. Consider the $q$-block directed graph $G$ given by $V(G)\=\fL_q(X)$ and $E(G)=\fL_{q+1}(X)$. An edge $u_1u_2\ldots u_{q+1}\in E(G)$ has initial vertex $u_1 u_2\ldots u_{q}\in V(G)$ and
terminal vertex $u_2u_3\ldots u_{q+1}\in V(G)$. Since $q\geq r$ and $X$ is
mixing, $G$ is primitive. For each $i\in\mathbb Z/p\mathbb Z$, consider $c_i\=z_iz_{i+1}\ldots z_{i+q-1}\in V(G)$ and $e_i=z_iz_{i+1}\ldots z_{i+q}\in E(G)$. Since $q\geq p$ and $p$ is the least
period of $z$, the vertices $c_0,\,c_1,\,\ldots,\,c_{p-1}$ are distinct.  Thus, $\cC_z=(V(\cC_z),E(\cC_z))\=(\{c_i:i\in\Z/p\Z\},\{e_i:i\in\Z/p\Z\})$ is a simple directed $p$-cycle in $G$. 

In the following, we first find a path avoiding edges in $E(\cC_z)$ with endpoints in $V(\cC_z)$, which generates a path avoiding $w_m$ with the same endpoints as $w_m$ on the $m-1$-block representation
graph $H_m$ of $X$, showing that $Y_m$ is mixing. 

Since $h_X>0$, there exists an edge $e^*\in E(G)\smallsetminus E(\cC_z)$. By primitivity of $G$, we choose a path from $c_0$ to the initial vertex of
$e^*$ and a path from the terminal vertex of $e^*$ to $c_0$. Their concatenation with $e^*$ forms a closed cycle $W$ through $e^*$. Let $c_a$ be the last vertex of $V(\cC_z)$ visited by $W$ before this passing through $e^*$, and let $c_b$ be the first vertex of $V(\cC_z)$ visited by $W$ after this passing, where $a,\,b\in\Z/p\Z$. Denote the subpath of $W$ from $c_a$ to $c_b$ by $Q\=v_0v_1\ldots v_{l}$, where $l\in\N$, $v_i\in V(G)\smallsetminus V(\cC_z)$ for each integer $1\leq i\leq l-1$, $(v_i,v_{i+1})\in E(G)\smallsetminus E(\cC_z)$ for each integer $0\leq i\leq l-1$, $v_0=c_a$, and $v_l=c_b$. 

Let $u,\,v$ be the initial and terminal vertex of $e^*$, respectively. Since $G$ is
primitive, there exists $\kappa\in\N$ such that there exists a path $P^-$ (resp.~$P^+$) from $v$ to $u$ with length $\kappa-1$ (resp.~$\kappa$). Consequently, $C_\kappa\=e^*P^-$ and $C_{\kappa+1}\=e^*P^+$ are cycles in $G$, both containing the edge $e^*$, and have respective lengths $\kappa$ and $\kappa+1$.

Now set $\rho\= b-1\pmod p$ and consider the residue class
\begin{equation*}
 \mathcal R\=\{m\in\N :m\equiv q+2+a-b\pmod p\}.
\end{equation*}
Let $m\in\mathcal R$ with $m\geq q+2+\kappa$. Denote $S_m\=z_\rho z_{\rho+1}\ldots z_{\rho+m-2}\quad\text{ and }\quad 
 T_m\=z_{\rho+1}z_{\rho+2}\ldots z_{\rho+m-1}$. For a word $u$ of length at least $q$, write $\operatorname{ini}_q(u)$ and
$\operatorname{ter}_q(u)$ for its initial and terminal $q$-blocks, respectively. Then $\operatorname{ter}_q(S_m)
 =c_{\rho+m-1-q}=c_a$ and $\operatorname{ini}_q(T_m)=c_{\rho+1}=c_b$.

Let $H_m$ be the $(m-1)$-block directed graph given by $V(H_m)\=\fL_{m-1}(X)$ and $E(H_m)\=\fL_m(X)$. Note that the word $w_m=z_\rho z_{\rho+1}\ldots z_{\rho+m-1}$ is an edge on $H_m$ with initial vertex $S_m$ and terminal vertex $T_m$. Let $H_m-w_m$ denote the graph obtained by deleting the edge $w_m$. Note that $Y_m\=\{x\in X:w_m\text{ does not occur in }x\}$ is a $(m-1)$-step subshift of finite type and that $H_m-w_m$ is the $(m-1)$-block directed graph with $V(H_m-w_m)\=\fL_{m-1}(Y_m)$ and $E(H_m-w_m)\=\fL_m(Y_m)$.

To prove (i), we first construct a path in $H_m$ with the same endpoints as $w_m$ (initial vertex $S_m$ and terminal vertex $T_m$) while avoiding $w_m$.

For each integer $j\in[1,l]$, let $\xi_j\in A$ be the last symbol of $v_j$, and thus $e'_j=v_{j-1}\xi_j$. Consider $R_m\=z_{\rho+1+q}z_{\rho+2+q}\ldots z_{\rho+m-1}$ and $U_m\=S_m\xi_1\xi_2\ldots\xi_lR_m$. Since 
$\operatorname{ter}_q(S_m)
 =c_{\rho+m-1-q}=c_a$ and $\operatorname{ini}_q(T_m)=c_{\rho+1}=c_b$, reading $\xi_1\xi_2\ldots\xi_\ell$ after
$S_m$ moves the terminal $q$-block from $c_a$ to $c_b$. As $T_m=c_bR_m$, we see that $\operatorname{ini}_{m-1}(U_m)=S_m$ and $\operatorname{ter}_{m-1}(U_m)=T_m$. Indeed, $U_m$ labels a path in $G$, and thus $U_m\in\fL(X)$.

We show that the path in $H_m$ labelled by $U_m$ has the same endpoints as $w_m$ and avoids $w_m$. Set $N_m\=l+m-1-q$. Let $V_j$ be the length $(m-1)$ factor of $U_m$
starting at position $j$ for $0\leq j\leq N_m$, and let $F_i$ be the length $m$
factor starting at position $i$ for $0\leq i<N_m$. Now let $D_m\=V_0V_1\ldots V_{N_m}$ denote the path in $H_m$ labelled by $U_m$. To show that $D_m$ avoids $w_m$, we examine the path labelled by $U_m$ in $G$, which is the concatenation of
\begin{equation}\label{eq: U_m projected onto G}
 \underbrace{\text{$m-q-1$ edges of }\cC_z}_{\text{encoded by }S_m},
 \quad Q,\quad\text{ and }\quad
 \underbrace{\text{$m-q-1$ edges of }\cC_z}_{\text{appending }R_m}.
\end{equation}
Recall that every edge in $Q$ is not on the cycle $\cC_z$. Note that for each intger $j\in[0,N_m)$, $F_j$ labels a consecutive path of $m-q$ edges in $G$, which must contain an edge of $Q$ (cf.~(\ref{eq: U_m projected onto G})). As the path labelled by $w_m$ in $G$ consists entirely of edges of $\cC_z$, we see that for each integer $j\in[0, N_m)$, $F_j$ cannot equal $w_m$.

Since $X$ is mixing, $H_m$ is strongly connected. Given vertices $A,\,B\in V(H_m)$, choose a path from $A$ to $B$ in $H_m$ and replace
each occurrence of $w_m$ by $D_m$ so that the resulting path lies in
$H_m-w_m$. Thus, $H_m-w_m$ is strongly connected. To show that $Y_m$ is mixing, we further prove that $H_m-w_m$ is primitive. 

Denote the greatest common divisor of the
lengths of the cycles in a strongly connected (directed) graph $H$ by $\operatorname{per}(H)$. It suffices to show that $\operatorname{per}(H_m-w_m)=1$. Consider the two bi-infinite paths in $G$ obtained by periodically repeating
$C_\kappa$ and $C_{\kappa+1}$.  In either bi-infinite path, every block of
$\kappa+1$ consecutive edges contains an occurrence of $e^*$. Recall that $m-q>\kappa+1$. It follows that every block of $m-q$ consecutive edges in either bi-infinite path contains $e^*\in E(G)\smallsetminus E(\cC_z)$. Such a block cannot be the $G$-path of $w_m$, which consists of $m-q$ edges of $\cC_z$. Thus, the two bi-infinite paths give cycles in $H_m-w_m$ of lengths $\kappa$ and $\kappa+1$. It follows that 
\begin{equation*}
 \operatorname{per}(H_m-w_m)\mid\gcd(\kappa,\kappa+1)=1.
\end{equation*}
We conclude that
$H_m-w_m$ is primitive. As $Y_m$ is a $(m-1)$-step subshift of finite type and $H_m-w_m$ is the $(m-1)$-block directed graph induced by $Y_m$, we see that $Y_m$ is mixing and $\fL(Y_m)=\{u\in\fL(X):w_m\text{ is not a factor of }u\}$. Statement~(i) is now verified. 

Now we prove (ii). In the graph $G$, the word $w_m$ is an edge block $W_m$ of length $m-q$. Note that $Y_m$ is conjugate to the subshift obtained from the edge shift of $G$ by forbidding $W_m$. Applying \cite[Theorem~3]{Li89} to the primitive adjacency matrix of $G${\footnote{While Lind formulates his result for two-sided edge shifts and we work with one-sided shifts, both $X$ and $Y_m$ are mixing subshifts of finite type, and their natural extensions have the same finite languages and topological entropies so that (\ref{ineq: entropy loss in proof}) holds.}}, we see that there exist constants $d_G>0$ and $L_0\in\N$ such that if $m-q>L_0$, then
\begin{equation}\label{ineq: entropy loss in proof}
     0\leq h_X-h_{\topo}(\sigma|_{Y_m})\leq d_G\exp(-h_X(m-q)).
\end{equation}
We conclude that (i) \& (ii) holds with $C(X,z)\=d_Ge^{h_Xq}$ and $m_0\=q+2+\kappa+L_0$. The proof is now complete.
\end{proof}

\begin{lemma}[Word count from a deterministic labelled graph]
\label{l: finite-presentation}
Let $G$ be a strongly connected finite directed graph with $M$ vertices, whose edges are labelled by symbols from a finite alphabet $A$. Assume that the label sequences of infinite paths in $G$ form a subshift $Z$ of $A^{\N}$ and no two distinct edges leaving the same vertex have the same label. If $h_{\topo}(\sigma|_Z)>0$, then for all $n\in\N$,
\begin{equation}\label{ineq: finite-presentation-count}
    \card(\fL_n(Z)) \leq M^3e^{(n+2M)h_{\topo}(\sigma|_Z)}.
\end{equation}
\end{lemma}

\begin{proof}
    Fix $n\in\N$. Denote the number of length $n$ paths in $G$ by $N_G(n)$. Every word in $\fL_n(Z)$ labels at least one length $n$ path in $G$. Conversely, after the initial vertex is fixed, there exists at most one path carrying a given word in $\fL_n(Z)$ since no two distinct edges leaving the same vertex have the same label. Thus,
    \begin{equation}\label{ineq: in proof of count I}
        \card(\fL_n(Z))\leq N_G(n)\leq M\cdot\card(\fL_n(Z)).
    \end{equation}
    Let $B$ be the adjacency matrix of $G$. It follows from (\ref{ineq: in proof of count I}) that the maximal eigenvalue $\rho(B)$ of $B$ is equal to $e^{h_{\topo}(\sigma|_Z)}$.
    
    Fix a vertex $v$ of $G$. For each pair of vertices $(i,j)$, choose two paths from $v$ to $i$ and from $j$ to $v$ of lengths at most $M-1$, respectively. Concatenation of every length $n$ path from $i$ to $j$ with the two chosen paths yields a loop at $v$ of length $l\in\N$, which satisfies $l\leq n+2M-2$. It follows that
    \begin{equation*}
       \card\{\text{length }n\text{ paths from }i\text{ to }j\}\leq \bigl(B^l\bigr)_{v,v}\leq \operatorname{tr}\bigl(B^l\bigr)\leq M\rho(B)^l=Me^{lh_{\topo}(\sigma|_Z)}.
    \end{equation*}
    Summing over all pairs of vertices $(i,j)$, we see that $N_G(n)\leq M^3e^{(n+2M)h_{\topo}(\sigma|_Z)}$. This, together with (\ref{ineq: in proof of count I}), yields (\ref{ineq: finite-presentation-count}). 
\end{proof}

From now on, we fix a metric $d_*$ on $\cM\bigl(\{0,\,1,\,2\}^\N,\sigma\bigr)$ compatible with the weak$^*$ topology. The following lemma is a direct consequence of the upper semi-continuity of the entropy function (cf.~\cite[Theorem~3.5.6]{PU10}).

\begin{lemma}[Entropy isolation]\label{l: entropy-isolation}
Let $X\subseteq\{0,\,1,\,2\}^{\N}$ be a mixing subshift of finite type and $\mu_X\in\MME(X)$ be the unique measure of maximal entropy. Then for every $\epsilon>0$, there exists a constant $\eta(X,\epsilon)>0$ such that if some invariant Borel probability
$\nu$ supported on $X$ satisfies
$h_\nu(\sigma)>h_{\topo}(\sigma|_X)-\eta(X,\varepsilon)$, then $d_*(\nu,\mu_X)<\epsilon$.
\end{lemma}

Now we are ready to construct the freezing subshift on the alphabet $\{0,\,1,\,2\}$.

\begin{prop}[The freezing subshift]\label{p: the freezing subshift}
There exist sequences of integers $\{m_s\}_{s\in\N}$, words $\{f_s\}_{s\in\N}$,  and mixing subshifts of finite type $\{Y_s\}_{s\in\N_0}$ on $\{0,\,1,\,2\}$ satisfying the following properties:
\begin{enumerate}
    \smallskip
    \item For some $0<\theta<\alpha<\frac{\log 3+\log 2}{2}$, $2m_s\leq  e^{\theta m_s}\leq m_{s+1}\leq e^{\alpha m_s}$ for all $s\in\N$.

    \smallskip
    \item $Y_0\=\{0,1,2\}^{\N}$ and for each $s\in\N$, $f_s\in\{0,\,1,\,2\}^{m_s}$ and
    \begin{equation*}
  \fL(Y_s)=\biggl\{u\in\bigcup_{n=1}^{+\infty}\{0,1,2\}^{n}:f_1,\,f_2,\,\ldots,\,f_s
       \text{ are not factors of }u\biggr\}.
    \end{equation*}

    \smallskip
    \item Set $s_0\=0$ and $s_k\=\inf\{s\in\N: s>s_{k-1},\, Y_{s}\subsetneq Y_{s-1}\}$ for each $k\in\N$. For each $k\in\N_0$, let $X_k\=Y_{s_k}$, $\mu_k$ be the unique measure of maximal entropy of $X_k$, $h_k\=h_{\topo}(\sigma|_{X_k})$, and $\eta_k\=\eta\bigl(X_k,2^{-k}\bigr)$ be the constant from Lemma~\ref{l: entropy-isolation}. Consider for each $i\in\N$, $\Delta_i=h_{i-1}-h_i$. Then $\inf_{k\in\N_0}h_k>\log 2$, there exists $a\in \bigl(\alpha,\frac{\log 3+\log 2}{2}\bigr)$ such that for each $i\in\N$, $\Delta_i\leq e^{-am_{s_i}}$, and for all $j,\,k\in\N$ with $j>k$, $\Delta_j\leq2^{-j}\eta_k$.
\end{enumerate}
Consequently, $K\=\bigcap_{s=0}^{+\infty}Y_s=\bigcap_{k=0}^{+\infty}X_k$ contains no periodic orbits, is of positive entropy, and satisfies (\ref{ineq: complexity condition}). Moreover, $\MME(K)$ consists of a unique element $\mu_K$.
\end{prop}

\begin{proof}
Denote $h_*\=\frac{\log3+\log2}{2}$ and choose constants $0<\theta<\alpha<a<h_*$. Enumerate the periodic orbits of $\{0,\,1,\,2\}^{\N}$ as $\{\cO_l\}_{l\in\N}$. Beginning with $Y_0\=\{0,\,1,\,2\}^{\N}$ and $h_0\=h_{\topo}(\sigma|_{Y_0})$, we first describe the inductive construction process producing $\{Y_s\}_{s\in\N}$. For each $s\in\N$, we try to delete some periodic orbit from $Y_{s-1}$ by applying Lemma~\ref{l: deletion of periodic orbit}. However, to ensure properties (i)--(iii), and (i) in particular, we need to set $Y_s\=Y_{s-1}$ for some $s\in\N$, which are called stationary stages. We show that after finitely many stationary stages, a genuine stage where $Y_{s}\subsetneq Y_{s-1}$ will take place. The strictly decreasing subsequence of $\{Y_s\}_{s\in\N_0}$ is recorded as $\{X_k\}_{k\in\N_0}$ and $K$ is obtained as $\bigcap_{s\in\N_0}Y_s=\bigcap_{k\in\N_0}X_k$. We then estimate further properties of $K$ with (i)--(iii).

Now we start the inductive construction. Fix a periodic point $z_1\in \cO_1$ and apply
Lemma~\ref{l: deletion of periodic orbit} to $(Y_0,z_1)$. Let $C^{\mathrm{del}}_0\=C(Y_0,z_1)$ be the constant from Lemma~\ref{l: deletion of periodic orbit}. Choose a large enough admissible $m_1${\footnote{As the admissible arithmetic progression is unbounded, one admissible $m_1$ can be chosen to satisfy the following requirements simultaneously.}} so that $C^{\mathrm{del}}_0e^{-h_0m_1}\leq e^{-am_1}$, $2m\leq e^{\theta m}\leq\bigl\lceil e^{\theta m}\bigr\rceil\leq e^{\alpha m}$ for all integers $m\geq m_1$, and
    \begin{equation}\label{ineq: entropy-budget}
    \sum_{s=1}^{+\infty}e^{-a2^{s-1}m_1}<\log3-h_*.
    \end{equation}
Then by Lemma~\ref{l: deletion of periodic orbit}, for the corresponding blocker word $f_1\=w_{m_1}(\cO_1)$ and the resulting mixing subshift of finite type $Y_1\=\{x\in Y_0: f_1\text{ does not occur in }x\}$, we have 
    \begin{equation*}
    h_{\topo}(\sigma|_{Y_0})-h_{\topo}(\sigma|_{Y_1})\leq C^{\mathrm{del}}_0e^{-h_0m_1}\leq e^{-am_1},
    \end{equation*}

The orbit $\cO_1$ is not in $Y_1$, so $Y_1\subsetneq Y_0$.  We then record $s_1=1$, $X_0=Y_0$, $X_1=Y_1$, $h_i\=h_{\topo}(\sigma|_{X_i})$ for $i=0,\,1$, and $\Delta_1=h_0-h_1$. In the following, whenever a new subshift $X_k$ is recorded, let $\mu_k$ be the unique measure of maximal entropy of $X_k$, write $\Delta_k\=h_{\topo}(\sigma|_{X_{k-1}})-h_{\topo}(\sigma|_{X_{k}})$, and take $\eta_k\=\eta\bigl(X_k,2^{-k}\bigr)$ from Lemma~\ref{l: entropy-isolation}. In particular, we fix $\eta_1$ at this point.

Now suppose that for some $s'\in\N$, words $f_1,\,f_2,\,\ldots,\,f_{s'}$ with lengths $m_1<m_2<\cdots<m_{s'}$ and the mixing subshifts of finite type $Y_0,\,Y_1,\,\ldots,\,Y_{s'}$ have been constructed so that for $1\leq s<s'$, $2m_s\leq e^{\theta m_s}\leq m_{s+1}\leq e^{\alpha m_s}$; for $1\leq s\leq s'$, $f_s\in\{0,\,1,\,2\}^{m_s}$ and
\begin{equation*}
     \fL(Y_s)=\biggl\{u\in\bigcup_{n=1}^{+\infty}\{0,\,1,\,2\}^{n}:f_1,\,f_2,\,\ldots,\,f_s\text{ are not factors of }u\biggr\}.
\end{equation*}
Assume that we have recorded $s_1<s_2<\cdots<s_k\leq s'$, where $k\in\N$, and for $1\leq j\leq k$, $\Delta_j\leq e^{-am_{s_j}}$; for $2\leq j\leq k$,
\begin{equation}\label{ineq: entropy gap for j from 2 to k}
 \Delta_j\leq2^{-j}\min_{1\leq i<j}\eta_i.
\end{equation}
By (\ref{ineq: entropy-budget}), we see that 
\begin{equation}\label{ineq: entropy estimate for Xk}
 h_k=h_0-\sum_{j=1}^k\Delta_j>\log3-\sum_{t=1}^{+\infty}e^{-am_t}>\log 3-\sum_{t=1}^{+\infty}e^{-a2^{t-1}m_1}>h_*>a.
\end{equation}

It follows that the mixing subshift of finite type $Y_{s'}=X_k$ is of positive entropy. Let $\cO$ be the first orbit in the fixed enumeration $\{\cO_l\}_{l\in\N}$ which is contained in
$Y_{s'}$.  Fix some periodic point $z_\cO\in \cO$. Applying Lemma~\ref{l: deletion of periodic orbit} to $(Y_{s'},z_\cO)$, we get a constant $C_k^{\mathrm{del}}\=C(Y_{s'},z_\cO)>0$, an arithmetic progression of admissible lengths, and a threshold $M_k$ such
that, for every admissible integer $m\geq M_k$, the entropy loss
\begin{equation*}
 \Delta_{k+1}(m)\= h_k-h_{\topo}\bigl(\sigma|_{Y_{s'}\cap\{w_m(\cO)\text{ does not occur}\}}\bigr)
\end{equation*}
satisfies
\begin{equation}\label{ineq: stage-deletion-cost}
 0\leq\Delta_{k+1}(m)
 \leq C_k^{\mathrm{del}}e^{-h_km}.
\end{equation}
Since $h_k-a>0$, we increase $M_k$ if necessary, so that for every admissible integer $m\geq M_k$,
\begin{equation*}
 C_k^{\mathrm{del}}e^{-(h_k-a)m}\leq1\quad\text{ and }\quad
 C_k^{\mathrm{del}}e^{-h_km}\leq2^{-(k+1)}\min_{1\leq i\leq k}\eta_i.
\end{equation*}
Together with~\eqref{ineq: stage-deletion-cost}, we see that for every admissible integer $m\geq M_k$,
\begin{equation}\label{ineq: stage-deletion-cost-for induction}
 \Delta_{k+1}(m)\leq e^{-am}\quad\text{ and }\quad\Delta_{k+1}(m)\leq2^{-(k+1)}\min_{1\leq i\leq k}\eta_i.
\end{equation}

The remaining induction hypothesis to be satisfied is that $e^{\theta m_{s'}}\leq m\leq e^{\alpha m_{s'}}$. To this end, let $I_{s'}\=[e^{\theta m_{s'}},e^{\alpha m_{s'}}]$. If $I_{s'}$ contains an admissible integer $m\geq M_k$, then we call $s'+1$ a genuine stage and set 
\begin{equation*}
 m_{s'+1}\=m,\quad
 f_{s'+1}\=w_m(\cO),\quad\text{ and }\quad
 Y_{s'+1}\=Y_{s'}\cap\{w_m(\cO)\text{ does not occur}\}.
\end{equation*}
The orbit $\cO$ is not in $Y_{s'+1}$, so
$Y_{s'+1}\subsetneq Y_{s'}$. We then record
\begin{equation*}
 s_{k+1}\=s'+1,\quad X_{k+1}\=Y_{s'+1},\quad
 h_{k+1}\=h_{\topo}(\sigma|_{X_{k+1}}),\quad\text{ and }\quad\Delta_{k+1}\=\Delta_{k+1}(m).
\end{equation*}
Otherwise, we call $s'+1$ a stationary stage and set
\begin{equation*}
    m_{s'+1}\=\bigl\lceil e^{\theta m_{s'}}\bigr\rceil,\quad f_{s'+1}\=f_1 0^{m_{s'+1}-m_1},\quad\text{ and }\quad Y_{s'+1}\=Y_{s'}.
\end{equation*}
Note that in both cases, the requirements (i)--(iii) are satisfied for $1\leq s\leq s'+1$ (cf.~(\ref{ineq: stage-deletion-cost-for induction})). 

Now we show that genuine stages will always take place after finite steps of stationary stages. Note that during the stationary stages, the subshift $Y_{s'}=X_k$, the periodic orbit $\cO$, the constant $C_k^{\mathrm{del}}$, the admissible arithmetic progression of integers $m$, the threshold $M_k$, and $\eta_1,\,\eta_2\,\ldots,\,\eta_k$ remain fixed,
whereas both the left endpoint of the interval $I_{s'}$ and the length of $I_{s'}$ are tending to infinity. Thus, $I_{s'}$ will contains an admissible integer beyond the threshold $M_k$ when $s'$ is large enough, leading to a genuine stage. 

We conclude that the induction process produces infinitely many genuine stages with requirements (i)--(iii) satisfied, and at the $k$-th genuine stage the first periodic orbit in $\{\cO_l\}_{l\in\N}$ which survives in $X_{k-1}$ is deleted so that $\cO_i$ is not in $X_k$ for $1\leq i\leq k$. This proves that $K=\bigcap_{s=0}^{+\infty}Y_s=\bigcap_{k=0}^{+\infty}X_k$ contains no periodic orbits. It remains to establish further properties of $K$.

Recall from (\ref{ineq: entropy estimate for Xk}) that $h_k>h_*>\log 2$ for each $k\in\N$. Let $\mu$ be a weak$^*$ accumulation point of $\mu_k\in\MME(X_k)$ as $k\to+\infty$. Since $K=\bigcap_{k=0}^{+\infty}X_k$ and the entropy is upper semicontinuous (cf.~[Theorem~3.5.6]\cite{PU10}), we see that $\mu$ is supported on $K$ and 
\begin{equation*}
 \limsup_{k\to +\infty}h_k=\limsup_{k\to+\infty}h_{\mu_k}(\sigma)
 \leq h_\mu(\sigma)\leq h_{\topo}(\sigma|_K)\leq\inf_{k\in\N_0}h_k.
\end{equation*}
Thus, $K$ is of positive entropy with $h_{\topo}(\sigma|_K)\ge h_*>\log 2$.

Recall from (iii) (cf.~(\ref{ineq: entropy gap for j from 2 to k}) and~(\ref{ineq: stage-deletion-cost-for induction})) that for all $j,\,k\in\N$ with $j>k$, $\Delta_j\leq2^{-j}\eta_k$. Consequently, for each $k\in\N$,
\begin{equation*}
 h_k-h_{\topo}(\sigma|_K)=\sum_{j=k+1}^{+\infty}\Delta_j
 \leq\eta_k\sum_{j=k+1}^{+\infty}2^{-j}
 =2^{-k}\eta_k<\eta_k.
\end{equation*}
Thus for each $\nu\in\MME(K)$, we have $\nu\in\cM(X_k,\sigma)$ and
$h_\nu(\sigma)=h_{\topo}(\sigma|_K)>h_k-\eta_k$ for all $k\in\N$. Since $\eta_k=\eta\bigl(X_k,2^{-k}\bigr)$, we see from Lemma~\ref{l: entropy-isolation} that $d_*(\nu,\mu_k)<2^{-k}$. It follows that for $\nu,\,\nu'\in\MME(K)$, $d_*(\nu,\nu')<2^{1-k}$ for all $k\in\N$, and thus $\nu=\nu'$. We conclude that $K$ has a unique measure of maximal entropy, denoted by $\mu_K$. 

Now it suffices to verfiy (\ref{ineq: complexity condition}) for $K$.

Recall that $\fL(Y_s)=\bigl\{u\in\bigcup_{n=1}^{+\infty}\{0,\,1,\,2\}^n:f_1,\,f_2,\,\ldots,\,f_s\text{ are not factors of }u\bigr\}$. Let $u\in\bigcup_{n=1}^{+\infty}\{0,\,1,\,2\}^{n}$. If $f_s$ is not a factor of $u$ for every $s\in\N$, then $[u]\cap Y_s$ is nonempty for every $s\in\N$. Since $[u]\cap Y_s$ is compact and nonincreasing with respect to $s$, we see that the limit $[u]\cap K$ is nonempty. It follows that 
\begin{equation}\label{eq: language of K}
 \fL(K)=\biggl\{u\in\bigcup_{n=1}^{+\infty}\{0,\,1,\,2\}^n:f_s\text{ is not a factor of } u\text{ for every }s\in\N\biggr\}.
\end{equation}

Let $k\in\N$. Note that the word $f_{s_k}\in\fL(Y_{s_k-1})$ and $\abs{f_t}>\abs{f_{s_k}}$ for all $t>s_k$. Thus, every proper factor of $f_{s_k}$ avoids $f_t$ for every $t\in\N$, and $f_{s_k}$ is minimal forbidden for $K$.  Conversely, by \eqref{eq: language of K}, every forbidden word contains $f_{s}$ for some $s\in\N$, and thus contains $f_{s_k}$ for some $k\in\N$. We conclude that $\{f_{s_k}\}_{k\in\N}$ are the minimal forbidden words for $K$. Since the length of these minimal forbidden words are different, we see that for all $m\in\N$,
\begin{equation}\label{ineq: count of minimal forbidden}
b_K(m)\leq 1. 
\end{equation}

To analyze the language complexity of $K$, we construct a deterministic labelled graph presenting $Y_{s-1}$ for each integer $s\geq 2$ to apply Lemma~\ref{l: finite-presentation}. Let $s\geq2$ be an integer. Consider $\cD_{s-1}\=
 \{f_{s_j}:j\in\N,\,s_j\leq s-1\}$ and
\begin{equation*}
 \cP_{s-1}\=\{\epsilon\}\cup\{v:v\text{ is a proper prefix of some } d\in\cD_{s-1}\},
\end{equation*}
where $\epsilon$ denotes the empty word. For $v\in\cP_{s-1}$ and $c\in\{0,\,1,\,2\}$, draw an edge labelled $c$ from $v$ unless $vc$ has a suffix belonging to
$\cD_{s-1}$.  When this edge is present, its terminal vertex is
defined to be the longest suffix of $vc$ which belongs to $\cP_{s-1}$. {\footnote{Such a suffix always exists because
$\epsilon\in\mathcal P_{s-1}$.}} In this way, we construct a directed graph $G_{s-1}$ with edges labelled by $\{0,\,1,\,2\}$. Walking along paths of $G_{s-1}$ corresponds to reading a word. More precisely, we say that a finite word $u=u_1u_2\ldots u_n$ is \emph{readable from a vertex $v$} if there exists a path starting at $v$ whose successive
edge labels are $u_1,\,u_2,\,\ldots,\,u_n$. 

\smallskip
\emph{Claim.} A finite word $u$ on $\{0,\,1,\,2\}$ is readable from $\epsilon$ if and only if $u\in\fL(Y_{s-1})$. Whenever the word $u$ is readable from $\epsilon$, the path representing $u$ ends at the longest suffix of $u$ belonging to $\cP_{s-1}$.\footnote{The claim follows from an application of \cite[Lemma~2 and Theorem~3]{CMR98}. Here we keep the proof of the claim for readers' convenience.}
\smallskip

For a finite word $u$ on $\{0,\,1,\,2\}$, let $q(u)$ be the longest suffix of $u$ belonging to $\cP_{s-1}$. We prove the claim by induction on $\lvert u\rvert$. Both assertions of the claim are immediate for $u=\epsilon$.  Suppose that they hold for a word $u$, and consider $uc$, where $c\in\{0,\,1,\,2\}$.

If $u$ contains a word from $\mathcal D_{s-1}$, then $u$ is not
readable by the induction hypothesis.  Hence $uc$ is not readable, as any path representing $uc$ would have an initial subpath representing $u$.  Meanwhile, $uc$ still contains the same forbidden word. If $u$ avoids $\mathcal D_{s-1}$, by the induction hypothesis, $u$ is readable from $\epsilon$ and
the corresponding path ends at $q(u)$.
It follows that the word $uc$ fails to avoid $\mathcal D_{s-1}$ precisely when some $d\in\mathcal D_{s-1}$ is a suffix of $uc$.  Note that for $d\in\cD_{s-1}$,
\begin{equation}\label{eq:forbidden-suffix-equivalence}
 d\text{ is a suffix of }uc
 \quad\Longleftrightarrow\quad
 d\text{ is a suffix of }q(u)c.
\end{equation}
The reverse implication follows immediately from the fact that $q(u)$ is
a suffix of $u$. For the forward implication, suppose that $d$ is a
suffix of $uc$.  Write $d\=d^{-}c$.  Then $d^{-}\in\mathcal P_{s-1}$ and $d^{-}$ is a suffix of $u$. As $q(u)$ is the longest suffix of $u$ belonging to
$\mathcal P_{s-1}$, we see that $d^{-}$ is a suffix of $q(u)$.  Thus, $d=d^{-}c$ is a suffix of $q(u)c$, proving \eqref{eq:forbidden-suffix-equivalence}. By the definition of the graph $G_{s-1}$, an edge labelled $c$ leaves $q(u)$
precisely when no member of $\mathcal D_{s-1}$ is a suffix of $q(u)c$, which is equivalent to
$uc$ avoiding $\mathcal D_{s-1}$ (cf.~(\ref{eq:forbidden-suffix-equivalence})).  Hence,
\begin{equation*}
 uc\text{ is readable from }\epsilon
 \quad\Longleftrightarrow\quad
 uc\text{ avoids }\mathcal D_{s-1}.
\end{equation*}
Suppose now that $uc$ is readable from $\epsilon$. It is straightforward to check that $q(uc)=q(q(u)c)$. By induction hypothesis, the corresponding path will arrive at $q(u)$ when finishing $u$, and the last edge is from $q(u)$ to $q(q(u)c)=q(uc)$, which completes the induction. The claim is now established.

\smallskip
Since $Y_{s-1}$ is a
mixing subshift of finite type, we can  choose $x\in Y_{s-1}$ in which every word of $\fL(Y_{s-1})$ occurs infinitely often. The path obtained by
reading $x$ eventually remains in a strongly connected component of
$G_{s-1}$, denoted by $G_{s-1}'$. It is straightforward to see that the label sequences of infinite paths in $G'_{s-1}$ forms the subshift $Y_{s-1}$. Let $n_{s-1}$ denote the number of vertices of $G_{s-1}'$. Then
\begin{equation}\label{ineq:automaton-state-count}
 n_{s-1}
 \leq \lvert\mathcal P_{s-1}\rvert
 \leq 1+\sum_{d\in\mathcal D_{s-1}}(\lvert d\rvert-1)
 \leq 1+\sum_{i=1}^{s-1}m_i
 \leq 3m_{s-1},
\end{equation}
where the third inequality follows from the fact that $\abs{f_i}=m_i$ for $i\in\N$ and the last inequality follows from the fact that $2m_i\leq m_{i+1}$ for $i\in\N$.

Now let $n\geq m_2$ and consider the unique integer $s\geq2$ such that $m_s\leq n<m_{s+1}$. Since $K\subseteq Y_{s-1}$, we apply Lemma~\ref{l: finite-presentation} to $Y_{s-1}$ and $G_{s-1}'$ and  obtain
\begin{equation}\label{ineq:normalized-finite-stage-count}
 \log (\card(\fL_n(K)))-nh_{\topo}(\sigma|_K)
 \leq
 3\log n_{s-1}+2n_{s-1}h_{\topo}(\sigma|_{Y_{s-1}})
 +n(h_{\topo}(\sigma|_{Y_{s-1}})-h_{\topo}(\sigma|_K)).
\end{equation}
Recall that $e^{\theta m_{s-1}}\leq m_s$. It follows from \eqref{ineq:automaton-state-count} and our choice of $m_s$ (cf.~(i)) that
\begin{equation}\label{ineq: count of n_s-1}
 n_{s-1}\leq3m_{s-1}\leq3/\theta\log m_s\leq3/\theta\log n.
\end{equation}
Since $h_{\topo}(\sigma|_{Y_{s-1}})\leq\log3$, the first two terms on the right-hand side of \eqref{ineq:normalized-finite-stage-count} are bounded by
$A+B\log(n+1)$ for some constants $A,\,B>0$ (cf.~(\ref{ineq: count of n_s-1})). For the last term, we see from $\Delta_i\leq e^{-am_{s_i}}$ (cf.~(iii)) and $2m_i\leq m_{i+1}$ (cf.~(i)) for all $i\in\N$ that
\begin{align}
 h_{\topo}(\sigma|_{Y_{s-1}})-h_{\topo}(\sigma|_K)
 &=\sum_{t=s}^{+\infty}(h_{\topo}(\sigma|_{Y_{t-1}})-h_{\topo}(\sigma|_{Y_t}))\\
 &\leq\sum_{t=s}^{+\infty}e^{-am_t}
 \leq\sum_{r=0}^{+\infty}e^{-a2^rm_s}\leq\sum_{r=0}^{+\infty}e^{-a(r+1)m_s}
 \leq C_{\mathrm{tail}}e^{-am_s},
 \label{ineq:finite-stage-entropy-tail}
\end{align}
where $C_{\mathrm{tail}}\=\frac{1}{1-e^{-am_1}}\in(0,+\infty)$. Using $n<m_{s+1}\leq e^{\alpha m_s}$ and $\alpha<a$, we see that $n(h_{\topo}(\sigma|_{Y_{s-1}})-h_{\topo}(\sigma|_K))
 \leq C_{\mathrm{tail}}e^{-(a-\alpha)m_s}
 \leq C_{\mathrm{tail}}$.

By (\ref{ineq: count of n_s-1}) and~(\ref{ineq:finite-stage-entropy-tail}), we conclude that there exists constants
$C,\,D>0$ such that for all $n\geq m_2$,
\begin{equation*}
    \log (\card(\fL_n(K)))-nh\leq C+D\log(n+1).
\end{equation*}
After increasing $C$ to include the finitely many scales $n<m_2$ and setting $C_0\=e^C\geq 1$, we conclude that for all $n\in\N_0$, $\card(\fL_n(K))\leq C_0(n+1)^D e^{hn}$. This, together with (\ref{ineq: count of minimal forbidden}), yields (\ref{ineq: complexity condition}). The proof is now complete.
\end{proof}

\begin{proof}[\bf Proof of Theorem~\ref{t: freezing for not log Dini}]
    Let $K\subseteq\{0,\,1,\,2\}^{\N}$ be the subshift from Proposition~\ref{p: the freezing subshift} and $\mu_K$ be the unique element of $\MME(K)$. Let $\omega$ be a modulus of continuity not satisfying the weak log-log Dini condition. Note that $-\omega(d_2(\,\cdot\,,K))$ is $\omega$-continuous as $d_2$ is an ultrametric. Applying Lemma~\ref{l: freezing criterion for positive entropy subshift} to $K$ and $\omega$, we see that  $-\omega(d_2(\,\cdot\,,K))$ admits a freezing phase transition and $\mu_K$ is the unique frozen state for $-\omega(d_2(\,\cdot\,,K))$.
\end{proof}

\section{Super-exponential suppression beyond Dini regularity}\label{s: super-exponential suppression beyond Dini}

Here we establish Theorem~\ref{t: super-exponential suppression for non-Dini} by inducing on an open set $U$. With Kac's formula (see e.g.\ \cite[Lemma~4.1]{BQ16}) and Abramov's formula (see e.g.\ \cite[Theorem~4.3.3]{Do11}), we estimate both $\int\!\omega(d_2(\,\cdot\,,\cO))\,\mathrm{d}\mu$ and $h_\mu$ in terms of the probabilities $p_n$ that the first return to $U$ occurs at time $n$ (cf.~(\ref{ineq: energy estimate from first return}) and (\ref{ineq: hmu in terms of pn})), where $\mu$ is an equilibrium state for $-\beta\omega(d_2(\,\cdot\,,\cO))$. Finally, we use a partition function technique (cf.~(\ref{eq: defi of the partition trick})) to show the super-exponential suppression.  

\begin{proof}[\bf Proof of Theorem~\ref{t: super-exponential suppression for non-Dini}]
Choose $r\in\N$ such that $r\geq 2$ and the prefixes of length $r-1$ distinguish
the points of $\cO$, i.e., the map $z\mapsto z_{1}z_2\ldots z_{r-1}$ from $\cO$ to $A^{r-1}$ is injective. Consider the open subset $U\=X\smallsetminus\bigcup_{w\in\fL_r(\cO)}[w]$ of $X$. Since the prefixes of length $r-1$ already distinguish the points of $\cO$, we see that $\bigcap_{j=0}^{+\infty}\sigma^{-j}(X\smallsetminus U)=\cO$. It then follows from $h_{\topo}(\sigma)>0$ that $U$ is nonempty. We claim that 
\begin{equation}\label{eq: super-exponential suppression at U}
    \lim_{\beta\to+\infty}\beta^{-1}\log\sup_{\mu\in\fE(\sigma,-\beta\omega(d_2(\,\cdot\,,\cO)))}\mu(U)=-\infty.
\end{equation}

Now we induce on $U$. Let $\tau(x)$ denote the first return time of $x\in U$ to $U$. For each $n\in\N$, consider the set $B_n\=\bigl\{w\in A^{n+r}: \text{there exists }x\in U\cap\{y\in U:\tau(y)=n\}\text{ such that }x|_{n+r}=w\bigr\}$ of return branches of length $n$. We next show that the sequence $\{\card B_n\}_{n\in\N}$ has an upper bound. 

Suppose $x\in U$ satisfies $\tau(x)=n$ for some intger $n\geq 2$. Then $\sigma^j(x)\notin U$ for $1\leq j<n$. Since prefixes of length $r-1$ distinguish points of $\cO$, we see that the length $r$ block beginning at $x_2$ determines a unique point $z$ in $\cO$, and $x_{k+1}\ldots x_{k+r-1}=\sigma^{k-1}(z)|_{r-1}$ for $1\leq k\leq n$. Thus, the length $r$ block beginning at $x_2$, together with $x_1$ and $x_{n+r}$, uniquely determines the return branch $x_1x_2\ldots x_{n+r}$. Since the length $r$ block beginning at $x_2$ is in $\fL_r(\cO)$, we see that $\card B_n\leq\card(\fL_r(\cO))\cdot(\card A)^2\leq \card\cO\cdot(\card A)^2$. We conclude that the sequence $\{\card B_n\}_{n\in\N}$ has an upper bound $C_B\=\max\bigl\{\card\cO\cdot(\card A)^2,\,(\card A)^{r+1}\bigr\}$.

After the above preparations, we further investigate equilibrium states $\mu\in\fE(\sigma,-\beta\omega(d_2(\,\cdot\,,\cO)))$ for $\beta\in (0,+\infty)$. If $\mu(U)=0$ for each $\beta\in (0,+\infty)$ and each equilibrium state $\mu$ for $-\beta\omega(d_2(\,\cdot\,,\cO))$, then (\ref{eq: super-exponential suppression at U}) immediately holds. Now suppose $\mu(U)>0$, where $\beta>0$ and $\mu\in\fE(\sigma,-\beta\omega(d_2(\,\cdot\,,\cO)))$. Denote a periodic measure supported on $\cO$ by $\delta_{\cO}$. It is straightforward to see that the measure $\mu$ satisfies 
\begin{equation}\label{ineq: pressure of mu is nonnegative}
h_{\mu}(\sigma)-\beta\int\!\omega(d_2(\,\cdot\,,\cO))\,\mathrm{d}\mu\geq h_{\delta_\cO}(\sigma)-\beta\int\!\omega(d_2(\,\cdot\,,\cO))\,\mathrm{d}\delta_{\cO}=0.
\end{equation}
Denote $W\=\bigcup_{j=0}^{+\infty}\sigma^{-j}(U)$. Note that $X\smallsetminus W=\bigcap_{j=0}^{+\infty}\sigma^{-j}(X\smallsetminus U)=\cO$. Since $\mu(U)>0$, by the Poincar\'e
recurrence theorem, we see that $\mu\bigl(\sigma^{-1}(W)\bigr)=\mu(W)$. Consider $\nu\=\frac{1}{\mu(W)}\mu|_W$, $\mu_U\=\frac{1}{\mu(U)}\mu|_U$, and the first return map $F\=\sigma^\tau\: U\to U$. It is straightforward to see that $\nu$ is $\sigma$-invariant and $\mu_U$ is $F$-invariant. Denote $p_n\=\mu_U(\{x\in U:\tau(x)=n\})$ for each $n\in\N$. Kac's formula \cite[Lemma~4.1]{BQ16} implies
\begin{equation}\label{eq: kac formulas}
 \mu(U)\sum_{n=1}^{+\infty}np_n=\mu(W)
 \quad\text{ and }\quad
 \int_U\!\sum_{k=0}^{\tau(x)-1}\omega\bigl(d_2\bigl(\sigma^k(x),\cO\bigr)\bigr)\,\mathrm{d}\nu(x)=\int_W\!\omega(d_2(\,\cdot\,,\cO))\,\mathrm{d}\nu.
\end{equation}
 Set $\delta\=d_2(U,O)>0$. Suppose $x\in U$ satisfies $\tau(x)=n$ for some $n\in\N$. Then for $0\leq k<n$, $d_2\bigl(\sigma^k(x),\cO\bigr)\geq 2^{k-n}d_2(\sigma^n(x),\cO)\geq 2^{k-n}\delta$. It follows that 
\begin{equation*}\label{ineq: energy cost at tau=n}
\sum_{k=0}^{n-1}\omega\bigl(d_2\bigl(\sigma^k(x),\cO\bigr)\bigr)\geq \assocsum{2}{n}(\delta).
\end{equation*}
This, together with the second equality in (\ref{eq: kac formulas}), yields the following energy estimate
\begin{equation}\label{ineq: energy estimate from first return}
 \int\!\omega(d_2(\,\cdot\,,\cO))\,\mathrm{d}\mu\geq \int_W\!\omega(d_2(\,\cdot\,,\cO))\,\mathrm{d}\mu
 \geq \mu(U)\sum_{n=1}^{+\infty}p_n\assocsum{2}{n}(\delta).
\end{equation}

We then estimate $h_{\mu}(\sigma)$. Let $\sR\=\{[w]:w\in B_n\text{ for some }n\in\N\}$ be the partition of $U$ into return branches. It is straightforward to see that the partition $\sR$ is generating for the induced system $(U,F,\mu_U)$. As the sequence $\{\card B_n\}_{n\in\N}$ has an upper bound $C_B\in (0,+\infty)$, we see that the partition $\sR$ has at most $C_B$ atoms conditional on the return time $\tau=n$, and thus,
\begin{equation*}
 h_{\mu_U}(F)
 \leq -\sum_{[w]\in\sR}\mu_U([w])\log(\mu_U([w]))
 \leq\log C_B-\sum_{n=1}^{+\infty}p_n\log p_n.
\end{equation*}
Note that the sum $\sum_{n=1}^{+\infty}np_n<+\infty$ by (\ref{eq: kac formulas}), and
therefore $-\sum_{n=1}^{+\infty}p_n\log p_n<\infty$. It then follows from the affinity of the entropy and Abramov's formula (cf.~\cite[Theorem~4.3.3]{Do11}) that
\begin{equation}\label{ineq: hmu in terms of pn}
 h_\mu(\sigma)
 =\mu(W) h_\nu(\sigma)+(1-\mu(W))h_{\delta_O}(\sigma)
 =\mu(U)h_{\mu_U}(F)+0
 \leq \mu(U)\biggl(\log C_B-\sum_{n=1}^{+\infty}p_n\log p_n\biggr).
\end{equation}
This, together with (\ref{ineq: energy estimate from first return}) and~(\ref{ineq: pressure of mu is nonnegative}), yields
\begin{equation}\label{ineq: entropy-energy written out}
 \beta\sum_{n=1}^{+\infty}p_n\assocsum{2}{n}(\delta)\leq \log C_B-\sum_{n=1}^{+\infty}p_n\log p_n.
\end{equation}

We then estimate $\mu(U)$ from this inequality. Fix $R>0$ and define for each $\beta'\in (0,+\infty)$, 
\begin{equation}\label{eq: defi of the partition trick}
Z_{\beta'}\=\sum_{n=1}^{+\infty}\exp\bigl(-\beta'\assocsum{2}{n}(\delta)-ne^{-2R\beta'}\bigr).    \end{equation}
By Jensen's inequality, we see that for each $\beta'\in (0,+\infty)$,
\begin{equation}\label{ineq: use of Zbeta}
 -\sum_{n=1}^{+\infty}p_n\log p_n-\beta'\sum_{n=1}^{+\infty}p_n\assocsum{2}{n}(\delta)-e^{-2R\beta'}\sum_{n=1}^{+\infty}np_n
 \leq\log Z_{\beta'}.
\end{equation}
Since the modulus of continuity $\omega$ does not satisfy the Dini condition, we see that the sequnce $\{\assocsum{2}{n}(\delta)\}_{n\in\N}$ is strictly increasing and $\lim_{n\to+\infty}\assocsum{2}{n}(\delta)=+\infty$. Thus, there exists $N\in\N$ such that $\assocsum{2}{n}(\delta)\geq 3R$ for all integers $n\geq N$. It follows that 
\begin{equation*}
 \limsup_{\beta'\to+\infty}Z_{\beta'}
 \leq \lim_{\beta'\to+\infty}Ne^{-\beta'\assocsum{2}{1}(\delta)}
   +\limsup_{\beta'\to+\infty}e^{-3R\beta'}\sum_{n=N}^{+\infty}\exp\bigl(-ne^{-2R\beta'}\bigr)
   =0.
\end{equation*}
Then there exists $\beta_R>0$ such that $\log Z_{\beta'}+e^{-R\beta'}<-\log C_B$ for all $\beta'\in (\beta_R,+\infty)$. If $\beta>\beta_R$, then it follows from (\ref{ineq: entropy-energy written out}) and~(\ref{ineq: use of Zbeta}) that 
\begin{equation*}
    -\log C_B
    \leq \log Z_{\beta}+e^{-2R\beta}\sum_{n=1}^{+\infty}np_n
    <-\log C_B-e^{-R\beta}+e^{-2R\beta}\sum_{n=1}^{+\infty}np_n,
\end{equation*}
i.e., $\sum_{n=1}^{+\infty}np_n>e^{R\beta}$, and by (\ref{eq: kac formulas}), $\mu(U)\leq e^{-R\beta}\mu(W)\leq e^{-R\beta}$.

We conclude that $\sup_{\mu\in \fE(\sigma,-\beta\omega(d_2(\,\cdot\,,\cO)))}\mu(U)\leq e^{-R\beta}$ for all $\beta\in (\beta_R,+\infty)$. As $R>0$ is arbitrary, (\ref{eq: super-exponential suppression at U}) now follows. As $d_2$ is an ultrametric, we see that $-\omega(d_2(\cdot,\cO))$ is $\omega$-continuous. The proof is now complete.
\end{proof}

\appendix
\section{}\label{ss: Appendix A}
This appendix includes the proofs of several results in Section~\ref{s: TTF for Dini} for the reader's convenience. The proofs of Propositions~\ref{p: main construction} and~\ref{p: uniqueness of tropical eigenvalue} are given after the technical preparations from Proposition~\ref{p:L_continuous} to Lemma~\ref{l: uniform upper bound without transi and surje}. We conclude the appendix with the proofs of Proposition~\ref{p: Mane_potential_properties} to Proposition~\ref{p: condition_for_uniqueness} and Proposition~\ref{p: representation}~(i)(ii).

\begin{prop}  \label{p:L_continuous}
	Let $T\: X\to X$ be an open continuous distance-expanding map with expanding constants $(\lambda,\gamma,\xi)$, $\omega\in\sD$, and $\phi\in H^{\omega}(X,\R)$. Then $\cL_\phi(C(X,\Rminus)) \subseteq C(X,\Rminus)$. Assume, in addition, that $T$ is surjective. Then $\cL_\phi(C(X,\R)) \subseteq C(X,\R)$ and $\cL_\phi\bigl(H^{\assocmodu}(X,\R)\bigr)\subseteq H^{\assocmodu}(X,\R)$.
\end{prop}

\begin{proof}
    Since $\phi\in H^\omega(X,\R)$, $\phi\in C(X,\R)$. It then follows from \cite[Proposition~3.9]{LS26} that $\cL_\phi(C(X,\Rminus))\subseteq C(X,\Rminus)$ and if $T$ is surjective, then $\cL_{\phi}(C(X,\R))\subseteq C(X,\R)$. It suffices to verify that if $T$ is sujective, then $\cL_\phi\bigl(H^{\assocmodu}(X,\R)\bigr)\subseteq H^{\assocmodu}(X,\R)$.

    Assume that $T$ is surjective. Fix $u\in H^{\assocmodu}(X,\R)$ and $x\in X$. Enumerate $T^{-1}(x)=\{x_1,\,x_2,\,\ldots,\,x_n\}$, where $n\in\N$ since $T$ is surjective. By Lemma~\ref{l:expanding systems}, for all $y\in B(x,\xi)$,  denote $y_i\=T_{x_i}^{-1}(y)$ for $1\leq i\leq n$ and consequently $T^{-1}(y)=\{y_i:1\leq i\leq n\}$. Thus, 
    $\abs{\cL_\phi(u)(x)-\cL_\phi(u)(y)}
         \leq\oplusu{1\leq i\leq n}\abs{u(x_i)+\phi(x_i)-u(y_i)-\phi(y_i)}
        \leq \oplusu{1\leq i\leq n}(\assocmodu(d(x_i,y_i))\abs{u}_{\assocmodu}+\omega(d(x_i,y_i))\abs{\phi}_\omega)$.
    Note that $d(x_i,y_i)\leq \lambda^{-1}d(x,y)$ for $1\leq i\leq n$ and $\assocmodu(t)+\omega(t)=\assocmodu(\lambda t)$ for all $t\geq 0$. It follows that for all $y\in B(x,\xi)$,
    $        \abs{\cL_\phi(u)(x)-\cL_\phi(u)(y)}\leq \max\{\abs{u}_{\assocmodu},\,\abs{\phi}_{\omega}\}\assocmodu(d(x,y))$.
    We conclude that $\cL_\phi(u)\in H^{\assocmodu}(X,\R)$ with 
    $        \abs{\cL_\phi(u)}_{\widetilde{\omega}_{\lambda,\xi}}\leq \max\{\abs{u}_{\assocmodu},\,\abs{\phi}_{\omega}\}$.
\end{proof}

\begin{lemma}[Distortion Lemma] \label{l: distortion}
Let $T\: X\to X$ be an open continuous distance-expanding map with expanding constants $(\lambda,\gamma,\xi)$, $\omega\in\sD$, and $\phi\in H^{\omega}(X,\R)$. Then for all $x,\, y\in X$ with $d(x,y)<\xi$, $n\in\N$, and $u\in C(X,\R)$, we have
        \begin{equation}\label{ineq: equi estimate for L_A_n(u)}
            \cL_\phi^n(u)(x)\leq \cL_\phi^n(u)(y)+\abs{\phi}_{\omega}\assocmodu(d(x,y))+\var_u(\lambda^{-n}d(x,y))
        \end{equation}
        and if $x_1\in T^{-n}(x)$, then
        \begin{equation}\label{ineq: difference of the nearby trajectory sum}
        \Absbig{S_n\phi(x_1)-S_n\phi\bigl(T_{x_1}^{-n}(y)\bigr)}\leq\abs{\phi}_{\omega}\assocmodu(d(x,y)).
	   \end{equation} 
Moreover, if $T$ is transitive, then the following hold:
\begin{enumerate}[label=\rm{(\roman*)}]	
\smallskip
\item There exists a constant $C_1>0$ such that for all $n\in\N$ and $x,\, y\in X$,
	\begin{equation*}
		\AbsBig{\oplusu{\ox\in T^{-n}(x)}S_n\phi(\ox)-\oplusu{\oy\in T^{-n}(y)}S_n\phi(\oy)}\leq C_1.
	\end{equation*} 

\item There exists a constant $C_2>0$ depending only on $T$ and $\omega$ such that for all $\phi\in H^\omega(X,\R)$, $u\in H^{\assocmodu}(X,\R)$, and $n\in\N$,
	\begin{equation*}
	\Absbig{\cL_{\phi}^n(u)}_{\assocmodu}\leq C_2(\abs{\phi}_{\omega}+\abs{u}_{\assocmodu}).
	\end{equation*}
\end{enumerate}
\end{lemma}

\begin{proof}
Fix $x,\, y\in X$ with $d(x,y)<\xi$ and $n\in\N$. Note that if $T^{-n}(x)=\emptyset$, then it follows from Lemma~\ref{l:expanding systems} that $T^{-n}(y)=\emptyset$.

Assume that $T^{-n}(x)\neq\emptyset$. Enumerate $T^{-n}(x) =\{x_1,\,x_2,\,\ldots,\,x_k\}$ and $y_i\=T_{x_i}^{-n}(y)$ for all $1\leq i\leq k$. Then Lemma~\ref{l:expanding systems} implies that $T^{-n}(y)=\{y_1,\,y_2,\,\ldots,\,y_k\}$ with $d\bigl(T^l(x_i),T^l(y_i)\bigr)\leq\lambda^{l-n}d(x,y)$ for all $0\leq l\leq n$ and $1\leq i\leq k$. It follows that 
\begin{equation*}
\AbsBig{\oplusu{\ox\in T^{-n}(x)}S_n\phi(\ox)-\oplusu{\oy\in T^{-n}(y)}S_n\phi(\oy)}\leq\oplusu{1\leq i\leq k}\abs{S_n\phi(x_i)-S_n\phi(y_i)}.
\end{equation*} 
Since $\phi\in H^\omega(X,\R)$, by the triangle inequality, it follows that for all $1\leq i\leq k$,
\begin{equation*}
\abs{S_n\phi(x_i)-S_n\phi(y_i)}
\leq\abs{\phi}_{\omega}\bigl(\omega(\lambda^{-n}d(x,y))+\cdots+\omega\bigl(\lambda^{-1}d(x,y)\bigr)\bigr)
\leq\abs{\phi}_{\omega}\assocmodu(d(x,y)).
\end{equation*}
Now (\ref{ineq: difference of the nearby trajectory sum}) is verified. Moreover, it follows from (\ref{eq: iteration of Bousch}) and $d(x_i,y_i)\leq\lambda^{-n}d(x,y)$ that for all $u\in C(X,\R)$,
\begin{equation*}
    \begin{aligned}
    \cL_\phi^n(u)(x)
    &=\oplusu{1\leq i\leq k}(u(x_i)\otimes S_n\phi(x_i))\leq\oplusu{1\leq i\leq k}(u(y_i)\otimes \var_u(\lambda^{-n}d(x,y))\otimes S_n\phi(y_i)\otimes(\abs{\phi}_{\omega}\assocmodu(d(x,y))))\\
    &=\cL_\phi^n(u)(y)+\var_{u}(\lambda^{-n}d(x,y))+\abs{\phi}_{\omega}\assocmodu(d(x,y)).
    \end{aligned}
\end{equation*}
Note that (\ref{ineq: equi estimate for L_A_n(u)}) naturally holds if $T^{-n}(x)=\emptyset$. We conclude that (\ref{ineq: equi estimate for L_A_n(u)}) is verified.

Now assume that $T$ is transitive. It directly follows that $T$ is surjective. Thus, we take $C_0\=\abs{\phi}_{\omega}\assocmodu(\xi)$ and conclude that for all $n\in\N$ and $x,\, y\in X$ with $d(x,y)<\xi$,
\begin{equation}\label{ineq:distortion for nearby points}
		\AbsBig{\oplusu{\ox\in T^{-n}(x)}S_n\phi(\ox)-\oplusu{\oy\in T^{-n}(y)}S_n\phi(\oy)}\leq C_0.
\end{equation}

The transitivity of $T$ implies that there exists $N_{\xi}\in\N$ such that for all $x,\, y\in X$, there exists an integer $m$ satisfying $0\leq m\leq N_{\xi}$ and $T^m(B(x,\xi))\cap B(y,\xi)\neq\emptyset$. Now adopting an argument that is verbatim the same as that in the proof of \cite[Lemma~3.10]{LS26} with $\phi$ in place of $A$, we get that for all $x,\,y\in X$ and $n\in\N$,
\begin{equation}\label{ineq:distortion for all points}
	\AbsBig{\oplusu{\ox\in T^{-n}(x)}S_n\phi(\ox)-\oplusu{\oy\in T^{-n}(y)}S_n\phi(\oy)}\leq 2C_0+N_{\xi}\bigl(\sup_{z\in X}\phi(z)-\inf_{z\in X}\phi(z)\bigr)\leq C_1,
\end{equation}
with $C_1\=2C_0+N_{\xi}\abs{\phi}_\omega\omega(\diam X)$. 

Fix $\phi\in H^{\omega}(X,\R)$ and $u\in H^{\assocmodu}(X,\R)$. Proposition~\ref{p:L_continuous} implies that $\cL_\phi^n(u)\in H^{\assocmodu}(X,\R)$ for all $n\in\N$. Now it suffices to give the estimate for $\Absbig{\cL_{\phi}^n(u)}_{\assocmodu}$.  

It follows from (\ref{ineq: equi estimate for L_A_n(u)}) that for all $n\in\N$,
$\Absbig{\cL_{\phi}^n(u)}_{\assocmodu,\xi}\leq\abs{\phi}_{\omega}+\abs{u}_{\assocmodu}$.
For $x,\, y\in X$ with $d(x,y)\geq\xi$, it follows from (\ref{eq: iteration of Bousch}) and~(\ref{ineq:distortion for all points}) that for all $n\in\N$,
\begin{equation*}
		 \Absbig{\cL_{\phi}^n(u)(x)-\cL_{\phi}^n(u)(y)}\leq C_1+\sup_{x\in X} u(x)-\inf_{x\in X} u(x) \leq \Bigl(C_1+\sup_{x\in X} u(x)-\inf_{x\in X} u(x)\Bigr) \frac{\assocmodu(d(x,y))}{\assocmodu(\xi)}.
\end{equation*}
We conclude that for all $n\in\N$,
\begin{equation}\label{ineq:Lipschitz constant estimate for L_A^n(u)}
\Absbig{\cL_\phi^n(u)}_{\assocmodu}\leq\max\Bigl\{\abs{\phi}_{\omega}+\abs{u}_{\assocmodu},\, \Bigl(C_1+\sup_{x\in X} u(x)-\inf_{x\in X} u(x)\Bigr) \Big/ \assocmodu(\xi)\Bigr\}.
\end{equation}
Now we want to choose $C_2(\phi,u)$ to be certain constant satisfying $\Absbig{\cL_{\phi}^n(u)}_{\assocmodu}\leq C_2(\phi,u)(\abs{\phi}_{\omega}+\abs{u}_{\assocmodu})$ for a specific pair $(\phi,u)\in H^{\omega}(X,\R)\times H^{\assocmodu}(X,\R)$ and all $n\in\N$. 
More precisely, if $\phi\in H^\omega(X,\R)$ and $u\in H^{\assocmodu}(X,\R)$ are two constant functions, then for each $n\in\N$, $\cL_{\phi}^n(u)$ is a constant function, and $C_2(\phi,u)$ can be an arbitrary positive number. 
Now suppose that there is a nonconstant function among $\phi$ and $u$, i.e., $\abs{\phi}_{\omega}+\abs{u}_{\assocmodu}>0$. By (\ref{ineq:Lipschitz constant estimate for L_A^n(u)}), we can take $C_2(\phi,u)\=\max\bigl\{1,\,\frac{C_1+\sup_{x\in X} u(x)-\inf_{x\in X} u(x)}{\assocmodu(\xi)(\abs{\phi}_{\omega}+\abs{u}_{\assocmodu})}\bigr\}$.

Moreover, $u\in H^{\assocmodu}(X,\R)$ implies that 
	$		\sup_{x\in X} u(x)-\inf_{x\in X} u(x)\leq\abs{u}_{\assocmodu}\assocmodu(\diam X)$.
	Recall that $C_1=2C_0+N_{\xi}\abs{\phi}_{\omega}\omega(\diam X)$ and $C_0=\abs{\phi}_{\omega}\assocmodu(\xi)$. We conclude that
\begin{equation*}
C_2(\phi,u)
\leq \max\biggl\{1\,,\frac{2C_0+N_{\xi}\abs{\phi}_{\omega}\omega(\diam X)+\abs{u}_{\assocmodu}\assocmodu(\diam X)}{\assocmodu(\xi)(\abs{\phi}_{\omega}+\abs{u}_{\assocmodu})}\biggr\}
\leq C_2,
\end{equation*} 
where we define $C_2\=2+\frac{N_{\xi} \omega(\diam X)+\assocmodu(\diam X)}{\assocmodu(\xi)}$. Therefore, 
$\Absbig{\cL_\phi^n(u)}_{\assocmodu}\leq C_2(\abs{\phi}_{\omega}+\abs{u}_{\assocmodu})$ 
for all $\phi\in H^\omega(X,\R)$, $u\in H^{\assocmodu}(X,\R)$, and $n\in\N$. The proof is now complete.
\end{proof}

\begin{lemma}\label{l: projective estimate}
    Let $T\: X\to X$ be an open continuous distance-expanding map with expanding constants $(\lambda,\gamma,\xi)$, $\omega\in\sD$, and $\phi\in H^{\omega}(X,\R)$. Then
    \begin{equation*}
		\cL_\phi^k(u_2)\leq\norm{u_1-u_2}_{C^0}\otimes\cL_\phi^k(u_1)
	\end{equation*}
    for all $k\in\N$ and $u_1,u_2\in C(X,\R)$. Moreover, if $T$ is surjective, then for all $u_1,u_2\in C(X,\R)$,
    \begin{equation*}
        \norm{\cL_\phi(u_1)-\cL_\phi(u_2)}_{C^0}\leq\norm{u_1-u_2}_{C^0}.
    \end{equation*}
\end{lemma}
\begin{proof}
    This lemma is a direct consequence of the tropical linearity of $\cL_\phi$ (cf.~\cite[Lemma~3.1]{LS26}) and Proposition~\ref{p:L_continuous}.
\end{proof}

\begin{lemma}\label{l: uniform upper bound without transi and surje}
Let $T\: X\to X$ be an open continuous distance-expanding map with expanding constants $(\lambda,\gamma,\xi)$ and $\omega\in\sD$. Then there exists $L_1>0$ depending only on $T$ and $\omega$ such that for all $x\in X$, $n\in\N$, and $\phi\in H^\omega(X,\R)$, we have
$S_n\toverline{\phi}(x)\leq L_1\abs{\phi}_{\omega}$.
\end{lemma}
\begin{proof}
    The proof of this lemma is verbatim the same as that of \cite[Lemma~3.12]{LS26}, up to replacing $(\lambda,\eta,\xi,A)$ with $(\lambda,2\gamma,\xi,\phi)$, $\abs{A}_{d^\alpha}$ with $\abs{\phi}_\omega$, $\xi^\alpha(\lambda^\alpha-1)^{-1}$ with $\assocmodu(\xi)$, and $(\diam X)^\alpha$ with $\omega(\diam X)$.
\end{proof}

Now we are ready to prove Propositions~\ref{p: main construction} and~\ref{p: uniqueness of tropical eigenvalue}. 

\begin{proof}[\bf Proof of Proposition~\ref{p: main construction}]

\smallskip
    (i) Let $L_1$ be the constant in Lemma~\ref{l: uniform upper bound without transi and surje} depending only on $T$ and $\omega$. It immediately follows from \cite[Lemma~3.1~(ii)]{LS26} and Lemma~\ref{l: uniform upper bound without transi and surje} that $v_{\tmif}\fless L_1\abs{\phi}_{\omega}$ and we can take $L_2\=L_1$. By Lemma~\ref{l: projective estimate}, we have for all $n\in\N$, $\cL_{\toverline{\phi}}^n(\tmif)-\norm{u-\tmif}_{C^0}\fless \cL_{\toverline{\phi}}^n(u)\fless\norm{u-\tmif}_{C^0}+\cL_{\toverline{\phi}}^n(\tmif)$. Consequently, 
    \begin{equation}\label{ineq: comparison of v0 and vu}
    v_{\tmif}-\norm{u-\tmif}_{C^0}\fless v_u\fless \norm{u-\tmif}_{C^0}+v_{\tmif}.
    \end{equation}
    It follows that $v_u\leq L_2\abs{\phi}_{\omega}+\norm{u}_{C^0}$. By (\ref{ineq: equi estimate for L_A_n(u)}) and~(\ref{eq: defi of main construction}), we see that for all $x,\,y\in X$ with $d(x,y)<\xi$,
    \begin{equation*}
        v_u(x)\leq v_u(y)+\abs{\phi}_{\omega}\assocmodu(d(x,y))+\lim_{n\to+\infty}\var_{u}(\lambda^{-n}d(x,y))=v_u(y)+\abs{\phi}_{\omega}\assocmodu(d(x,y)),
    \end{equation*}
    where the equality follows from the uniform continuity of $u\in C(X,\R)$.
    
    \smallskip
    
    (ii) It follows from (i) that $v_{u}\in C(X,\Rminus)$. Thus, by \cite[Lemma~3.1~(iii)(iv)]{LS26}, 
    \begin{equation*}
        \cL_{\toverline{\phi}}(v_{u})
        =\lim_{n\to+\infty}\cL_{\toverline{\phi}}\Bigl(\sup_{k\geq n}\cL_{\toverline{\phi}}^k(u)\Bigr)
        =\lim_{n\to+\infty}\sup_{k\geq n}\cL_{\toverline{\phi}}^{k+1}(u)
        =v_{u}.
    \end{equation*}
    
    Now it suffices to show that $v_{u}\neq\taif$, i.e., $v_{u}(y)\in\R$ for some $y\in X$. Denote $a_n\=\sup_{x\in X}\cL_{\phi}^n(\tmif)(x)$ for all $n\in\N$. By Proposition~\ref{p:L_continuous}, $\cL_\phi^n(\tmif)\in C(X,\Rminus)$ and thus there exists $x_n\in X$ such that $\cL_\phi^n(\tmif)(x_n)=a_n$ since $X$ is compact. By \cite[Lemma~3.1~(ii)]{LS26},
    $a_n=\sup_{x\in X}S_n\phi(x)\in\R$. Thus, by \cite[Proposition~2.1]{Je06}, 
    $\energy=\limsup_{n\to+\infty}\frac{a_n}{n}$.
    Observe that $\{a_n\}_{n\in\N}$ is subadditive. It follows that $\lim_{n\to+\infty}\frac{a_n}{n}=\inf_{n\in\N}\frac{a_n}{n}=\energy\in\R$ and consequently $\cL_{\toverline{\phi}}^n(\tmif)(x_n)=a_n-n\energy\geq 0$ for all $n\in\N$. Since $X$ is compact, there exists a subsequence $\{x_{n_k}\}_{k\in\N}$ converging to $y\in X$. Thus, there exists $N\in\N$ such that $d(x_{n_k},y)<\xi$ for all $k\geq N$. It then follows from (\ref{ineq: equi estimate for L_A_n(u)}) and $\cL_{\toverline{\phi}}^n(\tmif)(x_n)\geq 0$ that 
    \begin{equation*}
        \cL_{\toverline{\phi}}^{n_k}(\tmif)(y)\geq\cL_{\toverline{\phi}}^{n_k}(\tmif)(x_{n_k})-\abs{\phi}_{\omega}\assocmodu(d(y,x_{n_k}))\geq -\abs{\phi}_{\omega}\assocmodu(d(y,x_{n_k})).
    \end{equation*}
    Hence, 
    $        v_{\tmif}(y)=\limsup\limits_{n\to+\infty}\cL_{\toverline{\phi}}^n(\tmif)(y)
        \geq \limsup\limits_{k\to+\infty}\cL_{\toverline{\phi}}^{n_k}(\tmif)(y)
        \geq \limsup\limits_{k\to+\infty} (-\abs{\phi}_{\omega}\assocmodu(d(y,x_{n_k})))=0$.
    Since $v_{\tmif}\in C(X,\Rminus)$, we see that $v_{\tmif}(y)\in\R$. It follows from (\ref{ineq: comparison of v0 and vu}) that $v_u(y)\in\R$.
    
    \smallskip
    
    (iii) Let $C_2>0$ be the constant from Lemma~\ref{l: distortion}. Denote $D_1\=C_2\abs{\phi}_{\omega}\assocmodu(\diam X)$.  
    
    Assume that $T$ is transitive. It follows from Lemma~\ref{l: distortion}~(ii) that $\Absbig{\cL_{\phi}^n(\tmif)}_{\assocmodu}\leq C_2\abs{\phi}_{\omega}$ for all $n\in\N$. Recall that $a_n=\sup_{x\in X}\cL_\phi^n(\tmif)(x)\in\R$. Thus, for all $n\in\N$,
    \begin{equation}\label{ineq: variation implies lower bound for L_A^n}
    	(a_n-D_1)\otimes\tmif\fless\cL_{\phi}^n(\tmif).
    \end{equation}
    Note that a tropical linear map between $\Rm$-subsemimodules of $\Rm^X$ preserves the order (cf.~\cite[Lemma~2.6]{LS26}).It then follows from the tropical linearity of $\cL_\phi$ (\cite[Lemma~3.1~(i)(iii)]{LS26}) that for all $n,\, m\in\N$,
    $(a_n-D_1)\otimes\cL_{\phi}^m(\tmif)
    =\cL_{\phi}^m((a_n-D_1)\otimes\tmif)\fless\cL_{\phi}^{n+m}(\tmif)$.
    It follows that $(a_n-D_1)\otimes a_m\leq a_{n+m}$ for all $n,\, m\in\N$, i.e., $\{D_1-a_n\}_{n\in\N}$ is subadditive. Thus, 
    $\inf_{n\in\N}(D_1-a_n)/n=\lim_{n\to+\infty}(D_1-a_n)/n=-\energy$.
    We conclude that $0\leq a_n-n\energy\leq D_1$ for all $n\in\N$. Thus by (\ref{ineq: variation implies lower bound for L_A^n}), for all $n\in\N$, $\Normbig{\cL_{\toverline{\phi}}^n(\tmif)}_{C^0}\leq D_1$.
   
    Recall that $\Absbig{\cL_{\toverline{\phi}}^n(\tmif)}_{\assocmodu}=\Absbig{\cL_{\phi}^n(\tmif)}_{\assocmodu}\leq C_2\abs{\phi}_{\omega}$ for all $n\in\N$. Thus, $\bigl\{\cL_{\toverline{\phi}}^n(\tmif)\bigr\}_{n\in\N}$ is an equicontinuous and uniformly bounded family, and consequently, so is 
    $\bigl\{\sup_{k\geq n}\cL_{\toverline{\phi}}^k(\tmif)\bigr\}_{n\in\N}$. Hence, $v_{\tmif}$, as the pointwise decreasing limit of $\sup_{k\geq n}\cL_{\toverline{\phi}}^k(\tmif)$ as $n\to +\infty$, is the uniform limit of $\sup_{k\geq n}\cL_{\toverline{\phi}}^k(\tmif)$ by the Arzel\`{a}--Ascoli theorem. Now it follows clearly that $\norm{v_{\tmif}}_{C^0}\leq D_1$ and $\abs{v_{\tmif}}_{\assocmodu}\leq C_2\abs{\phi}_{\omega}$. Take $C_3\=C_2$. It follows from (\ref{ineq: comparison of v0 and vu}) and~(i) that $v_u\in C(X,\R)$ with $\norm{v_u}_{C^0}\leq D_1+\norm{u}_{C^0}$. Furthermore, it follows from (i) that $\abs{v_u}_{\assocmodu,\xi}\leq\abs{\phi}_{\omega}$ and consequently $v_u\in H^{\assocmodu}(X,\R)$.
\end{proof}

\begin{proof}[\bf Proof of Proposition~\ref{p: uniqueness of tropical eigenvalue}]

The proof of this proposition is verbatim the same as that of \cite[Proposition~3.8]{LS26} up to replacing $A$ with $\phi$.
\end{proof}

For the proofs of Proposition~\ref{p: Mane_potential_properties} to Proposition~\ref{p: condition_for_uniqueness} and Proposition~\ref{p: representation}~(i)(ii), we first establish the two following lemmas.

\begin{lemma}\label{l: maximizing measures supported on Aubry set}
    Let $T\:X \rightarrow  X$ be a transitive expanding covering map, $\omega\in\sD$, and $\phi\in H^\omega(X,\R)$. Then an invariant Borel probability measure is a $\phi$-maximizing measure if and only if it is supported on the Aubry set $\Omega_\phi$.
\end{lemma}
\begin{proof}
    This lemma follows from the existence of a sub-action in the same way as in the proof of \cite[Lemma~3.15]{LS26}, where the relevant sub-action is provided by \cite[Theorem~1]{Bou01}.
\end{proof}

\begin{lemma}\label{l: mane potential is upper semi-continuous}
    Let $T\:X \rightarrow  X$ be an open continuous distance-expanding map, $\omega\in\sD$, and $\phi\in H^\omega(X,\R)$. Then the Ma\~{n}\'{e} potential $S_\phi$ is upper semi-continuous. 
\end{lemma}

\begin{proof}
    It suffices to show that for all $x,\,y\in X$,
    \begin{equation}\label{ineq: upper semi-continuity of mane potential}
        \lim_{\epsilon\to 0^+}\oplusu{\substack{z\in B(x,\epsilon)\\ w\in B(y,\epsilon)}}S_\phi(z,w)\leq S_\phi(x,y).
    \end{equation}
    Fix $\delta>0$, $z\in B(x,\delta/2)$, and $w\in B(y,\delta/2)$. It immediately follows from Definition~\ref{d: Mane potential} that 
    \begin{equation*}
      S_\phi(z,w)\leq\oplusu{n\in\N}  \oplusu{\substack{d(z_1,z)\leq\delta/2\\d(T^n(z_1),w)\leq\delta/2}}S_n\toverline{\phi}(z_1)\leq \oplusu{n\in\N}\oplusu{\substack{d(z_1,x)\leq\delta\\ d(T^n(z_1),y)\leq\delta}}S_n\toverline{\phi}(z_1).
    \end{equation*}
    Thus, 
    $        \oplusu{\substack{z\in B(x,\delta/2)\\ w\in B(y,\delta/2)}}S_\phi(z,w)\leq \oplusu{n\in\N}\oplusu{\substack{d(z_1,x)\leq\delta\\ d(T^n(z_1),y)\leq\delta}}S_n\toverline{\phi}(z_1)$.
As $\delta\to 0^+$ in the last inequality, we see that (\ref{ineq: upper semi-continuity of mane potential}) follows from Definition~\ref{d: Mane potential}.
\end{proof}

\begin{proof}[\bf Proof of Proposition~\ref{p: Mane_potential_properties}]
    The proof of this proposition is verbatim the same as that of \cite[Proposition~3.18]{LS26}, up to replacing $A$ with $\phi$ and applying (\ref{ineq: difference of the nearby trajectory sum}) for the local distortion. 
\end{proof}

\begin{proof}[\bf Proof of Lemma~\ref{l: mane potential gives eigenmeasures}]
    Let $(\lambda,\gamma,\xi)$ be expanding constants of $T\:X\to X$. The condition $T(x)\neq y$ implies that for all $\epsilon\in \bigl(0,\frac{1}{2}\min_{w\in T^{-1}(y)}d(x,w)\bigr)\cap (0,\xi)$, if $z\in B(x,\epsilon)$ and $n\in\N$ satisfy $d(T^n(z),y)<\epsilon$, then $n>1$. Thus, (\ref{eq: Mane potential (_,y) almost eigendensity}) follows immediately from the definition of $S_\phi$ (Definition~\ref{d: Mane potential}) and the continuity of $T$. This proves the first part of the lemma.

For the second part of the lemma, note that $S_\phi(\,\cdot\,,\,\cdot\,) \: X\times X \rightarrow \Rminus$ is upper semi-continuous by Lemmas~\ref{l: uniform upper bound without transi and surje} and~\ref{l: mane potential is upper semi-continuous}. Thus, $S_\phi(\,\cdot\,,y)\in\density(X)$ for all $y\in \Omega_\phi$ and it suffices to show $S_\phi(T(x),y)=S_\phi(x,y)-\toverline{\phi}(x)$ for all $y\in\Omega_\phi$ and $x\in T^{-1}(y)$.
 Fix $y\in\Omega_\phi$ and $x\in T^{-1}(y)$. By Proposition~\ref{p: Mane_potential_properties}~(iv), we have $S_\phi(T(x),y)=S_\phi(y,y)=0$. Thus, it suffices to prove $S_\phi(x,T(x))=\toverline{\phi}(x)$. 
 
 \smallskip
 
 \emph{Claim.} $S_\phi(x,T(x))=\toverline{\phi}(x)$ for all $x\in X$.
 
 \smallskip
 
 Fix $x\in X$. On the one hand, $S_n\toverline{\phi}(z)=\toverline{\phi}(x)$ for $n\=1, z\=x$ and it follows from Definition~\ref{d: Mane potential} that $S_\phi(x,T(x))\geq \toverline{\phi}(x)$. 

On the other hand, since $T$ is continuous, for every $\epsilon>0$, there exists $\eta(\epsilon)\in(0,\epsilon)$ such that $d(y,x)\leq\eta(\epsilon)$ implies $d(T(y),T(x))\leq\epsilon$ for all $y\in X$. Thus,
\begin{align*}
\oplusu{n\in\N}\oplusu{\substack{d(z,x)\leq\eta(\epsilon)\\ d(T^n(z),T(x))
	\leq\eta(\epsilon)}}S_n\toverline{\phi}(z)
	&\leq\Bigl(\oplusu{d(z,x)\leq\eta(\epsilon)}\toverline{\phi}(z)\Bigr)\otimes\Bigl(\oplusu{n\in\N}\oplusu{\substack{d(z,x)\leq\eta(\epsilon)\\ d(T^{n}(z),T(x))\leq\eta(\epsilon)}}S_{n-1}\toverline{\phi}(T(z))\Bigr)\\
&\leq\Bigl(\oplusu{d(z,x)\leq\eta(\epsilon)}\toverline{\phi}(z)\Bigr)\otimes\Bigl(\oplusu{m\in\N_0}\oplusu{\substack{d(\tz,T(x))\leq\epsilon\\ d(T^m(\tz),T(x))\leq\epsilon}}S_m\toverline{\phi}(\tz)\Bigr).
\end{align*}

Recall that $S_0\toverline{\phi}(\tz)=0$ for all $\tz\in X$. As $\epsilon\to 0^+$, we get
\begin{equation*}
S_\phi(x,T(x))\leq\toverline{\phi}(x)\otimes(0\oplus S_\phi(T(x),T(x))).
\end{equation*}
By Proposition~\ref{p: Mane_potential_properties}~(iv), $S_\phi(T(x),T(x))\leq 0$. We conclude that $S_\phi(x,T(x))\leq\toverline{\phi}(x)$ and the claim is now verified. 
\end{proof}

\begin{proof}[\bf Proof of Proposition~\ref{p: representation}~(i)(ii)]

\smallskip
(i) Fix $u\in \es{\cL_\phi}{C(X,\Rminus)}{\energy}$. By (\ref{eq: defi of the Bousch operator}), $u(x)=\oplusu{z\in T^{-1}(x)}\bigl(u(z)+\toverline{\phi}(z)\bigr)$ for all $x\in X$.
	It follows from Proposition~\ref{p: Mane_potential_properties}~(i) that  $u(\,\cdot\,)\geq\oplusu{x\in\Omega_\phi}(u(x)\otimes S_\phi(x,\,\cdot\,))$. It suffices to find an Aubry point $x_y$ for each $y\in X$ such that $u(x_y)\otimes S_\phi(x_y,y)\geq u(y)$. Fix $y\in X$. If $u(y)=-\infty$, then the above desired inequality trivially holds for all $x_y\in\Omega_\phi$. Now assume that $u(y)\in\R$.
 
	 By Lemma~\ref{l:expanding systems}, $T^{-1}(x)$ is finite for all $x\in X$. Since $u(y)=\oplusu{z\in T^{-1}(y)}\bigl(u(z)+\toverline{\phi}(z)\bigr)$, there exists $y_1\in T^{-1}(y)$ such that $u(y)=u(y_1)+\toverline{\phi}(y_1)$. Then since $u(y_1)=\oplusu{z\in T^{-1}(y_1)}\bigl(u(z)+\toverline{\phi}(z)\bigr)$, there exists $y_2\in T^{-1}(y_1)$ such that $u(y_1)=u(y_2)+\toverline{\phi}(y_2)$. Repeating this process recursively, we get a sequence $\{y_k\}_{k\in\N}$ in $X$ satisfying $T(y_{k+1})=y_k$ and $u(y_k)=u(y_{k+1})+\toverline{\phi}(y_{k+1})$ for all $k\in\N$. 
	
	\smallskip
    
	\emph{Claim.} Every accumulation point $x_y$ of the sequence $\{y_k\}_{k\in\N}$ is an Aubry point and satisfies $u(x_y)+S_\phi(x_y,y)\geq u(y)$.
    
	\smallskip
	
	Suppose the subsequence $\{y_{n_k}\}_{k\in\N}$ converges to $x_y$ with $n_k\to+\infty$ as $k\to+\infty$. Without loss of generality, we assume that $n_{k+1}>n_k+k$ for all $k\in\N$. It follows from Lemma~\ref{l: uniform upper bound without transi and surje} and the choice of $y_k$ that there exists $D>0$ such that $u(y_k)=u(y)-S_k\toverline{\phi}(y_k)\geq u(y)-D$ for all $k\in\N$. Since $u\in C(X,\Rminus)$ and $y_{n_k}$ converges to $x_y$ as $k\to+\infty$, we see that $u(x_y)\in\R$.
	
	Fix $\epsilon>0$. The continuity of $u$ together with $u(x_y)\in\R$ implies that there exists $\eta\in(0,\epsilon/2)$ such that $\abs{u(x_y)-u(z)}\leq\epsilon/2$ for all $z\in B(x_y,\eta)$. Thus, there exists $N\in\N$ such that for all $k\geq N$, $d(y_{n_k},x_y)<\eta<\epsilon/2$ and $d(y_{n_{k+1}},x_y)<\eta<\epsilon/2$. Denote $m_k\coloneqq n_{k+1}-n_k\in\N$ for all $k\in\N$. Then for all $k\geq N$,
	 \begin{equation*}
		\Absbig{S_{m_k}\toverline{\phi}(y_{n_{k+1}})}=\abs{u(y_{n_{k}})-u(y_{n_{k+1}})}\leq2\cdot(\epsilon/2)=\epsilon.
	\end{equation*}	
	 Since $d(y_{n_k},y_{n_{k+1}})\leq d(y_{n_k},x_y)+d(x_y,y_{n_{k+1}})\leq\epsilon$, by Definition~\ref{d: Aubry set}, we conclude that $x_y$ is an Aubry point. Moreover, $d(y_{n_k},x_y)<\epsilon/2$ implies that for all $k\geq N$,
	 \begin{equation*}
	 	u(y)
        =u(y_{n_k})\otimes S_{n_k}\toverline{\phi}(y_{n_k})
        \leq (u(x_y)+\epsilon/2)\otimes\Bigl(\oplusu{n\in\N}\oplusu{\substack{d(z,x_y)
        \leq\epsilon/2\\ d(T^n(z),y)\leq\epsilon/2}}S_n\toverline{\phi}(z)\Bigr).
	 	\end{equation*}
 	 As $\epsilon\to 0^+$ in the above inequality, we get $u(y)\leq u(x_y)\otimes S_\phi(x_y,y)$. Now the claim is verified and it follows that $u(\,\cdot\,)=\oplusu{x\in\Omega_\phi}(u(x)\otimes S_\phi(x,\,\cdot\,))$.

     \smallskip
     (ii)  Assume that $b\in\esbig{\cL_\phi^{\ostar}}{\density(X)}{\energy}$. We first show that
\begin{equation}\label{eq: representation of eigendensity as functionals}
\oplusu{x\in X}(u(x)\otimes b(x))=\oplusu{x\in X,y\in\Omega_\phi}(u(x)\otimes S_\phi(x,y)\otimes b(y))
\end{equation} for all $u\in C(X,\R)$. Now fix an arbitrary $u\in C(X,\R)$. 

We reduce the proof of $(\ref{eq: representation of eigendensity as functionals})$ to functions in $\es{\cL_\phi}{C(X,\Rminus)}{\energy}$ using Proposition~\ref{p: main construction} and then apply (i).

For the left-hand side of (\ref{eq: representation of eigendensity as functionals}), since $b\in\esbig{\cL_\phi^{\ostar}}{\density(X)}{\energy}$, it follows from (\ref{eq: defi of the Bousch operator}) and~(\ref{eq: defi of the adjoint Bousch operator}) that 
\begin{equation}\label{eq: functional induced by eigendensity}
\begin{aligned}
	\oplusu{x\in X}(\cL_\phi(u)(x)\otimes b(x))=\oplusu{y\in X}\bigl(u(y)\otimes \cL_{\phi}^{\ostar}(b)(y)\bigr)=\oplusu{y\in X}(u(y)\otimes\energy\otimes b(y)).
\end{aligned}
\end{equation} 

By repeated use of (\ref{eq: functional induced by eigendensity}), we get 
\begin{equation}\label{eq: pre for limit for iterations of L_A_n(u)}
\oplusu{x\in X}(u(x)\otimes b(x))
=\oplusu{x\in X}\bigl(\cL_{\toverline{\phi}}^n(u)(x)\otimes b(x)\bigr)
=\oplusu{x\in X}\Bigl(\Bigl(\oplusu{m\geq n}\cL_{\toverline{\phi}}^m(u)(x)\Bigr)\otimes b(x)\Bigr)
\end{equation}
for all $n\in\N$. It follows from (\ref{eq: iteration of Bousch}) and Lemma~\ref{l: uniform upper bound without transi and surje} that $\bigl\{\oplusu{m\geq n}\cL_{\toverline{\phi}}^m(u)\bigr\}_{n\in\N}$ has a uniform upper bound. It follows from (\ref{ineq: equi estimate for L_A_n(u)}) that $\bigl\{\oplusu{m\geq n}\cL_{\toverline{\phi}}^m(u)\bigr\}_{n\in\N}$ is equicontinuous. Recall that $v_u$ is the pointwise decreasing limit of $\oplusu{m\geq n}\cL_{\toverline{\phi}}^m(u)$ as $n\to+\infty$ (cf.~(\ref{eq: defi of main construction})). Thus, $\bigl\{\exp \bigl(\oplusu{m\geq n}\cL_{\toverline{\phi}}^m(u)\bigr)\bigr\}_{n\in\N}$ is a normal family and uniformly converges to $e^{v_u}$ as $n\to+\infty$. It then follows from (\ref{eq: pre for limit for iterations of L_A_n(u)}) that 
\begin{equation*}
    \begin{aligned}
    &\exp\Bigl(\oplusu{x\in X}(u(x)\otimes b(x))\Bigr)
    =\lim_{n\to+\infty}\exp\Bigl(\oplusu{x\in X}\Bigl(\Bigl(\oplusu{m\geq n}\cL_{\toverline{\phi}}^m(u)(x)\Bigr)\otimes b(x)\Bigr)\Bigr)\\
    &\qquad =\lim_{n\to+\infty}\oplusu{x\in X}\Bigl(\exp\Bigl(\oplusu{m\geq n}\cL_{\toverline{\phi}}^m(u)(x)\Bigr)\cdot e^{b(x)}\Bigr)
    =\oplusu{x\in X}\bigl(e^{v_u(x)}\cdot e^{b(x)}\bigr)
    =\exp\Bigl(\oplusu{x\in X}(v_u(x)\otimes b(x))\Bigr),
    \end{aligned}
\end{equation*}
where the third equality holds since $\exp b$ is either $\tinftyf$ (when $b=\tinftyf$) or has an upper bound (when $b\in\density(X)\smallsetminus\{\tinftyf\}$).

We conclude that 
\begin{equation}\label{eq: reduction to eigenfunctions}
	\oplusu{x\in X}(u(x)\otimes b(x))=\oplusu{x\in X}(v_u(x)\otimes b(x)).
\end{equation}

For the right-hand side of (\ref{eq: representation of eigendensity as functionals}), by Lemma~\ref{l: mane potential gives eigenmeasures}, $S_\phi(\,\cdot\,,y)\in\esbig{\cL_\phi^{\ostar}}{\density(X)}{\energy}$ for all $y\in\Omega_\phi$. It follows from Proposition~\ref{p: Mane_potential_properties}~(iv) that $S_\phi(y,y)=0$ and $S_\phi(\,\cdot\,,y)\in\density(X)\smallsetminus\{\tinftyf\}$ for all $y\in\Omega_\phi$. Thus, we can substitute $b(\,\cdot\,)$ in (\ref{eq: reduction to eigenfunctions}) with  $S_\phi(\,\cdot\,,y)$ and it follows that for all $y\in\Omega_\phi$,
\begin{equation*}
\oplusu{x\in X}(u(x)\otimes S_\phi(x,y))=\oplusu{x\in X}(v_u(x)\otimes S_\phi(x,y)).
\end{equation*} 
Hence,
\begin{align*}
&\oplusu{x\in X,y\in\Omega_\phi}(u(x)\otimes S_\phi(x,y)\otimes b(y))
=\oplusu{y\in\Omega_\phi}\Bigl(\oplusu{x\in X}(S_\phi(x,y)\otimes u(x))\otimes b(y)\Bigr)\\
&\qquad =\oplusu{y\in\Omega_\phi}\Bigl(\oplusu{x\in X}(S_\phi(x,y)\otimes v_u(x))\otimes b(y)\Bigr)
=\oplusu{x\in X,y\in\Omega_\phi}(v_u(x)\otimes S_\phi(x,y)\otimes b(y)).
\end{align*}

We have finished the first step of reduction. Since $v_u\in\es{\cL_\phi}{C(X,\Rminus)}{\energy}$ (cf.~Proposition~\ref{p: main construction})~(ii)), it suffices to prove (\ref{eq: representation of eigendensity as functionals}) for all $v\in\es{\cL_\phi}{C(X,\Rminus)}{\energy}$.

It follows from the discussions below that for all $v\in\es{\cL_\phi}{C(X,\Rminus)}{\energy}$,
\begin{align*}
&\oplusu{x\in X,y\in\Omega_\phi}(v(x)\otimes S_\phi(x,y)\otimes b(y))
=\oplusu{x\in X,y,z\in\Omega_\phi}(v(z)\otimes S_\phi(z,x)\otimes S_\phi(x,y)\otimes b(y))\\
&\qquad =\oplusu{y,z\in\Omega_\phi}(v(z)\otimes  S_\phi(z,y)\otimes b(y))
=\oplusu{y\in X,z\in\Omega_\phi}(v(z)\otimes S_\phi(z,y) \otimes b(y))
=\oplusu{y\in X}(v(y)\otimes b(y)).
\end{align*}
Here the first and the fourth identities follow from $v(x)=\oplusu{z\in\Omega_\phi}(v(z)\otimes S_\phi(z,x))$ for all $v\in\es{\cL_\phi}{C(X,\Rminus)}{\energy}$ and $x\in X$ (cf.~ (i)). The second identity follows from $S_\phi(z,x)\otimes S_\phi(x,y)\leq S_\phi(z,y)$ for $x,\,y,\,z\in X$ (cf.~Proposition~\ref {p: Mane_potential_properties}~(iii)) and $S_\phi(z,z)=0$ for $z\in\Omega_\phi$ (cf.~Proposition~\ref{p: Mane_potential_properties}~(iv)).
For the third identity, we remark that
$S_\phi(z,y)\otimes b(y)\leq b(z)$ for all $y,\,z\in X$ (cf.~Proposition~\ref{p: Mane_potential_properties}~(ii)) and if $z\in \Omega_\phi$, then the equality is achieved at $y=z\in\Omega_\phi$ as $S_\phi(z,z)=0$ (cf.~Proposition~\ref{p: Mane_potential_properties}~(iv)). Thus, for all $z\in\Omega_\phi$,
\begin{equation*}
\oplusu{y\in\Omega_\phi}(S_\phi(z,y)\otimes b(y))=\oplusu{y\in X}(S_\phi(z,y)\otimes b(y))=b(z).
\end{equation*}
We conclude that (\ref{eq: representation of eigendensity as functionals}) holds for all $v\in\es{\cL_\phi}{C(X,\Rminus)}{\energy}$ and thus for all $u\in C(X,\R)$.

Now assume in addition that $b\neq\tinftyf$. We need to show $b(x)=\oplusu{y\in\Omega_\phi}(S_\phi(x,y)\otimes b(y))$ for all $x\in X$. Recall that for all $b_1\:X\to\Rm$, there exists a unique $b_2\in\density(X)$ equivalent to $b_1$ (cf.~Remark~\ref{r: equivalence of densities}). Since we have shown $b(\,\cdot\,)$ is equivalent to $\oplusu{y\in\Omega_\phi}(S_\phi(\,\cdot\,,y)\otimes b(y))$, it suffices to show that $\oplusu{y\in\Omega_\phi}(S_\phi(\,\cdot\,,y)\otimes b(y))\in\density(X)$. Since $b\in\density(X)\smallsetminus\{\tinftyf\}$, $b$ has an upper bound. Note that $S_\phi$ also has an upper bound (by Definition~\ref{d: Mane potential} and Lemma~\ref{l: uniform upper bound without transi and surje}). We see that $\oplusu{y\in\Omega_\phi}(S_\phi(\,\cdot\,,y)\otimes b(y))\: X\to\Rminus$ has an upper bound. Thus it suffices to show that $\oplusu{y\in\Omega_\phi}(S_\phi(\,\cdot\,,y)\otimes b(y))$ is upper semi-continuous, i.e., for all $x\in X$,
\begin{equation}\label{ineq: upper semi-continuity of representation}
    \lim_{\epsilon\to 0^+}\oplusu{z\in B(x,\epsilon)}\oplusu{y\in\Omega_\phi}(S_\phi(z,y)\otimes b(y))\leq\oplusu{y\in\Omega_\phi}(S_\phi(x,y)\otimes b(y)).
\end{equation}

Fix $x\in X$. Note that the limit in the left-hand side of (\ref{ineq: upper semi-continuity of representation}) is a nonincreasing limit. It suffices to show that for all $\delta>0$, there exists $\epsilon>0$ such that 
\begin{equation*}
    \oplusu{z\in B(x,\epsilon)}\oplusu{y\in\Omega_\phi}(S_\phi(z,y)\otimes b(y))\leq\log\Bigl(\exp\Bigl(\oplusu{y\in\Omega_\phi}(S_\phi(x,y)\otimes b(y))\Bigr)+\delta\Bigr)
\end{equation*}
We argue by contradiction. Suppose that for some $\delta>0$ and all $n\in\N$, there exists $z_n\in B(x,1/n)$ such that 
\begin{equation*}
    \oplusu{y\in\Omega_\phi}(S_\phi(z_n,y)\otimes b(y))>\log\Bigl(\exp\Bigl(\oplusu{y\in\Omega_\phi}(S_\phi(x,y)\otimes b(y))\Bigr)+\delta\Bigr).
\end{equation*} 
Note that $S_\phi$ is upper semi-continuous (see Lemma~\ref{l: mane potential is upper semi-continuous}) and $b\in\density(X)$ is also upper semi-continuous. Recall that $\Omega_\phi\subseteq X$ is closed and $X$ is compact. It follows that for each $n\in\N$, there exists $y_n\in\Omega_\phi$ such that $S_\phi(z_n,y_n)\otimes b(y_n)=\oplusu{y\in\Omega_\phi}(S_\phi(z_n,y)\otimes b(y))$. Furthermore, there exists a subsequence $\{y_{n_k}\}_{k\in\N}$ (with $\lim_{k\to+\infty}n_k=+\infty$) converging to $y'\in\Omega_\phi$ as $k\to+\infty$. It then follows from the upper semi-continuity of $S_\phi$ and $b$ that
\begin{equation}\label{ineq: contra for upper semi of representation}
\begin{aligned}
    &\log\Bigl(\exp\Bigl(\oplusu{y\in\Omega_\phi}(S_\phi(x,y)\otimes b(y))\Bigr)+\delta\Bigr)
   \leq\limsup_{k\to+\infty}\oplusu{y\in\Omega_\phi}(S_\phi(z_{n_k},y)\otimes b(y))\\
   & \qquad =\limsup_{k\to+\infty}S_\phi(z_{n_k},y_{n_k})\otimes b(y_{n_k})
    \leq S_\phi(x,y')\otimes b(y')
    \leq\oplusu{y\in\Omega_\phi}(S_\phi(x,y)\otimes b(y)).
\end{aligned}
\end{equation}
 Recall that $\oplusu{y\in\Omega_\phi}(S_\phi(\,\cdot\,,y)\otimes b(y))\: X\to\Rminus$. It follows that (\ref{ineq: contra for upper semi of representation}) forms a contradiction and~(\ref{ineq: upper semi-continuity of representation}) is verified. We conclude that $\oplusu{y\in\Omega_\phi}(S_\phi(\,\cdot\,,y)\otimes b(y))\in\density(X)$ and $b(x)=\oplusu{y\in\Omega_\phi}(S_\phi(x,y)\otimes b(y))$ for all $x\in X$.
\end{proof}

\begin{proof}[\bf Proof of Proposition~\ref{p: condition_for_uniqueness}]
    Assume that $\phi$ has a unique maximizing measure. We claim that $S_\phi(x,y)\otimes S_\phi(y,x)=0$ for all $x,\, y\in \Omega_\phi$. 

    If the claim holds, then it follows from  Proposition~\ref{p: Mane_potential_properties}~(iii)(iv) that $S_\phi(y,x)\otimes S_\phi(x,\,\cdot\,)=S_\phi(y,\,\cdot\,)$ and $S_\phi(\,\cdot\,,x)\otimes S_\phi(x,y)=S_\phi(\,\cdot\,,y)$ for all $x,\,y\in\Omega_\phi$ in the same way as in the proof of \cite[Lemma~3.21]{LS26}. Moreover, it follows from Proposition~\ref{p: representation}~(i)(ii) that 
    $\tdim\es{\cL_\phi}{C(X,\Rminus)}{\energy}
   =1
   = \tdim\esbig{\cL_\phi^{\ostar}}{\density(X)}{\energy}$ in the same way as in the proof of \cite[Proposition~3.23]{LS26}.

    Now we prove the claim. It follows from Proposition~\ref{p: Mane_potential_properties}~(iii)(iv) that for all $x,\,y\in X$,
    \begin{equation}\label{ineq: mane potential x-y-x smaller than zero}
    S_\phi(x,y)\otimes S_\phi(y,x)\leq S_\phi(x,x)\leq 0.
    \end{equation}
    It follows from (\ref{eq: Mane potential (_,y) almost eigendensity}) in Lemma~\ref{l: mane potential gives eigenmeasures} that $S_\phi(T^n(x),x)=S_\phi\bigl(T^{n-1}(x),x\bigr)-\toverline{\phi}\bigl(T^{n-1}(x)\bigr)$ for all $n\in\N$ and $x\in \Omega_\phi$. Since $S_\phi(x,x)=0$ for $x\in\Omega_\phi$ (cf.~Proposition~\ref{p: Mane_potential_properties}~(iv)), $S_\phi(T^n(x),x)=S_\phi(x,x)-S_n\toverline{\phi}(x)=-S_n\toverline{\phi}(x)$ for all $n\in\N$ and $x\in\Omega_\phi$. By Definition~\ref{d: Mane potential}, we see that $S_\phi(x,T^n(x))\geq S_n\toverline{\phi}(x)$ for all $n\in\N$ and $x\in X$. Thus, $S_\phi(x,T^n(x))\otimes S_\phi(T^n(x),x)\geq S_n\toverline{\phi}(x)\otimes\bigl(-S_n\toverline{\phi}(x)\bigr)=0$ for all $n\in\N$ and $x\in\Omega_\phi$. It follows from this inequality and~(\ref{ineq: mane potential x-y-x smaller than zero}) that $S_\phi(x,T^n(x))\otimes S_\phi(T^n(x),x)=0$ for all $n\in\N$ and $x\in\Omega_\phi$. 

Denote $L_x\=\overline{\{T^n(x):n\in\N\}}$ for all $x\in\Omega_\phi$. It then follows from the upper semi-continuity of $S_\phi$ (see Lemma~\ref{l: mane potential is upper semi-continuous}) and~(\ref{ineq: mane potential x-y-x smaller than zero}) that $S_\phi(x,z)\otimes S_\phi(z,x)=0$ for all $z\in L_x$. Note that $S_\phi(x,y)\otimes S_\phi(y,x)=0$ defines an equivalence relation between points in $\Omega_\phi$ by Proposition~\ref{p: Mane_potential_properties}~(iii)(iv). 

We conclude that if $S_\phi(x,y)\otimes S_\phi(y,x)\neq 0$ for some $x,\,y\in\Omega_\phi$, then $L_x\cap L_y=\emptyset$. Recall from Lemma~\ref{l: maximizing measures supported on Aubry set} that an invariant Borel probability measure is a $\phi$-maximizing measure if and only if it is supported on $\Omega_\phi$. Since $L_x$ and $L_y$ are both compact invariant subsets of $\Omega_\phi$, we see that there are at least two maximizing measures for $\phi$ supported respectively on $L_x$ and $L_y$, which contradicts the assumption that $\phi$ has a unique maximizing measure. 
\end{proof}

\end{document}